\documentclass[a4paper,11pt]{amsart}
\usepackage[utf8]{inputenc}
\usepackage[T1]{fontenc}

\usepackage{geometry}
\usepackage{amssymb}
\usepackage{amsfonts}
\usepackage{amsthm}
\usepackage{mathrsfs}
\usepackage{array}
\usepackage{bbm}
\usepackage{stmaryrd}
\usepackage{hyperref}
\usepackage{enumerate}
\usepackage{mathabx}
\usepackage{enumitem}
\usepackage[demo]{graphicx}
\usepackage{caption}
\usepackage{subcaption}

\usepackage{csquotes}

\usepackage{mathtools}
\mathtoolsset{showonlyrefs=true}
\usepackage[svgnames]{xcolor}
\usepackage{hyperref}

\usepackage{tikz}
\usetikzlibrary{positioning}

\newcommand{\R}{\mathbb{R}}

\newcommand{\e}{\epsilon}

\newcommand{\jap}[1]{\langle #1 \rangle}

\newcommand{\reg}{\textnormal{reg}}

\newcommand{\m}[1]{
	\ifdefequal{#1}{1}
	{\mathbbm{#1}}
	{\mathbb{#1}}
}

\newcommand{\q}[1]{\mathscr{#1}}

\newcommand{\cal}[1]{\mathcal{#1}}

\newcommand{\ds}{\displaystyle}

\DeclareMathOperator{\Supp}{\mathrm{supp}}

\newcommand{\fia}{\mathbbm{1}_{|\Phi-\alpha|<M}}
\newcommand{\psia}{\mathbbm{1}_{|\Psi-\alpha|<M}}

\theoremstyle{plain}
\newtheorem{thm}{Theorem}[section]
\newtheorem*{thm*}{Theorem}
\newtheorem{prop}[thm]{Proposition}
\newtheorem{cor}[thm]{Corollary}
\newtheorem{lem}[thm]{Lemma}

\theoremstyle{definition}

\theoremstyle{remark}
\newtheorem{nb}{Remark}[section]

\numberwithin{equation}{section}
\newtheoremstyle{mytheoremstyle} 
{\topsep}                    
{\topsep}                    
{}                   
{}                           
{\scshape}                   
{.}                          
{.5em}                       
{}  

\theoremstyle{mytheoremstyle} 
\theoremstyle{mytheoremstyle} 

\makeatletter
\@namedef{subjclassname@2020}{%
	\textup{2020} Mathematics Subject Classification}
\makeatother

\date{}
\author{Sim\~ao Correia \and Raphaël Côte}

\title{Sharp blow-up stability for self-similar solutions of the modified Korteweg-de Vries equation}

\thanks{S. Correia was partially supported by Funda\c{c}\~ao para a Ci\^encia e Tecnologia, through CAMGSD, IST-ID
	(projects UIDB/04459/2020 and UIDP/04459/2020) and through the project NoDES (PTDC/MAT-PUR/1788/2020). R. Côte acknowledges support from the University of Strasbourg Institute for Advanced Study (USIAS) for a Fellowship within the French national programme ``Investment for the future'' (IdEx-Unistra).}


\allowdisplaybreaks

\begin{document}
	
\parindent=0pt
	
\begin{abstract}
	We consider the modified Korteweg-de Vries equation. Given a self-similar solution, and a subcritical perturbation \emph{of any size}, we prove that there exists a unique solution to the equation which behaves at blow-up time as the self-similar solution plus the perturbation. 
		
	To this end, we develop the first robust analysis in spaces  of functions with bounded Fourier transforms. To begin, we prove the local well-posedness in subcritical spaces through an appropriate restriction norm method. As this method is not sufficient to capture the critical self-similar dynamics, we develop an infinite normal form reduction (INFR) to derive time-dependent \emph{a priori} $L^\infty$ bounds in frequency variables. Both approaches rely on frequency-restricted estimates, which are specific positive multiplier estimates capable of capturing the oscillatory nature of the equation. As a consequence of our analysis, we also prove local well-posedness for small subcritical perturbations of self-similar solutions at positive time.
\end{abstract}

\maketitle
	
\section{Introduction}
	
\subsection{Setting and motivation}

In this work, we consider the modified Korteweg-de Vries equation on the real line
\begin{equation}\tag{mKdV}\label{mKdV}
	\partial_t u + \partial_{x}^3 u = \pm \partial_x (u^3), \quad (t,x) \in \m R^2, \quad u(t,x) \in \m R.
\end{equation}
The sign will be irrelevant in our analysis and thus we focus on the focusing case (+). As this equation is invariant through the scaling $u_\lambda(t,x)=\lambda u(\lambda^3t, \lambda x)$, one may look for self-similar solutions, that is, solutions invariant under scaling. A direct computation shows that self-similar solutions are necessarily
	of the form
\begin{gather} \label{eq:selfsim}
	\mathcal{S}(t,x)=\frac{1}{t^{1/3}}\mathcal{S}\left(\frac{x}{t^{1/3}}\right),\quad \text{where } \mathcal{S}'' - \frac{y}{3}\mathcal{S} = \mathcal{S}^3 + \alpha,\ \alpha\in\R.
\end{gather}
The existence of self-similar profiles $\mathcal{S}$ has been proven in \cite{CCV19, DZ95, DK21, HM80,  PV07} using ODE, stationary phase or complete integrability techniques. The corresponding solutions are critical in terms of time, space and frequency decay. In  physical space, the profiles have decay like $|x|^{-1/4}$ as $x\to -\infty$, while their derivative grows as $|x|^{1/4}$ (implying a strong oscillatory behavior, see for example \cite{CCV20} for precise asymptotics). In frequency space, the solutions are merely bounded, with logarithmic oscillations at infinity and a jump discontinuity at $\xi=0$ (induced by the parameter $\alpha$): we refer to Proposition \ref{prop:selfsimilar} below for more details.
	
\bigskip{}

Due to their scaling-invariant nature, self-similar solutions are not included in any existing local well-posedness theory. Indeed, in the $H^s$-scale, well-posedness is known to be analytic (cf. \cite{KPV93}) for $s\ge 1/4$, continuous for $s>-1/2$ (cf. \cite{HKV20}), and \textit{fails} at the critical regularity $s=-1/2$ due to an instantaneous norm inflation mechanism (this mechanism can also be seen in the evolution of self-similar solutions). In the scale of Fourier-Lebesgue spaces $\widehat{H}^s_r= \{ u \in \cal S'(\m R): \jap{\xi}^s \hat u \in L^{p '} \} $, the analytic well-posedness can be shown for almost-critical spaces $r>1$ and $s>(1-r)/2$ (cf. \cite{GV09}). They barely missed $\widehat{H}^0_1$ (which corresponds to $\hat u \in L^\infty_\xi$), which is a critical space where self-similar solutions lie.
	
\bigskip
	
Self-similar solutions present a natural blow-up behavior at $t=0$, which is connected to the formation of logarithmic spirals and corners in the evolution of vortex patches in the plane (\cite{GP92}). On the other hand, they also determine the long time asymptotics for small solutions of \eqref{mKdV} (\cite{DZ93,GPR16,HaGr16,HN01, HN99}). In the refered articles, the analysis in performed for subcritical solutions. In \cite{CCV20}, together with Luis Vega, we introduced a critical space where long-time asymptotics could be analyzed. In particular, we proved that any small critical object converges to the self-similar solution with the same zero Fourier mode. As such, the most relevant open problems related with self-similar solutions concern the dynamics near $t=0$. In a parallel problem for the cubic nonlinear Schrödinger equation on $\R$, the use of the pseudo-conformal transformation translates the problem from $t=0$ to one at $t=\infty$ (while also turning self-similar solutions into constants), see \cite{ BV12, BV13}. For the modified KdV, the absence of such a transformation forces us to analyze the problem directly.

\bigskip
	
In our previous work \cite{CC22}, we managed to prove the stability of the blow-up phenomena under \textit{smooth and small} perturbations at the blow-up time. More precisely, given a self-similar solution $\mathcal{S}$ (small in a critical space) and a perturbation $z$ small in a sufficiently strong topology\footnote{We denote $\hat\cdot$ or $\cal F$ the Fourier transform, possibly specifying variables when it is space-time.}, namely
\begin{equation} \label{est:z_cond_old}
	z\in L^1_x,\qquad \jap{\xi}^2\hat{z}\in L^1_\xi,\qquad \jap{\xi}\partial_\xi\hat{z} \in L^1\cap L^2,
\end{equation}
then there exists a unique solution $u$ to \eqref{mKdV}, defined for $t>0$, such that
\begin{equation}\label{eq:defistab}
	u(t)-\mathcal{S}(t)\to z \quad \text{ as } t\to 0^+.
\end{equation}
There were two main ingredients in the proof. Define the \emph{profile} 
\[ \tilde{u}(t)= \cal F_x (e^{-t\partial_x^3}u(t)), \]
the phase
\[ \Phi = \Phi(\xi,\xi_1,\xi_2,\xi_3) = \xi^3 - \xi_1^3- \xi_1^3-\xi_1^3, \]
and the hyperplane of convolution $H_\xi = \{ (x_1,\xi_2,\xi_3) \in \m R^3 : \xi_1 + \xi_2+ \xi_3 = \xi \}$.
Then the corresponding equation becomes
\begin{gather} \label{mkdv:profile}
	\partial_t \tilde{u} = \frac{i\xi}{4\pi^2}\iint\nolimits_{H_\xi} e^{it \Phi}\tilde{u}(t,\xi_1)\tilde{u}(t,\xi_2)\tilde{u}(t,\xi_3)d\xi_1d\xi_2,
\end{gather}
so that the nonlinear term takes the form of an  oscillatory integral. By stationary phase arguments, this equation reduces to an almost-pointwise ODE in time:
\begin{equation}\label{eq:ode}
	\partial_t \tilde{u} = \frac{\pi \xi^3}{\jap{\xi^3t}}\left(i|\tilde{u}(t,\xi)|^2\tilde{u}(t,\xi) - \frac{1}{\sqrt{3}}e^{-8it\xi^3/9}\tilde{u}^3\left(t,\frac{\xi}{3}\right) \right)+ R[u](t,\xi).
\end{equation}
The remainder $R[u]$ is controlled (and integrable in time) as long as we control both
\[
	\tilde{u}\in L^\infty_\xi\qquad \mbox{and}\qquad \partial_\xi\tilde{u}\in L^2_\xi.
\]
The $L^\infty$ bound on $\tilde u$ can be bootstrapped using the profile equation \eqref{mkdv:profile} itself. Unfortunately, this is not possible for $\partial_\xi\tilde{u}$, as differentiating in frequency the profile equation would introduce strong divergent oscillations. It is at this moment that the second ingredient comes into play, the \textit{vector-field} or \textit{scaling operator} $(x-3t\partial_x^2)u$. Indeed, the required $L^2$ bound for $\partial_\xi \tilde{u}$ corresponds (up to controlled terms) to an $L^2$ bound for the vector-field, and this can be obtained through a direct energy estimate. These two mechanisms are enough to bound solutions for times away from $0$ (as done in \cite{CCV20}). If one introduces sufficiently smooth subcritical  perturbations, the bounds on the remainder can actually be shown to be uniform up to $t=0$, which lead to the blow-up stability result.
	
\bigskip
	
The main issue with the above result lies with the conditions \eqref{est:z_cond_old} imposed on the perturbation $z$ (and also the smallness in those spaces). These restrictions can be traced to several deficiencies in our approach: first, the analysis is performed simultaneously in physical and frequency space. Second, the solution is constructed in a critical space, leading to a smallness assumption. Third, the control in Fourier space does not exploit oscillations in time, as the profile equation gives a pointwise control on the \textit{time derivative} $\partial_t\tilde{u}$. As optimal results in dispersive equations usually require a refined analysis of the oscillations in frequency, it becomes clear that our previous result was far from being sharp.
	
\bigskip
	
The main goal of the present work is to provide a \textit{sharp} stability result of self-similar solutions at blow-up time. Roughly speaking, given a perturbation $z$ with \emph{almost critical} decay in frequency space (measured in $W^{1,\infty}_\xi$) and of \emph{arbitrary size}, we show here that there exists a unique solution $u$ to \eqref{mKdV}, defined around $t=0$, for which \eqref{eq:defistab} holds.
	
\subsection{Main results and description of the proofs}
	
As mentioned before, a sharp stability result requires one to abandon the  expansion \eqref{eq:ode} (as it does not exploit the time oscillations) and perform the complete analysis in frequency space. We are forced to go back to \eqref{mkdv:profile} and work directly with the full oscillatory integral. 
	
A first attempt would be to try a fixed-point argument using the Fourier restriction norm method as implemented in \cite{GV09}, extending their result to weighted $\widehat{L^\infty_\xi}$ spaces. In this direction, we have
	
\begin{thm} \label{thm:lwpS=0}
    Equation \eqref{mKdV} is locally well-posed in\footnote{Here, $\widehat{L^\infty}(\jap{\xi}^\mu d\xi)=\{u\in \mathcal{S}'(\R^N): \jap{\xi}^\mu \hat{u}(\xi)\in L^\infty(\R)\}$.} $\widehat{L^\infty}(\jap{\xi}^\mu d\xi)$, for any $\mu >0$.
\end{thm}
	
This result follows in a rather neat way (done in Section \ref{sec:proof_th1}) from the \emph{frequency-restricted estimate}\footnote{A frequency-restricted estimate is a multiplier estimate over sublevel sets of the resonance function.} (see Proposition \ref{prop:fre_phi})
\begin{equation}\label{eq:fre_intro}
	\forall M \ge 1, \quad	\sup_{\xi,\alpha\in \R} \iint\nolimits_{H_\xi} \frac{\jap{\xi}^\mu }{\jap{\xi_1}^\mu \jap{\xi_2}^\mu \jap{\xi_3}^\mu}\fia d\xi_1d\xi_2 \le M^\beta,
\end{equation}
for some $\beta \in (0,1)$. 
The relation between the Fourier restriction norm method and frequency-restricted estimates was observed in \cite{COS23} for $L^2_\xi$-based spaces. Here, we extend the argument to $L^\infty_\xi$ spaces. 
	
	\medskip
	
The fact that a local well-posedness theory exists in almost-critical weighted $L^\infty$ spaces is a good starting point to analyze the stability of self-similar solutions in $L^\infty_\xi$. Let us give some insight why the critical regularity $\mu=0$ seems currently out of reach. The description of the self-similar solution in space-time frequency variables $(\tau,\xi)$ presents unavoidable logarithmic divergences: indeed, a direct computation using the scale-invariant structure of $\mathcal{S}$ shows that
\[
	\mathcal{F}_{t,x}(e^{t\partial_x^3}\mathcal{S}) = \frac{1}{\xi^3}g\left(\frac{\tau}{\xi^3}\right), \quad\mbox{for some }g\in \mathcal{S}'(\R).
\]
Therefore its $L^\infty_\xi$ norm behaves as $1/\tau$, which is not controllable in $L^2(\jap{\tau}^bd\tau)$, for any $b \ge 1/2$.
	
A similar problem appears in physical variables $(t,x)$: consider the toy problem of the linearized equation 
\[ \partial_t v + \partial^3_{x} v + 3 \partial_x (\mathcal{S}^2 v) =0. \]
 The most natural move is to use the estimates of Kenig, Ponce and Vega \cite{KPV00}, which allows to recover the loss of a derivative:
	\[ \| v \|_{L^\infty_t L^2_x} \lesssim \| v(0) \|_{L^2_x} + \| K^2 v \|_{L^1_x L^2_t}. \]
	Now one can essentially only use Hölder estimate:
	\[ \| \mathcal{S}^2 v \|_{L^1_x L^2_t} \le \| \mathcal{S}^2 \|_{L^4_x L^\infty_t}^2 \| v \|_{L^2_{x,t}}. \]
	However, as it was observed in \cite{CCV19}, $\mathcal{S}(t,x)\sim t^{-1/3}\jap{xt^{-1/3}}^{-1/4}\notin L^4_xL^\infty_t$, and the argument cannot be closed.
	
	\medskip
	In $(t,\xi)$ variables however, self-similar solutions are actually agreeable\footnote{See \eqref{def:W_weighted} for the definition of the weighted spaces $W^{1,\infty}_{\mu,\mu'}$.}:
	\begin{prop}[\cite{CCV19}]
		\label{prop:selfsimilar}
		Given $A\in\mathbb{C}$ small, there exists a self-similar solution $\cal S$ whose profile $S(\xi):=e^{-i\xi^3}\widehat{\mathcal{S}(1)}(\xi)$ belongs to $W^{1,\infty}_{0,1}(\R\setminus \{0\})$ (but not better). It can be decomposed
		\begin{gather*}
			S(\xi) = S_0(\xi) + S_{\reg}(\xi),
		\end{gather*} 
		where $S_{\reg}\in W^{1,\infty}_{(4/7)^-,(11/7)^-}(\R\setminus \{0\})$ (it may have a jump at $0$) and
		\begin{align}\label{asymp}
			S_0(\xi) = \chi(\xi) e^{i a \ln |\xi|} \left( A + B e^{2ia \ln |\xi|} \frac{e^{-i \frac{8}{9} \xi^3}}{\xi^3} \right),
		\end{align}
		where $\chi\in \q C^\infty(\R)$ satisfy $\chi\equiv 0$ for $|\xi|<1$ and $\chi\equiv 1$ for $|\xi|>2$, and for some  $B\in \mathbb{C}$ with $|A|+|B|\sim \|S\|_{W^{1,\infty}_{0,1}}(\R\setminus \{0\})$.
	\end{prop}
	Of course, $e^{-it\xi^3}\widehat{\mathcal{S}(t)}(\xi) = S(t^{1/3} \xi)$. When it carries no confusion, we will often denote $S(t,\xi) =S(t^{1/3}\xi)$ and do the same for $S_0$ and $S_{\reg}$.
	
	\medskip 
	We are then lead to the analysis of the full oscillatory integral in $(t,\xi)$ variables at critical regularity. As it is becoming clear in the recent years \cite{CS20,GKO13,Kishimoto22,KwonOh, KOY20, OW21}, \textit{the analogue of the Fourier restriction method in $(t,\xi)$ variables is the infinite normal form reduction (INFR)}. Let us briefly explain the idea behind this procedure. Write \eqref{mkdv:profile} in integral form,
	\begin{gather} \label{mkdv:profile2}
		\tilde{u}(t,\xi) = \tilde{u}(0,\xi) + \frac{i\xi}{4\pi^2}\int_0^t \iint\nolimits_{H_\xi} e^{is\Phi}\tilde{u}(s,\xi_1)\tilde{u}(s,\xi_2)\tilde{u}(s,\xi_3)d\xi_1d\xi_2ds,
	\end{gather}
	where $\Phi=\xi^3-\xi_1^2-\xi_2^3-\xi_3^3$. In order to exploit the oscillations in time, we want to integrate by parts in time, using the relation
	\[
	e^{it\Phi}=\frac{\partial_t(e^{it\Phi})}{i\Phi}.
	\]
	It does not induce singularities, as long as $\Phi$ is not small. This motivates the introduction of a parameter $N>0$ to split the frequency domain in the near-resonant region $|\Phi|<N$ and the nonresonant region $|\Phi|>N$ and to integrate by parts in the latter (all double integrals are over $H_\xi$):
	\begin{align*}
		\tilde{u}(t,\xi) &= \tilde{u}(0,\xi) + \frac{i\xi}{4\pi^2}\int_0^t \iint e^{is\Phi}\mathbbm{1}_{|\Phi|<N}\tilde{u}(s,\xi_1)\tilde{u}(s,\xi_2)\tilde{u}(s,\xi_3)d\xi_1d\xi_2ds \\&\qquad+ \frac{i\xi}{4\pi^2}\int_0^t \iint e^{is\Phi}\mathbbm{1}_{|\Phi|>N}\tilde{u}(s,\xi_1)\tilde{u}(s,\xi_2)\tilde{u}(s,\xi_3)d\xi_1d\xi_2ds\\&= \tilde{u}(0,\xi) + \frac{i\xi}{4\pi^2}\int_0^t \iint e^{is\Phi}\mathbbm{1}_{|\Phi|<N}\tilde{u}(s,\xi_1)\tilde{u}(s,\xi_2)\tilde{u}(s,\xi_3)d\xi_1d\xi_2ds \\&\qquad+ \frac{i\xi}{4\pi^2}\left[ \iint e^{is\Phi}\mathbbm{1}_{|\Phi|>N}\tilde{u}(s,\xi_1)\tilde{u}(s,\xi_2)\tilde{u}(s,\xi_3)d\xi_1d\xi_2\right]_{s=0}^{s=t} \\&\qquad-\frac{i\xi}{4\pi^2}\int_0^t \iint \frac{e^{is\Phi}}{i\Phi}\mathbbm{1}_{|\Phi|>N}\partial_t\left(\tilde{u}(s,\xi_1)\tilde{u}(s,\xi_2)\tilde{u}(s,\xi_3)\right)d\xi_1d\xi_2ds.
	\end{align*}
	In the last integral, after distributing the time derivative, we are now free to use \eqref{mkdv:profile} and rewrite the last integral as an oscillatory integral which is now quintic in $u$. We arrive at an expanded version of \eqref{mkdv:profile2}, with a few well-behaved cubic terms plus a quintic term: this concludes the first step of the INFR. 
	
	We can play this game again, splitting the quintic term into ``good'' quintic terms plus a septic integral, then replacing the septic integral and so forth. This process can then be iterated indefinitely: the end result is an infinite expansion of \eqref{mkdv:profile2} in well-behaved terms of arbitrary order:
	\[
	\tilde{u}(t,\xi)=\tilde{u}(0,\xi) + \sum_{J\ge 1} \left(\mbox{Boundary terms at step }J + \mbox{ Near-resonant terms at step }J\right).
	\]
	The main difficulty that now arises is the derivation of multilinear bounds for every single term in the expansion, together with some decay in $J$ to ensure summability. This has been for some time an unclear topic, especially in what concerns $L^2_\xi$-bounds: apart from some concrete cases \cite{GKO13, OW21}, the general framework announced in \cite{KOY20} seems not to be completely correct (see Remark \ref{nb:pb_INFR_L2}) and the validity of the approach described therein remains to be proved.

	In this work, we rigorously formalize the derivation of $L^\infty_\xi$ \textit{a priori} bounds through the INFR. As it will be explained in Section \ref{sec:INFR}, these bounds may be reduced to {frequency-restricted estimates} akin to \eqref{eq:fre_intro} (see in particular the paragraph \ref{sec:INFR_bounds}).
	%
	%

	\begin{nb}
		One of the main takeaways of the present work is the effectiveness of 
		frequency-restricted estimates in deducing bounds for nonlinear dispersive equations. Not only can these estimates be used to derive multilinear estimates in Bourgain spaces and \textit{a priori} bounds through the INFR, but they can also  be used to prove Strichartz estimates. In conclusion,
		\[
		\text{Frequency-restricted estimates (in } \xi) \Rightarrow \begin{cases}
			\text{INFR }a\  priori \text{ bounds in }(t,\xi)\\ 
			\text{multilinear Bourgain estimates in }(\tau,\xi)\\
			\text{Strichartz estimates in }(t,x)
		\end{cases}
		\]
		As we already pointed out above, the last two approaches are not suitable to deal with a self-similar background.	
	\end{nb}
	
	\bigskip
	
	In order to implement the INFR around self-similar solutions, we need to decompose (as done in \cite{CC22}) the solution into
	\begin{equation} \label{decomp_sol}
		\text{self-similar $S$ + linear evolution of perturbation $z$ + interaction remainder $w$.}
	\end{equation}
	As the first two terms are already defined, the problem reduces to the derivation of appropriate bounds for the remainder $w$. As in \cite{CC22}, we expect it to grow in time as a small power of $t$. This is actually a crucial point: it allows to avoid the failure of the time integrability induced by the linearized operator around the self-similar solution, which gives a logarithmic singular behavior. In other words, we are forced to work in $(t,\xi)$ in order to introduce time weights for the interaction remainder.
	
	Assuming subcriticallity for the perturbation, we show that the remainder is subcritical as well, even in the presence of the self-similar background. 
	Indeed, when expanding the nonlinearity using the decomposition \eqref{decomp_sol}, the INFR can be used to prove weighted $L^\infty_\xi$ \textit{a priori} bounds for the remainder,  and for most of the source terms, except one (called $F_2$ in following). This one term has to be dealt with integration by parts in space, and this requires additional smoothness on $z$.
	
	The construction of the remainder term follows from \textit{a priori} bounds on an approximate problem: we chose to cut off the nonlinearity at large frequency and then pass to the limit. As such, a mere weighted $L^\infty_\xi$ bound is not sufficient to show that the limit is indeed a solution to \eqref{mKdV}. To bypass this problem, we will also derive weighted $L^\infty_\xi$ bounds for the derivatives. Such a control cannot be obtained directly, but an effective way is to use the scaling operator which, in Fourier variables, reads
	\begin{equation} \label{def:dilation_op}
		\Lambda = \partial_\xi - \frac{3t}{\xi}\partial_t.
	\end{equation}
	A key algebraic point is that the equation for $\Lambda w$ \emph{has the same algebraic structure as that for $w$}: this is due to the specific critical structure of \eqref{mKdV}, and the scaling invariance of the self-similar solution (this was already a decisive ingredient in the stability proof of \cite{CC22}). In particular, assuming subcriticality for $\partial_\xi z$, the INFR can be used once more to produce an \textit{a priori} $L^\infty_\xi$ bound on $\Lambda w$. This control of the derivatives of $w$ allows for the application of Ascoli-Arzelà theorem and the construction of the interaction remainder is achieved.
	
	%

	\bigskip
	
	With these elements of context in mind, we can now present the main stability result. Given $\mu',\mu>0$, we define the  Banach spaces
	\begin{gather} \label{def:W_weighted}
	L^\infty_\mu(\Omega):=\left\{ v\in L^\infty(\Omega) : \|v\|_{L^\infty_\mu(\Omega)}<\infty \right\} \quad \text{and} \quad  W^{1,\infty}_{\mu,\mu'}(\Omega):=\left\{ v\in W^{1,\infty}(\Omega) : \|v\|_{W^{1.\infty}_{\mu,\mu'}(\Omega)}<\infty \right\}, \\
    \text{where} \quad 
	\|v\|_{L^\infty_\mu(\Omega)}:=\|\jap{\xi}^\mu v(\xi)\|_{L^\infty(\Omega)},\quad \text{and} \quad 
	\| v \|_{W^{1,\infty}_{\mu,\mu'}(\Omega)} := \|v\|_{L^\infty_\mu(\Omega)} + \|\partial_\xi v\|_{L^\infty_{\mu'}(\Omega)}.
	\end{gather}
	If $\Omega=\R$, we omit the domain. Moreover, when $\mu=\mu'$, we write $W^{1,\infty}_{\mu,\mu}(\Omega)$ simply as $W^{1,\infty}_\mu(\Omega)$. Throughout this work, we fix 
	\begin{equation}\label{eq:defieta}
		0<\mu<\nu<1/2 \quad \text{and} \quad 0<\gamma<  \frac{1}{3}  \min \left( \mu, \nu-\mu \right).
	\end{equation}

\begin{thm}[Blow-up stability of self-similar solutions]\label{thm:exist}
	Fix $\epsilon>0$ small and a self-similar solution $S$ with $\|S\|_{W^{1,\infty}_{0,1}(\R\setminus\{0\})}<\epsilon$. 
	Given $z\in W^{1,\infty}_{\nu}$, there exist $T=T(\epsilon, \mu,\nu, \|z\|_{W^{1.\infty}_{\nu}})>0$ and a unique 
	\[ w\in L^\infty((0,T), W^{1,\infty}_{\mu}) \]
	such that 
	\begin{equation} \label{est:td_tw}
		\forall t \in [0,T], \quad 
		\|w(t)\|_{W^{1,\infty}_{\mu}} + \left\| \frac{t\partial_t w(t)}{\xi}\right\|_{L^\infty_\mu}\lesssim \epsilon^3 t^\gamma,
	\end{equation}
	and $u=e^{-t\partial_x^3}(S+z+w)^\vee$ is a distributional solution of \eqref{mKdV} on $(0,T)\times \R$. Moreover, if $z_1,z_2\in W^{1,\infty}_{\nu}$ and $w_1,w_2$ are the corresponding solutions, then
	\begin{equation} \label{est:cont_dep}
		\sup_{s\in[0,t]} \|w_1(s)-w_2(s)\|_{L^\infty_\mu}\lesssim \epsilon^2 t^\gamma \|z_1-z_2\|_{L^\infty_\nu},\quad 0<t<T.
	\end{equation}
\end{thm}
	
\begin{nb}
	Notice that the remainder has an arbitrarely small loss in decay when compared to the perturbation. If one aims at deriving bounds without any loss, a logarithmic divergence appears at $t=0$. In view of to the critical nature of the self-similar solution, this is to be expected.
\end{nb}
	
\begin{nb}
	A mere control of the perturbation in $L^\infty_\xi$ is not enough to prove the stability result. Indeed, in order to  control the linearized operator $\partial_x (\mathcal{S}^2u)$, we need to exploit oscillations in frequency (this was already a crucial observation in \cite{CCV20}).
\end{nb}

\begin{nb}
    We emphasize that our framework is only based on weighted $L^\infty$ estimates in (space) Fourier variable. We do not rely on any energy estimate. We believe that this is the correct framework to consider self-similar related problems. 
\end{nb}

	As a byproduct of the analysis performed, we also obtain a local well posedness result, at strictly positive time,  for data which are a perturbation of a self-similar solution.

\begin{thm}[Local well-posedness around a self-similar solution at positive times]\label{thm:lwp}
	Fix $t_0>0$. For $\epsilon>0$ small, given a self-similar solution with profile $S \in W^{1,\infty}_{0,1}(\R\setminus\{0\})$, and $w_0\in W^{1,\infty}_{\mu}$ such that 
	\[ \|S\|_{W^{1,\infty}_{0,1}(\R\setminus\{0\})} + \|w_0\|_{W^{1,\infty}_{\mu}}<\epsilon, \]
	there exist $T_\pm=T_\pm(\epsilon,\mu)>0$,  $T_-<t_0<T_+$, and a unique $w\in  L^\infty((T_-,T_+), W^{1,\infty}_{\mu})$ such that 
	\[
		\|w(t)\|_{W^{1.\infty}_{\mu}} + \left\| \frac{t\partial_t w(t)}{\xi}\right\|_{L^\infty_\mu}\lesssim \epsilon,\quad T_-<t<T_+,
	\]
	$w(t_0)=w_0$, and $u=e^{-t\partial_x^3}(S+w)^\vee$ is a distributional solution of \eqref{mKdV} on $(T_-,T_+)\times \R$. Furthermore, 
	\begin{equation}\label{eq:limiteT}
		\lim_{\epsilon\to 0}T_-(\epsilon)=0,\qquad \lim_{\epsilon\to 0}T_+(\epsilon)=+\infty.
	\end{equation}
	Finally, if $w_{01},w_{02}\in W^{1,\infty}_{\mu}$ and $w_1,w_2$ are the corresponding solutions, then
	\[
		\sup_{s\in[T_-,T_+]} \|w_1(s)-w_2(s)\|_{L^\infty_\mu}\lesssim \epsilon^2  \|w_{01}-w_{02}\|_{L^\infty_\mu}.
	\]
\end{thm}

When comparing Theorems \ref{thm:exist} and \ref{thm:lwp}, the main difference is the smallness assumption on the subcritical perturbation in the latter. This is somehow unexpected, as the blow-up problem should be considerably more difficult to handle. The difference is a consequence of the use of the scaling operator (which behaves better at $t=0$) in the derivation of the required \textit{a priori} bounds (see Remark \ref{rmk:small}).
	
\bigskip
	
The structure of this article is as follows. We start in Section \ref{sec:2} by solving an approximate problem and giving an outline of the structure of the proof of Theorem \ref{thm:exist}. In Section \ref{sec:multi}, we prove several useful multilinear estimates (in particular, frequency restricted estimates) related to the expansion of the nonlinearity into various components. Section \ref{sec:INFR} is devoted to the implementation of the INFR and the derivation of adequate \textit{a priori} bounds. In Section \ref{sec:proof_th1}, we prove Theorem \ref{thm:exist}. The proof of Theorem \ref{thm:lwp} is presented in Section \ref{sec:th2}, while Section \ref{sec:thm1} is devoted to that of Theorem \ref{thm:lwpS=0}.

\bigskip

\textbf{Notation.} 
Given $a\in \R$, $a^+$ (resp. $a^-$) denotes any number sufficiently close to $a$ which is strictly greater (resp. smaller) then $a$. 

Given $a,b\ge 0$, $a\lesssim b$ means there exists a constant $C>0$ (depending at most on $\mu,\nu,\gamma$) such that $a\le Cb$. We may indicate extra dependencies as indices, for example $a \lesssim_n b$. If $a\lesssim b \lesssim a$, we write $a\simeq b$. 

\section{Early steps of the stability proof} \label{sec:2}
	
In view of \eqref{mkdv:profile2}, given $\xi$, all integrals will be done in the convolution hyperplane (in $\m R^3$), which we recall is
\begin{equation} \label{def:H_xi}
	H_\xi = \{ (\xi_1,\xi_2, \xi_3) \in \m R^3 :  \xi = \xi_1 + \xi_2 + \xi_3 \}.
\end{equation}
We also recall the phase function $\Phi$, which we define on the hyperplane (in $\m R^4$) 
\[ H = \{ (\xi, \xi_1,\xi_2, \xi_3) \in \m R^4:  \xi = \xi_1 + \xi_2 + \xi_3 \} \]
 by
\begin{equation} \label{def:Phi}
	\Phi (\xi, \xi_1,\xi_2, \xi_3)=\xi^3-\xi_1^3-\xi_2^3-\xi_3^3.
\end{equation}
The convolution plane $H_\xi$  is endowed with its usual surface measure: we will write it $d\xi_i d\xi_j$ where $i,j$ are two distinct indices in $\llbracket 1, 3 \rrbracket$ to specify which variables are used for the parametrization. As all changes of variable lie in $SL_3(\m R)$, we can choose freely which parametrization to use, and this plays a role for the integration by parts.
	
Given smooth functions $f,g,h \in \q C_c^\infty([0,T]\times\R)$, define
\begin{equation} \label{def:N}
	N[f,g,h](t,\xi) =  \xi \iint\nolimits_{H_\xi} e^{it\Phi}  f(t,\xi_1)  g(t,\xi_2)  h(t,\xi_3) d\xi_1 d\xi_2.
\end{equation}
For the sake of simplicity, we abbreviate $N[f,f,f]$ as $N[f]$. Using these notations, and as we are interested in constructing a solution such that $w(t) \to 0$ as $t \to 0$, the equation for $w$ reads
\begin{align}\label{eq:eq_w}
	\begin{cases}
        \partial_t w(t,\xi)=N[S+z+w]-N[S], \\
        w|_{t=0} = 0.
    \end{cases}
\end{align}
Equivalently, after integration in time:
\[ w(t) = \int_0^t \left( N[S+z+w](s) -N[S](s) \right) ds. \]
	
\subsection{Construction of an approximating sequence}

\label{sec:aprox}Fix a smooth cut-off $\chi\in \mathcal{S}(\R)$, radially decreasing, such that $\chi\equiv 1$ on $[-1,1]$, $0<\chi\le 1$ and $|\chi'(\xi)|\lesssim \chi(\xi)$ for all $\xi \in \m R$. For $n\in\mathbb{N}$, define $\chi_n(\xi):=\chi(\xi/n)$. Given a perturbation $z\in W^{1,\infty}(\R)$, set
	\[ z_n(\xi):=\chi_n(\xi)z(\xi). \]
	Consider the approximate problem
	\begin{equation}\label{eq:approx}
		\begin{cases}
			\partial_t w_n = \chi_n^2 (N[S+z_n+w_n]-N[S]), \\
			w_n\big|_{t=0} =0.
		\end{cases}
	\end{equation}
	
	In order to properly bound the quadratic terms in $S$, we rewrite
	\begin{gather*}
		N[S,S,w+z]=\frac{1}{t^{1/3}}\int \xi e^{it\Psi}(w+z)(t,\xi-\eta)K(S,S)(t^{1/3}\eta) d\eta, \\
		\text{and} \quad N[u]-N[S]=3N[S,S,w+z]+R[S,z,w],
	\end{gather*}
	where $R$ is at least quadratic in $(z,w)$. Finally, given a time interval $I$, we define the function space
\begin{equation} \label{def:Xn}
	X_{n}(I) :=\left\{ w\in \q C_b(I \times \m R): \| w \|_{\q C_b(I,X_{n})}<+\infty \right\} \quad \text{where} \quad  \| v \|_{X_{n}} := \|\chi_n^{-1} v\|_{L^\infty_\mu}.
\end{equation}

	(Here, $v$ is a function not \textit{a priori} depending on time; we denote by $\q C_b$ the set of continuous and bounded functions).
	
	\begin{prop}\label{prop:exist_aprox_0}
		There exists $T_0^n=T_0^n(S,z)>0$ and a unique $w_n\in X_{n}([0,T_0^n))$ (integral) solution to \eqref{eq:approx}. If $T_0^n<+\infty$, then $\|{w}_n(t)\|_{X_n} \to +\infty$ as $t \to T_0^n$. Moreover, there exists $t_{0}^n >0$ such that
		\begin{equation}\label{eq:dtwbounded}
			\forall t \in (0,t_{0}^n], \quad \| \partial_t w_n(t)\|_{X_n} \lesssim_{n,S,z} \jap{t^{-1/2}} (1+\|w_n(t)\|_{X_n}^3).
		\end{equation}
	\end{prop}
	
	\begin{proof}
		The proof follows from a standard fixed point argument on the map
		\begin{equation}\label{eq:fixedpoint}
			\Theta: w \mapsto \left(  (t,\xi) \mapsto \chi_n^2(\xi) \int_0^t (N[S+z_n+w]-N[S])(s,\xi)ds \right)
		\end{equation}
		defined on the complete metric space $(B_{R,T},d)$, where, given $R, T >0$,
		\[
		B_{R,T}=\{w\in \q C_b([0,T]\times \R): \| w\|_{L^\infty([0,T],X_n)} \le  R\},\quad d(v,w)=\|v-w\|_{L^\infty([0,T], X_n)}.
		\]
		Fix $n \in \m N$, and let $w \in B_{R,T}$. A simple computation yields, for $t \in [0,T]$,
		\begin{align}
			|\jap{\xi}^\mu\chi_n(\xi)R[S,z_n,w](t,\xi)|&\lesssim_{S,z} \left(\int \jap{\xi}^\mu|\xi|\chi_n(\xi)\chi_n(\xi_1)\chi_n(\xi_2)d\xi_1d\xi_2\right)(1+\| \chi_n^{-1} w(t) \|_{L^\infty}^3) \nonumber \\
			&\lesssim_{S,z} 1+\|  w(t) \|_{X_n}^3 \lesssim 1+R^3. \label{est:Theta_R}
		\end{align}
		Similarly, as $|K(S,S)(\eta)| \lesssim |\eta|^{-1/2}$ (see \cite{CCV20} or Lemma \ref{lem:estK} below):
		\begin{align}
			|\jap{\xi}^\mu\chi_n(\xi)N[S,S,w+z_n](t,\xi)| & \lesssim_{S,z} \left( \int \frac{\jap{\xi}^\mu|\xi|\chi_n(\xi)\chi_n(\xi-\eta)}{t^{1/2}\sqrt{\eta}} d\eta\right)(1+\| w(t) \|_{X_n}) \nonumber \\
			& \lesssim_{S,z} \frac{1}{t^{1/2}} (1+\| w(t) \|_{X_n}) \lesssim \frac{1}{t^{1/2}} (1+R). \label{est:Theta_L}
		\end{align}
		The point in putting a weight $\chi_n^2$ (with a square) is apparent notably in the above computation, where the factor $\chi_n(\xi) \chi_n(\xi-\eta)$ allows to make the integral bounded uniformly in $\xi$.
		After integrating in $t \in [0,T]$, we infer
		\begin{align*}
			\| \Theta[w]\|_{L^\infty([0,T], X_n)}\lesssim_{S, z} (T^{1/2}+T)(1+R^3).
		\end{align*}
		Furthermore, if $w_1,w_2\in B_{R,T}$ then
		\begin{align*}
			\MoveEqLeft \jap{\xi}^\mu\chi_n(\xi)|R[S,z_n,w_1](t,\xi)-R[S,z_n,w_2](t,\xi) | \\
			& \lesssim_{S,z} \left(\int \jap{\xi}^\mu|\xi|\chi_n(\xi)\chi_n(\xi_1)\chi_n(\xi_2)d\xi_1d\xi_2\right)   (1+\|\chi_n^{-1} {w}_1(t)\|_{L^\infty}^2+\|\chi_n^{-1} {w}_2(t)\|_{L^\infty}) \\
			& \hspace{80mm} \cdot \|\chi_n^{-1}({w}_1-w_2)(t)\|_{L^\infty} \\
			& \lesssim_{S,z} (1+ \| w_1(t) \|_{X_n}^2 + \| w_1(t) \|_{X_n}^2 ) \| w_1(t) - w_2(t) \|_{X_n} \lesssim_{S,z}  (1+R^2)d(w_1,w_2),
		\end{align*}
		and similarly
		\begin{align*}
			\MoveEqLeft \jap{\xi}^\mu\chi_n(\xi)|N[S,S,w_1+z_n](t,\xi)- N[S,S,w_2+z_n](t,\xi)| \\
			& \lesssim \left( \int \frac{\jap{\xi}^\mu|\xi|\chi_n(\xi)\chi_n(\xi-\eta)}{t^{1/2}\sqrt{\eta}} d\eta\right)\| (w_1-w_2)(t)\|_{X_n}\lesssim \frac{1}{t^{1/2}}d(w_1,w_2).
		\end{align*}
		In particular, after integrating in $t \in [0,T]$,
		\[
		d(\Theta[w_1],\Theta[w_2])\lesssim (T^{1/2} +T) (1+R^2)d(w_1,w_2).
		\]
		These estimates show that, for $T$ small (depending on $S$, $z$ and $R$), $\Theta$ is a contraction over $(B_{R,T},d)$, yielding a fixed point $w$. 
		By uniqueness, this solution can be extended up to a maximal time of existence $T_0$ and the blow-up alternative holds.
		
		Finally, the estimate \eqref{eq:dtwbounded} holds due to \eqref{est:Theta_R} and \eqref{est:Theta_L}. 
	\end{proof}
	
	Even though it is possible to obtain \textit{a priori} estimates for $w\in L^\infty_\mu$ alone (Proposition \ref{prop:estimatew}), these do not suffice to extract a limiting profile. We therefore need to work in $W^{1,\infty}_\mu$. Keeping in mind the dilation operator $\Lambda$ \eqref{def:dilation_op}, and given a time interval $I$, we are led to consider the norm
	\begin{gather} 
	 \| w \|_{Y_{n}(I)} := \sup_{t \in I} \left(  \| w (t) \|_{X_{n}} + \left\| \frac{t}{|\xi|} \partial_t {w}(t) \right\|_{X_n} + \left\| \partial_\xi w (t)\right\|_{X_n} \right), \\
	 \text{and the space} \quad Y_{n}(I):=\left\{w\in X_{n}(I)  : \frac{t}{|\xi|} \partial_t {w}, \ \partial_\xi {w} \in X_n(I)  \right\}. \label{def:Yn}
	\end{gather}
If $t \in I$, we make the slight abuse of notation $\| w \|_{Y_n(t)} = \| w \|_{Y_n(\{ t \})}$. The statement below provides the approximation $w_n$ in this functional setting.
	
	\begin{prop}\label{prop:exist_aprox_1}
		There exists $T_1^n=T_1^n(S,z)>0$ and a unique $w_n\in Y_{n}([0,T_1^n))$ (integral) solution to \eqref{eq:approx}. If $T_1^n<+\infty$, then $\| w_n \|_{Y_n(t)} \to +\infty $ as $t \to T_1^n$. Moreover, there exists $t_{1}^n >0$ such that 
		\[
		\forall t \in [0,t_{1}^n], \quad	 \| w_n \|_{Y_n(t)} \lesssim_{n,S,z} (1+ \| w_n \|_{L^\infty([0,t], X_n)}^3) t^{1/2} (1+t).
		\]
	\end{prop}
	
	\begin{proof}
		We consider $\Theta$ defined by \eqref{eq:fixedpoint}, but now over the space
		\[
		B_{R,T} =\left\{ w\in \q C_b([0,T] \times \R): \| w \|_{Y_n([0,T])}  \le R \right\},
		\]
		with $R,T >0$, and endowed with the metric
		\[
		d(v,w)= \| v-w \|_{Y_n([0,T])}.
		\]
		Fix $n \in \m N$ and let $w \in B_{R,T}$. The estimates for $\Theta[w]$ and $\partial_t\Theta[w]$ in $L^\infty_\mu$ are completely analogous to the previous proof, so that
		\[ \| \Theta[w] (t)\|_{X_n} +  \left\| \frac{t}{|\xi|} \partial_t \Theta[w](t) \right\|_{X_n} \lesssim_{S,z} \left( t^{1/2} + t \right) (1+ \| w\|_{L^\infty([0,t],X_n)}^3) \lesssim_{S,z} \left( t^{1/2} + t \right) (1+ R^3). \]
		We therefore focus on the estimate for $\partial_\xi \Theta[w]$. First, observe that since  $S(t,\xi) = S(t^{1/3}\xi)$ (see Proposition \ref{prop:selfsimilar}) and $|\chi_n'|\lesssim \chi_n$,
		\[ \forall t>0, \ \xi \in \m R, \quad 
		|\partial_\xi S(t,\xi)|\lesssim_{S} \frac{t^{1/3}}{\jap{t^{1/3}\xi}},\quad |\partial_\xi z_n(\xi)|\lesssim_{z}  \chi_n(\xi).
		\]
		We have
		\[
		\partial_\xi N[S,S,w+z_n](t,\xi) =  \int \partial_\xi \left[\xi\chi_n^2(\xi) e^{it\Psi}(w+z_n)(t,\xi-\eta) \right] \frac{1}{t^{1/3}}K(S,S)(t^{1/3}\eta)d\eta.
		\]
		If the derivative falls on $\xi\chi_n^2(\xi)$ or $w+z$, the estimate follows as in the previous proof. When it falls on $e^{it\Psi}$, we bound it by
		\begin{align*}
			\MoveEqLeft \left|\jap{\xi}^\mu\int \xi\chi_n(\xi)\partial_\xi( e^{it\Psi})(w+z_n)(t,\xi-\eta)\frac{1}{t^{1/3}}K(S,S)(t^{1/3}\eta)d\eta\right|\\
			& \lesssim_{S,z} \int \left(\frac{\jap{\xi}^\mu|\xi| t (|\xi|^2+|\xi-\eta|^2)\chi_n(\xi)\chi_n(\xi-\eta)}{t^{1/2}\sqrt{\eta}}d\eta\right)(1+\|\chi_n^{-1}w(t)\|_{L^\infty})\lesssim_{S,z} t^{1/2}(1+R).
		\end{align*}

		For $\partial_\xi R[S,z_n,w]$, we gather the terms depending on whether the derivative falls on $\xi\chi_n^2$,  $e^{it\Phi}$ or on the remaining factors (recall that $|\chi'| \lesssim \chi$):
		\begin{align*}
			\MoveEqLeft \left|\jap{\xi}^\mu(\chi_n + 2\xi\chi_n')\int  e^{it\Phi}(3S+w+z_n)(t,\xi_1)(w+z_n)(t,\xi_2)(w+z_n)(t,\xi_3)d\xi_1d\xi_2\right| \\
			& \lesssim_{S,z} \left|\int  \chi_n(\xi_2)\chi_n(\xi_3) d\xi_1d\xi_2\right|(1+\|\chi_n^{-1}w(t)\|_{L^\infty}^3) \lesssim_{S,z} 1+R^3, \\
			\MoveEqLeft \left|\jap{\xi}^\mu\chi_n(\xi)\int \xi\partial_\xi (e^{it\Phi})(3S+w+z_n)(t,\xi_1)(w+z_n)(t,\xi_2)(w+z_n)(t,\xi_3) d\xi_1d\xi_2\right| \\
			& \lesssim_{S,z} \left|\int  \jap{\xi}^\mu\chi_n(\xi)|\xi|t(|\xi_1|^2+|\xi_2|^2+|\xi_3|^2)\chi_n(\xi_2)\chi_n(\xi_3) d\xi_1d\xi_2\right|(1+\|\chi_n^{-1}w(t)\|_{L^\infty}^3) \\
			& \lesssim_{S,z} 1+R^3, \\
			\MoveEqLeft \left|\jap{\xi}^\mu\chi_n(\xi)\int \xi e^{it\Phi}\partial_\xi \left((3S+w+z_n)(t,\xi_1)(w+z_n)(t,\xi_2)(w+z_n)(t,\xi_3)\right) d\xi_1d\xi_2\right| \\
			& \lesssim_{S,z} \left|\jap{\xi}^\mu\chi_n(\xi)\int  |\xi|\chi_n(\xi_2)\chi_n(\xi_3) d\xi_1d\xi_2\right|(1+\|\chi_n^{-1}w(t)\|_{L^\infty}^3) \lesssim_{S,z} 1+R^3.
		\end{align*}
		The integration of these estimates in time yields
		\[
		\left\|\partial_\xi{\Theta[w]} \right\|_{L^\infty([0,T], X_n)} \lesssim_{S,z} T^{1/2}(1+T) (1+R^3).
		\]
		Analogous computations show that one can bound the difference, so that given $w_1,w_2\in B_{R,T}$,
		\[
		d(\Theta[w_1],\Theta[w_2])\lesssim_{S,z} T^{1/2} (1+T) (1+R^2)d(w_1,w_2).
		\]
		Thus, for $T$ small, $\Theta$ is again a contraction over $(B_{R,T},d)$. The rest of the proof follows as in the proof of Proposition \ref{prop:exist_aprox_0}.
	\end{proof}

	The goal is now to show that $w_n$ exist on a uniform (in $n$) time interval, on which one can obtain uniform bounds (in $n$). This is the purpose of the next three sections, and the heart of the paper. 
	Before this, we decompose the nonlinearity into several terms, depending on their behavior.
	
	\subsection{A first decomposition of the nonlinearity}
	
	Recalling \eqref{def:N}, we can expand
	\[ N[S+z+w] - N[S]=:(F+L+Q)[S,z,w], \]
	where $F$ is the source term and does not depend on $w$, $L$ is linear in $w$ and $Q$ includes both quadratic or cubic terms in $w$.
	
	As it is expected, the term $N[S,S,w]$ is the hardest to handle, being the linearized operator around the self-similar solution. An analysis of this term as a trilinear object is not adequate for our purposes, since it is unable to capture the precise oscillatory structure of $S=S_0+S_{\reg}$ (see Proposition \ref{prop:selfsimilar}).
	To this effect, we define
	\[
	K[f,g](\eta) :=  \int e^{3i\eta\lambda^2/4}   f\left(\frac{\eta+\lambda}{2}\right)  g\left(\frac{\eta-\lambda}{2}\right)d\lambda
	\]
	and 
	\begin{equation} \label{def:Psi}
		\Psi(\xi, \eta) : =  -3\eta \xi^2 + 3\xi \eta^2 -3\eta^3/4 = -\frac{3}{4}\eta(\eta-2\xi)^2.
	\end{equation}
	Then we may rewrite
	\begin{align}
		N[S_0,S_0,w](t,\xi) & = \iint\nolimits_{H_\xi} \xi e^{it\Phi} S_0(t^{1/3}\xi_1)   S_0(t^{1/3}\xi_2)   w(t,\xi_3) d\xi_1 d\xi_2\\ 
		& = \int \xi e^{it\Psi}  w(t,\xi-\eta)\left( \int e^{3it\eta\lambda^2/4}   S_0\left(\frac{t^{1/3}(\eta+\lambda)}{2}\right)   S_0\left(\frac{t^{1/3}(\eta-\lambda)}{2}\right)d\lambda\right)d\eta\\ 
		& = \frac{1}{t^{1/3}}\int \xi e^{it\Psi}  w(t,\xi-\eta)K(S_0,S_0)(t^{1/3}\eta)  d\eta=: L_K[w].\label{eq:K}
	\end{align}
	Thus the linear term is decomposed as
	\[ L[S,z,w]=L_K[w]+L_2[S_0,S_{\reg},z,w], \]
	where $L_2$ contains the other terms linear in $w$, which are at most linear in $S_0$.
	
	For the source terms, we perform an analogous splitting. There is no trilinear term in $S$ (due to the cancelation), and once again we set aside the term $N[S_0,S_0,z]$, but this time only for high frequency. 
	
	More precisely, we extract the low-frequency part of all the source terms, which can be bounded directly without any integration in time. For this purpose, given $\tau\in(0,1)$,  we define the regions
	\begin{gather}
		D(\tau)=\left\{(\zeta,\eta)\in \R^2: |\zeta|+|\eta| \ge \tau^{-1/3}/10 \right\}, \nonumber \\
		D_1 = \left\{(\xi_1,\xi_2,\xi_3)\in H_\xi: (\xi_1+\xi_2+\xi_3,\xi_1+\xi_2)\in D(\tau) \right\},\quad D_2= H_\xi \setminus D_1,  \label{def:domains} \\
		D_3 = \left\{(\xi_1,\xi_2,\xi_3)\in H_\xi: |\xi_1|+|\xi_2|+|\xi_3| \ge \tau^{-1/3} \right\},\quad D_4=H_\xi \setminus D_3. \nonumber
	\end{gather}
	The regions $D_i= D_i(\tau,\xi)$, $i \in \llbracket 1,4 \rrbracket$, depend on $\tau$ and $\xi$: we do not write this dependence explicitely to keep notations reasonable. 
	
	The reason for the extra parameter $\tau$ is the following: the domain decomposition should naturally depend on the time\footnote{The correct variable to consider when dealing with self-similar solutions is $t\xi^3$.} $t$. However, when performing integrations by parts in time, if the domain depends explicitely on $t$, the boundary terms can controlled to the expense of lengthy  computations (as it will be clear from Section \ref{sec:INFR}). To avoid those, we decompose time dyadically on intervals $[t_{k+1},t_k]$ where $t_k = T/2^{k}$, and perform integration by parts on $t \in [t_{k+1},t_k]$. In pratice, $\tau$ will be one of the $t_k$ and in any way $\tau/2 \le t \le 2\tau$.
	
	The regular source term is
	\begin{align}
		F_2[S,z](t,\tau,\xi) &:=\xi\int_{D_4}e^{it\Phi}({z}(\xi_1)+3{S}(t,\xi_1)){z}(\xi_2){z}(\xi_3)d\xi_1 d\xi_2 \\
		& \qquad + 3\xi \int_{D_2} e^{it\Phi}{S}(t,\xi_1){S}(t,\xi_2){z}(\xi_3)d\xi_1 d\xi_2.\label{eq:decompF2}
	\end{align}
	The point in considering these regions is that, in these low-frequency domains, integration in time (and so the INFR) is unable to gain some regularity. However, we can exploit the oscillations in frequency for the self-similar solution by integrating by parts in space and we  prove bounds on $F_2$ in weighted $L^\infty$ spaces directly. This is done in Proposition \ref{prop:control_F2}.
	
	The remaining (high-frequency) source term writes
	\begin{align}
		F_1[S,z](t,\tau,\xi) &=\xi\int_{D_3}e^{it\Phi}({z}(\xi_1)+3{S}(t,\xi_1)){z}(\xi_2){z}(\xi_3)d\xi_1 d\xi_2 \\
		& \qquad + 3\xi \int_{D_1} e^{it\Phi}(S_{\reg}+S_0)(t,\xi_1){S}_{\reg}(t,\xi_2){z}(\xi_3)d\xi_1 d\xi_2\\
		& \qquad + 3\xi \int_{D_1} e^{it\Phi}{S_0}(t,\xi_1){S_0}(t,\xi_2){z}(\xi_3)d\xi_1 d\xi_2\\
		&= \xi\int_{D_3}e^{it\Phi}({z}(\xi_1)+3{S}(t,\xi_1)){z}(\xi_2){z}(\xi_3)d\xi_1 d\xi_2 \\
		& \qquad + 3\xi \int_{D_1} e^{it\Phi}(S_{\reg}+S_0)(t,\xi_1){S}_{\reg}(t,\xi_2){z}(\xi_3)d\xi_1 d\xi_2\\
		& \qquad +3\xi \int_{D(\tau)}  e^{it\Psi} z(\xi-\eta)\frac{1}{t^{1/3}}K(S_0,S_0)(t^{1/3}\eta)d\eta \\
		& =:F_{11}[S,z] +F_{12}[S_0,S_{\text{reg}},z] + L_{K,D(\tau)}[z].\label{eq:decompF1}
	\end{align}
	In conclusion,
	\begin{align}\label{eq:decomp_nonli}
		N[S+z+w]-N[S] & = F_{11}[S,z] +F_{12}[S_0,S_{\reg},z] + F_2[S,z] +  L_{K,D(\tau)}[z] \\
		& \qquad + L_K[w] + L_2[S_0,S_{\reg},z,w] + Q[S,z,w].
	\end{align}
	
	For the terms other than $F_2$, the oscillations in time can be exploited, leading us to the application of the INFR: the construction is made in Section \ref{sec:INFR}. In order to be able to bound the terms in infinite expansion, we first prove in Section \ref{sec:frequencyrestricted} the necessary frequency-restricted estimates (Proposition \ref{prop:fre_phi} and Proposition \ref{prop:fre_psi}). Moreover, source terms require an extra care, as we will need a polynomial bound in time for these terms (which ultimately dictates the polynomial growth of $w$). This is done in Propositions \ref{prop:estPhi_source}, \ref{lem:estSregz} and \ref{prop:sourceSz}.

	\section{Multilinear estimates} \label{sec:multi}
	
	\subsection{Bounds on the self-similar solution}
	
	Here we are interested in taking care of the oscillations of the self-similar solution, by deriving bounds on $K(S,S)$.
	
	\begin{lem}\label{lem:estK}
		Let $f,g\in W^{1,\infty}_{0,1}(\R\setminus \{0\})$, then
		\[  \forall \eta \in \m R \setminus \{ 0 \}, \quad |K(f,g)(\eta)|\lesssim \frac{1}{|\eta|^{1/2}}\|f\|_{W^{1,\infty}_{0,1}(\R\setminus \{0\})}\|g\|_{W^{1,\infty}_{0,1}(\R\setminus \{0\})}. \]
	\end{lem}
	\begin{proof}
		Assume that $f$ and $g$ have unit norms. Given $\eta>0$, write
		\[
		h(\eta,\lambda)= f\left(\frac{\eta+\lambda}{2}\right)g\left(\frac{\eta-\lambda}{2}\right).
		\]
		Then
		\[
		|h(\eta,\lambda)|\lesssim 1 \quad \text{and} \quad |\partial_\lambda h(\eta,\lambda)|\lesssim \frac{1}{\jap{\eta+\lambda}}+\frac{1}{\jap{\eta-\lambda}}.
		\]
		Hence
		\begin{align*}
			| K(f,g)(\eta)|&=\frac{1}{\sqrt{\eta}}\left|\int e^{3i\zeta^2/4}h\left(\eta,\zeta/\sqrt{\eta}\right)d\zeta\right| \\&\lesssim \frac{1}{\sqrt{\eta}}\left| \frac{\zeta h(\eta, \zeta/\sqrt{\eta})}{1+3i\zeta^2/2} \right|_{\zeta=\pm|\eta|^{3/2}} + \frac{1}{\sqrt{\eta}}\left|\int \zeta e^{3i\zeta^2/4}\partial_\zeta \left(\frac{1}{1+3i\zeta^2/2}h\left(\eta,\zeta/\sqrt{\eta}\right)\right)d\zeta\right|\\&\lesssim\frac{1}{\sqrt{\eta}} + \frac{1}{\sqrt{\eta}} \int \frac{\zeta^2}{|1+i\zeta^2|^2} d\zeta + \frac{1}{\sqrt{\eta}} \int\frac{\zeta}{|1+i\zeta^2|}\frac{1}{\sqrt{\eta}}\left(\frac{1}{\jap{\eta+\zeta/\sqrt{\eta}}} + \frac{1}{\jap{\eta-\zeta/\sqrt{\eta}}}\right)d\zeta.
		\end{align*}
		The first integral is bounded. For the second, if $|\zeta| \le 10 |\eta|^{3/2}$,
		\[
		\int\frac{\zeta}{|1+i\zeta^2|}\frac{1}{\sqrt{\eta}}\left(\frac{1}{\jap{\eta+\zeta/\sqrt{\eta}}} + \frac{1}{\jap{\eta-\zeta/\sqrt{\eta}}}\right)d\zeta \lesssim \int_{|\zeta|\lesssim |\eta|^{3/2}}\frac{\zeta}{|1+i\zeta^2|}\frac{1}{\sqrt{\eta}}d\zeta \lesssim 1.
		\]
		If $|\zeta|\ge 10 |\eta|^{3/2}$, then $|\eta\pm\zeta/\sqrt{\eta}|\sim |\zeta/\sqrt{\eta}|$ and
		\[
		\int\frac{\zeta}{|1+i\zeta^2|}\frac{1}{\sqrt{\eta}}\left(\frac{1}{\jap{\eta+\zeta/\sqrt{\eta}}} + \frac{1}{\jap{\eta-\zeta/\sqrt{\eta}}}\right)d\zeta \lesssim \int \frac{1}{|1+i\zeta^2|}d\zeta <+\infty.
		\]
		The proof for $\eta<0$ follows from analogous computations.
	\end{proof}
	
	We will also  need some estimates on the derivative, which is the purpose of the next result.

	\begin{lem}\cite[Lemma 14]{CCV19} \label{lem:estdK}
		One has
		\begin{equation}\label{eq:dK}
			\forall \xi \in \m R^*, \quad |\partial_\xi K(S_0,S_0)(\xi)|\lesssim \frac{\epsilon^2}{|\xi|^{3/2}}.
		\end{equation}
	\end{lem}

	In Lemma \ref{lem:estK}, there is no need to split $S$ into $S_0$ and $S_{\reg}$. However, it is not known whether the bound \eqref{eq:dK} holds for $\partial_\xi K(S,S)$. Indeed, in \cite[Lemma 14]{CCV19}, one uses crucially the exact expression of $S_0$ and its derivative. In order to derive the estimate for $\partial_\xi K(S,S)$, as we do not have an explicit formula for $S_{\reg}$, we would need to use bounds up to the \textit{second} derivative of $S_{\reg}$. The existence of such bounds is not a trivial problem, due to the presence of the highly oscillatory term $e^{-8i\xi^3/9}$ in \eqref{asymp}.

	\subsection{Bounds on \texorpdfstring{$F_2$}{F2}}
	
	\begin{lem}\label{lem:D4}
		Given $t,\tau\in(0,1)$, 
		\[
		\left\|\xi\int_{D_4}e^{it\Phi}f(\xi_1)g(\xi_2)h(\xi_3)d\xi_1 d\xi_2\right\|_{L^\infty_{\mu}}\lesssim \tau^{-1+\frac{\nu-\mu}{3}}\|f\|_{L^\infty}\|g\|_{L^\infty}\|h\|_{L^\infty_\nu}.
		\]
	\end{lem}
	
	\begin{proof}
		Recall the definition \eqref{def:domains} of $D_4$, so that $D_4$ is empty if $|\xi| > \tau^{-1/3}$. Now, if $|\xi| \le \tau^{-1/3}$, then
		\begin{align*}
			\MoveEqLeft |\xi|\jap{\xi}^\mu\int_{D_4} |f(\xi_1)g(\xi_2)h(\xi_3)|d\xi_1 d\xi_2\\
			& \lesssim \left(\tau^{-(1+\mu)/3}\int_{|\xi_2|+|\xi_3|\le \tau^{-1/3}}\frac{d\xi_2d\xi_3}{\jap{\xi_3}^\nu}\right)\|f\|_{L^\infty}\|g\|_{L^\infty}\|h\|_{L^\infty_\nu} \\
			&\lesssim \tau^{-1+\frac{\nu-\mu}{3}}\|f\|_{L^\infty}\|g\|_{L^\infty}\|h\|_{L^\infty_\nu}. \qedhere
		\end{align*}
	\end{proof}
	
	\begin{prop}\label{prop:control_F2}
		For $t,\tau\in(0,1)$,
		\begin{align*}
			\|F_2[S,z](t,\tau)\|_{L^\infty_{\mu}} &\lesssim (t^{-1} + \tau^{-1}) \tau^{\frac{\nu-\mu}{3}} \Big[ \|{S}\|_{W^{1,\infty}_{0,1}(\R\setminus\{0\})}^2\|{z}\|_{L^\infty_\nu} + \|{z}\|_{L^\infty_\nu}^3\Big].
		\end{align*}
	\end{prop}
	\begin{proof}
		Recall that 
		\begin{align*}
			F_2[S,z](t,\tau,\xi)&=\xi\int_{D_4}e^{it\Phi}({z}(\xi_1)+3{S}(t,\xi_1)){z}(\xi_2){z}(\xi_3)d\xi_1 d\xi_2 \\
			& \quad + 3\xi \int_{D_2} e^{it\Phi}{S}(t,\xi_1){S}(t,\xi_2){z}(\xi_3)d\xi_1 d\xi_2.
		\end{align*}
		By Lemma \ref{lem:D4}, these integrals are bounded adequately in $D_4$. It remains to estimate the second integral in $D_2\setminus D_4$. In this region, we have $|\xi_3| \le |\xi| + |\xi_1+\xi_2| \le \tau^{-1/3}/5$ and $ |\xi_1| \ge 7/20 \tau^{-1/3}$ (because $2|\xi_1| \ge |\xi_1| + |\xi_2| +|\xi_3| - |\xi_1 + \xi_2| - |\xi_3| \ge 7\tau^{-1/3}/10$). As a consequence, $|\partial_{\xi_1}\Phi|=3|\xi_1^2-\xi_3^2|\gtrsim |\xi_1|^2$ and we can integrate by parts in $\xi_1$ (observe that, since $\xi_1$ and $\xi_2$ are far from 0, no boundary terms appear from the jump of $S$ at $\xi=0$):
		\begin{align*}
			\MoveEqLeft \left|\xi \int_{D_2\setminus D_4} e^{it\Phi}{S}(t,\xi_1){S}(t,\xi_2){z}(\xi_3)d\xi_1 d\xi_2\right| \\ 
			& \lesssim \left|\xi \int_{D_2 \setminus D_4} e^{it\Phi}\partial_{\xi_1}\left(\frac{1}{it\partial_{\xi_1}\Phi}{S}(t,\xi_1){S}(t,\xi_2)\right){z}(\xi_3)d\xi_1 d\xi_3\right|\\
			&\qquad + \left| \xi\int_{\partial(D_2\setminus D_4)} \frac{e^{it\Phi}}{it\partial_{\xi_1}\Phi}S(t,\xi_1)S(t,\xi_2)z(\xi_3)d\sigma(\xi_1,\xi_3)  \right|.
		\end{align*}
		Omitting norms in $S$ and $z$,
		\begin{align*}
			\MoveEqLeft \left|\xi\jap{\xi}^\mu \int_{D_2\setminus D_4} e^{it\Phi}\partial_{\xi_1}\left(\frac{1}{it\partial_{\xi_1}\Phi}{S}(t,\xi_1){S}(t,\xi_2)\right){z}(\xi_3)d\xi_1 d\xi_3\right| \\
			&\qquad + \left| \xi\jap{\xi}^\mu\int_{\partial(D_2\setminus D_4)} \frac{e^{it\Phi}}{it\partial_{\xi_1}\Phi}S(t,\xi_1)S(t,\xi_2)z(\xi_3)d\sigma(\xi_1,\xi_3)  \right| \\
			& \lesssim 	\left|\xi\jap{\xi}^\mu \int_{D_2\setminus D_4} \frac{1}{t\xi_1^3}\frac{1}{\jap{\xi_3}^\nu}d\xi_1 d\xi_3\right| + \left| \xi\jap{\xi}^\mu\int_{\partial(D_2\setminus D_4)} \frac{1}{t\xi_1^2}\frac{1}{\jap{\xi_3}^\nu}d\sigma(\xi_1,\xi_3)  \right| \\
			& \lesssim  t^{-1}\tau^{-(1+\mu)/3}\left(\int_{|\xi_3|\le \tau^{-1/3}/5} \frac{d\xi_3}{\jap{\xi_3}^\nu} \right)\left( \int_{|\xi_1| \ge 7\tau^{-1/3}/20} \frac{d\xi_1}{\xi_1^3} \right) + t^{-1}\tau^{\frac{\nu-\mu}{3}} \lesssim t^{-1}\tau^{\frac{\nu-\mu}{3}}. \qedhere
		\end{align*}
	\end{proof}

	\subsection{Frequency-restricted estimates}\label{sec:frequencyrestricted}
	
	In this section, we derive the frequency-restricted estimates needed for the implementation of the infinite normal form reduction (see Section \ref{sec:INFR}). We begin with some elementary computations.
	
	\begin{lem}\label{lem:quadraticas}
		Given $\alpha\in \R$ and $M>0$,
		\begin{align}\label{eq:posdef}
			\iint_{B_1(0)} \m 1_{|q_1^2+q_2^2-\alpha|<M}dq_1dq_2 &\lesssim \min(1,M), \\\label{eq:indef}
			\iint_{B_1(0)} \m 1_{|q_1q_2-\alpha|<M}dq_1dq_2 &\lesssim \min(1,M|\ln M|), \\
			\label{eq:singular}
			\forall \delta \in (0,1), \quad \int \m 1_{|q-\alpha| \le M} \frac{dq}{\left|q \right|^{\delta}} &\lesssim_{\delta} M^{1-\delta}, \quad \text{and} \quad \\
			\label{eq:quadratica}
			\int_{-1}^{1} \m 1_{|q^2  -\alpha | < M} dq &\lesssim \min(1,\sqrt M).
		\end{align}
		
	\end{lem}
	
	\begin{proof}
		The first estimate is direct. For \eqref{eq:indef}: if $M<10$, the integral is uniformly bounded; if $M>10$,
		\begin{align*}
			\iint_{B_1(0)} \m 1_{|q_1q_2-\alpha|<M}dq_1dq_2 &\lesssim \int_{M<|q_2|<1} \int \mathbbm{1}_{|q_1-\alpha/q_2|<M/q_2} dq_1dq_2 + \int_{|q_2|<M} \int_{|q_1|<1} 1 dq_1dq_2 \\&\lesssim \int_{M<|q_2|<1} \frac{dq_2}{|q_2|} + M \lesssim M|\ln M|.
		\end{align*}
		The proof of \eqref{eq:singular} follows from direct integration if $|\alpha|<2M$. Otherwise,
		$$
		\int \m 1_{|q-\alpha| \le M} \frac{dq}{\left|q \right|^{\delta}} \lesssim (|\alpha|+M)^{1-\delta}-(|\alpha|-M)^{1-\delta}\lesssim \frac{M}{(|\alpha|+M)^{\delta}} \lesssim M^{1-\delta}.
		$$
		Estimate \eqref{eq:quadratica} follows from \eqref{eq:singular} with $\delta=1/2$ by a change of variables.
	\end{proof}
	
	
	\begin{prop}\label{prop:fre_phi}
		Given $M \ge 1$,
		\begin{equation}\label{eq:TalphaM_multiplier}
			\sup_{\xi,\alpha}\  \jap{\xi}^\mu\iint  \left(\max_{j=1,2,3}|\xi_j|\right)\m 1_{|\Phi -\alpha| < M} \frac{d\xi_1 d\xi_2}{\jap{\xi_1}^\mu\jap{\xi_2}^\mu} \lesssim M^{1-\mu/3}.  
		\end{equation}
		In particular,
		\begin{equation}\label{eq:TalphaM_multiplierxi}
			\sup_{\xi,\alpha}\   \jap{\xi}^\mu |\xi| \iint \m 1_{|\Phi -\alpha| < M} \frac{d\xi_1 d\xi_2}{\jap{\xi_1}^\mu\jap{\xi_2}^\mu} \lesssim M^{1-\mu/3}.  
		\end{equation}
	\end{prop}

	
	\begin{proof}
		We reduce \eqref{eq:TalphaM_multiplier} to the case $|\xi_1|\le|\xi_2|\le |\xi_3|$, since
		\begin{align*}
			\MoveEqLeft \iint \left(\max_{j=1,2,3}|\xi_j|\right)\m 1_{|\Phi-\alpha| \le M} \frac{\jap{\xi}^{\mu}}{\jap{\xi_2}^\mu\jap{\xi_3}^\mu}d\xi_1 d\xi_2 \\
			& \le   \iint \left(\max_{j=1,2,3}|\xi_j|\jap{\xi_j}^\mu \right) \m 1_{|\Phi-\alpha| \le M} \frac{ \jap{\xi}^{\mu}}{\jap{\xi_1}^\mu \jap{\xi_2}^\mu \jap{\xi_3}^\mu}d\xi_1 d\xi_2 \\
			& \le 6 \iint  \m 1_{|\xi_1| \le |\xi_2| \le |\xi_3|}  \m 1_{|\Phi-\alpha| \le M} \frac{\jap{\xi}^{\mu}|\xi_3|}{\jap{\xi_1}^\mu \jap{\xi_2}^\mu} d\xi_1 d\xi_2.
		\end{align*}
		First notice that
		\[  \iint  \m 1_{|\xi_3| \le 1} \m 1_{|\xi_1| \le |\xi_2| \le |\xi_3|} \m 1_{|\Phi-\alpha| \le M} \frac{\jap{\xi}^{\mu}|\xi_3|}{\jap{\xi_1}^\mu \jap{\xi_2}^\mu} d\xi_1 d\xi_2 \lesssim \iint_{[-1,1]^2} d\xi_1 d\xi_2 = O(1), \]
		because $|\xi_3| \ge |\xi|/3$. Hence we want to bound
		\[ I := \iint  \m 1_{|\xi_3| \ge 1} \m 1_{|\xi_1| \le |\xi_2| \le |\xi_3|}  \m 1_{|\Phi-\alpha| \le M} \frac{|\xi_3| \jap{\xi}^{\mu}}{\jap{\xi_1}^\mu \jap{\xi_2}^\mu} d\xi_1 d\xi_2. \]
		We split the integration into different regions, depending on the behavior of the phase function $\Phi$. We view $\Phi$ as a function of $(\xi_1,\xi_2)$, with $\xi_3$ given by $\xi-\xi_1-\xi_2$, so that
		\[ \frac{\partial \Phi}{\partial \xi_1} = - 3 (\xi_1^2 - \xi_3^2), \quad  \frac{\partial \Phi}{\partial \xi_2} = - 3 (\xi_2^2 - \xi_3^2). \] 
		The idea now is to perform a change of variable in the intermediate frequency $\xi_2 \to \Phi$ and then integrate in the smallest frequency $\xi_1$. This is possible when $\frac{\partial \Phi}{\partial \xi_2}$ is of the order of $\xi_3^2$ (Case 4). When it is too small, then we are either close to a stationary point of $\Phi$ (Cases 1 and 2), given by
		\[ s_0 = (\xi/3,\xi/3), \quad s_1 = (\xi,\xi), \quad s_2 = (\xi,-\xi), \quad s_3 = (-\xi,\xi), \]
		or we're near $|\xi_2|=|\xi_3|$ with $|\xi_1|\not\simeq |\xi_3|$ (Cases 3 and 5). 
		
		\bigskip
		
		\noindent \textit{The stationary cases.} In this region, all frequencies are comparable. The bound will follow from a careful description of the geometry of $\Phi$ coming from Morse theory.
		\medskip
		
		Let $c>0$ small to be chosen later.
		
		\emph{Case 1. Near the stationary point $s_0$, that is, }
		\[ (\xi_1,\xi_2)\in B_{c|\xi|}(s_0). \]

		In this region, all frequencies are of order $\xi/3$. In particular, as $|\xi_3| \ge 1$, $|\xi| \gtrsim \jap{\xi}$. We normalize the frequencies, so that the region becomes independent of $|\xi|$. Denote $r_0 = (1/3,1/3)$, \begin{equation}\label{eq:normalize}
			p_j=\xi_j/\xi\text{ for }i=1,2,\text{ and }\phi(p_1,p_2) = 1 - p_1^3 - p_2^3 - (1-p_1-p_2)^3,
		\end{equation}
		so that $\Phi(\xi_1,\xi_2) = \xi^3 \phi(p_1,p_2)$. Then
		\begin{align*}
			I_0 & := \iint_{B_{c|\xi|}(s_0) }\m 1_{|\Phi - \alpha| \le M} \frac{\jap{\xi}^{\mu}|\xi_3|}{\jap{\xi_1}^\mu\jap{\xi_2}^\mu}d\xi_1 d\xi_2 \lesssim |\xi|^{1-\mu} \iint_{B_{c|\xi|}(s_0) } \m 1_{|\xi^3 \phi - \alpha| \le M} d\xi_1 d\xi_2 \\
			& \lesssim |\xi|^{3-\mu}  \int_{B_c(r_0)} \m 1_{|\xi^3 \phi - \alpha| \le M} dp_1dp_2.
		\end{align*}
		Observe that
		\[ \phi(r_0) = \frac{8}{9}, \quad \nabla \phi(r_0) =0, \quad D^2 \phi(r_0) = -2 \begin{pmatrix} 
			2  & 1 \\
			1 & 2 \\
		\end{pmatrix} \text{ is definite negative.} \]
		By Morse's lemma, if $c$ is sufficiently small, there exist a domain $D_0 \subset B_1(0)$ in $\m R^2$ and a $\q C^1$ diffeomorphism $\varphi_0: D_0 \to B_c(r_0)$ such that $\varphi_0(0) = r_0$ and 
		\[ \forall q \in D_0, \quad \phi(\varphi_0(q)) = \frac{8}{9} - |q|^2. \]
		Therefore, by \eqref{eq:posdef},
		\begin{align*}
			\iint_{B_c(r_0)}  \m 1_{|\xi^3 \phi - \alpha| \le M} dp_1dp_2  &= \iint_{D_0} \m 1_{|\xi^3 ( \frac{8}{9} -  |q|^2) - \alpha| \le M} |\det D \varphi_0(q)| dq \\&\lesssim  \iint_{B_{1}(0)} \m 1_{\left| |q|^2- \left(8/9 - \frac{\alpha}{\xi^3}\right) \right| \le \frac{M}{|\xi|^3}} dq \lesssim |\xi|^{3-\mu}\min \left(1,\frac{M}{|\xi|^3}\right)\lesssim M^{1-\mu/3}. 
		\end{align*} 

		\emph{Case 2. Near the other stationary points, that is,}
		\[
		(\xi_1,\xi_2)\in B_{c|\xi|}(s_j),\quad \text{for } j=1,2,3.
		\]
		We perform the analysis around $s_1$, the other ones being similar. As before, denote $r_1 = (1,1)$. Using the normalization \eqref{eq:normalize},
		\begin{align*}
			I_1 & := \iint_{B_{c |\xi|}(s_1) } 1_{|\Phi - \alpha| \le M} \frac{\jap{\xi}^{\mu}|\xi_3|}{\jap{\xi_1}^\mu\jap{\xi_2}^\mu}d\xi_1 d\xi_2  \lesssim |\xi|^{3-\mu}\iint_{B_{c}(r_1) }  1_{|\xi^3 \phi - \alpha| \le M} dp_1dp_2.
		\end{align*}
		We compute
		\[ \phi(r_1) = 0, \quad \nabla \phi(r_0) =0, \quad \nabla^2 \phi(r_0) = 3 \begin{pmatrix} 
			1  & -2 \\
			-2 & 0 \\
		\end{pmatrix}, \text{ which has signature } (1,-1). \]
		Applying again Morse's lemma, for small $c$, there exist a domain $D_1 \subset B_1(0)$ and a $\q C^1$ diffeomorphism $\varphi_1: D_1 \to B_c(r_1)$ such that $\varphi_1(0) = r_1$ and 
		\[ \forall (q_1,q_2) \in D_1, \quad \phi(\varphi_1(q_1,q_2)) = q_1q_2. \]
		Therefore, by \eqref{eq:indef}
		\begin{align*}
			\MoveEqLeft \iint_{B_c(r_1)}  \m 1_{|\xi^3 \phi - \alpha| \le M} dp_1dp_2 = \iint_{D_1} \m 1_{|\xi^3 q_1q_2 - \alpha| \le M} |\det D \varphi_1(q_1,q_2)| dq_1dq_2 \\
			& \lesssim  \iint_{B_1(0)} \m 1_{\left| q_1q_2 - \frac{\alpha}{\xi^3}\right| \le \frac{M}{|\xi|^3}} dq_1 dq_2   \lesssim \min \left( 1, \frac{M}{| \xi|^3} \left| \ln \frac{M}{|\xi|^3} \right| \right) \lesssim \frac{M^{1-\mu/3}}{| \xi|^{3-\mu}}.
		\end{align*}
		and thus
		\[ I_1 \lesssim  |\xi|^{1-\mu/3} \frac{M^{1-\mu/3}}{| \xi|^{3-\mu}} \lesssim M^{1-\mu/3}. \]

		\bigskip
		
		\noindent \textit{The nonstationary cases.} Before we proceed, let us observe that, in the remaining region
		\[  D: = \left\{  (\xi_1, \xi_2) \in \m R^2 : \forall j=0,\dots, 4,\ |(\xi_1, \xi_2) - s_j| \ge c|\xi|, \ |\xi_1| \le |\xi_2| \le |\xi_3|, \ |\xi_3| \ge 1 \right\}, \]
		there exists a constant $c'>0$ (depending solely on $c$) such that
		\begin{align} \label{est:xi_1-xi_3-D2}
			\forall (\xi_1, \xi_2) \in D, \quad |\xi_1| \le (1-c') |\xi_3|.
		\end{align}
		Indeed, if this was not the case, we would be near one of the four stationary points. 
		As a consequence,
		\begin{align}
			\label{est:dPhi1}
			\forall (\xi_1, \xi_2) \in D, \quad \left|\frac{\partial \Phi}{\partial \xi_1}\right| = 3 (\xi_3^2 - \xi_1^2) \ge 3 c'  \xi_3^2.
		\end{align}
		
		\bigskip
		
		\emph{Case 3. In the comparable-frequencies region.}
		
		Let $d>0$ small to be chosen later, and denote
		\[ D_3 =  D \cap \{ (\xi_1, \xi_2) \in \m R^2: |\xi_1| \ge d |\xi_3| \}. \]
		Since $|\xi_3|\ge |\xi_1|$, we have $|\xi_1|\simeq |\xi_2|\simeq |\xi_3|\gtrsim |\xi|$. 
		
		We decompose $D_3$ further into
		\[
		D_{31}=D_3\cap\{ (\xi_1, \xi_2) \in \m R^2: |\xi| \ge d |\xi_1| \},\quad D_{32}=D_3\setminus D_{31}.
		\]
		In $D_{31}$, we perform a change of variable $\varphi_3: (\xi_1,\xi_2) \mapsto (\Phi(\xi_1, \xi_2),\xi_2)$ and 
		\begin{align*}
			I_{31} & := \iint_{D_{31}} \m 1_{|\xi_1| \le |\xi_2| \le |\xi_3|}  \m 1_{|\Phi-\alpha| \le M} \frac{\jap{\xi}^{\mu}|\xi_3|}{\jap{\xi_1}^\mu \jap{\xi_2}^\mu} d\xi_1 d\xi_2 \le \iint_{\varphi_3(D_{31})}  \frac{\jap{\xi}^{\mu}|\xi_3|}{\jap{\xi_2}^{2\mu}} \m 1_{|\Phi-\alpha| \le M} \frac{d\xi_2 d\Phi}{\left| \frac{\partial \Phi}{\partial \xi_1}\right|} \\
			& \lesssim  \iint \frac{1}{|\xi_3|^\mu|\xi|}  \m 1_{|\Phi-\alpha| \le M} d\xi_2 d\Phi.
		\end{align*}
		As $|\Phi| \lesssim  |\xi_3|^3$, by \eqref{eq:singular},
		\begin{align*}
			I_{31} \lesssim  \left( \int \m 1_{|\Phi-\alpha| \le M} \frac{d\Phi}{\left| \Phi \right|^{\mu/3}} \right)\left(\int_{|\xi_2|\sim |\xi|} \frac{d\xi_2}{|\xi|} \right)\lesssim M^{1-\mu/3}.
		\end{align*}
		%
		%
		%
		
		In $D_{32}$, $\xi_1+\xi_2\simeq -\xi_3$. We perform the change of variables $(\xi_1,\xi_2)\mapsto (\Phi,p)=(\Phi(\xi_1,\xi_2),\xi_1/\xi_2)$, whose associated jacobian is
		\begin{equation}\label{eq:jacobiano}
			\left|\frac{\partial\Phi}{\partial\xi_1}\frac{1}{\xi_1} + \frac{\partial \Phi}{\partial \xi_2}\frac{\xi_2}{\xi_1^2}\right| = \left|\frac{\xi_1^3+\xi_2^3-\xi_3^2(\xi_1+\xi_2)}{\xi_2^2} \right|\simeq |\xi_3|,
		\end{equation}
		when $d$ is small enough. Thus
		\begin{align}
			I_{32} & := \iint_{D_{32}} \m 1_{|\xi_1| \le |\xi_2| \le |\xi_3|}  \m 1_{|\Phi-\alpha| \le M} \frac{\jap{\xi}^{\mu}|\xi_3|}{\jap{\xi_1}^\mu \jap{\xi_2}^\mu} d\xi_1 d\xi_2 \lesssim \iint_{|p|<1}  \frac{\jap{\xi}^{\mu}|\xi_3|}{\jap{\xi_2}^{2\mu}} \m 1_{|\Phi-\alpha| \le M} \frac{ d\Phi dp}{|\xi_3|} \\
			& \lesssim  \iint_{|p|<1}  \frac{1}{|\xi_3|^\mu}  \m 1_{|\Phi-\alpha| \le M} d\Phi dp\lesssim \iint_{|p|<1}  \frac{1}{|\Phi|^{\mu/3}}  \m 1_{|\Phi-\alpha| \le M} d\Phi dp\lesssim M^{1-\mu/3}.\label{eq:I22}
		\end{align}

		\bigskip
		
		\emph{Case 4. Away from the bisectors $\xi_3 = \pm \xi_2$, with small $\xi_1$.}
		
		We consider the domain
		\[ D_4 = D \cap \{ (\xi_1, \xi_2) \in \m R^2: |\xi_1|\le d|\xi_3|,\ |\xi_2| \le (1-d^2) |\xi_3| \}. \]
		In this region, there exists $c'' >0$ such that
		\[ \forall (\xi_1,\xi_2) \in D_4, \quad  \left|\frac{\partial \Phi}{\partial \xi_2}\right| = 3 (\xi_3^2-\xi_2^2) \ge c'' \xi_3^2. \]
		We divide this region as
		$$
		D_{41}=D_4\cap \{(\xi_1,\xi_2)\in\m R^2 : |\xi_2|\ge |\xi|\},\quad D_{42}=D_4\setminus D_{41}.
		$$

		Over $D_{41}$, we perform the change of variables $(\xi_1,\xi_2)\mapsto (\Phi,p)=(\Phi(\xi_1,\xi_2),\xi_1/\xi_2)$ as in the previous case. Observe that, since we're far away from the region $|\xi_2|=|\xi_3|$, \eqref{eq:jacobiano} holds. Since $|\Phi|\lesssim |\xi_2|^3$,
		\begin{align*}
			I_{41} & := \iint_{D_{41}} \m 1_{|\xi_1| \le |\xi_2| \le |\xi_3|}  \m 1_{|\Phi-\alpha| \le M} \frac{\jap{\xi}^{\mu}|\xi_3|}{\jap{\xi_1}^\mu \jap{\xi_2}^\mu} d\xi_1 d\xi_2 \lesssim \iint_{|p|<1}  \frac{|\xi_3|}{\jap{\xi_2p}^{\mu}} \m 1_{|\Phi-\alpha| \le M} \frac{ d\Phi dp}{|\xi_3|} \\
			&\lesssim  \iint_{|p|<1} \frac{1}{|\xi_2p|^\mu}  \m 1_{|\Phi-\alpha| \le M} d\Phi dp \lesssim  \iint_{|p|<1} \frac{1}{|\Phi|^{\mu/3}|p|^\mu}  \m 1_{|\Phi-\alpha| \le M} d\Phi dp\lesssim M^{1-\mu/3}.
		\end{align*}
		For $D_{42}$, we have $|\xi|\gtrsim |\xi_3|$ and it suffices to consider the change of variables $\xi_2\mapsto \Phi$:
		\begin{align*}
			I_{42} & := \iint_{D_{42}} \m 1_{|\xi_1| \le |\xi_2| \le |\xi_3|}  \m 1_{|\Phi-\alpha| \le M} \frac{\jap{\xi}^{\mu}|\xi_3|}{\jap{\xi_1}^\mu \jap{\xi_2}^\mu} d\xi_1 d\xi_2 \\&\lesssim \int \frac{|\xi|^{1+\mu}}{\jap{\xi_1}^{2\mu}} \left(\int \m 1_{|\Phi-\alpha| \le M} d\xi_2 \right)^{1-\mu/3}\left(\int_{|\xi_2|\le|\xi|} d\xi_2\right)^{\mu/3} d\xi_1 \\&\lesssim |\xi|^{-1+2\mu}\int_{|\xi_1|\le|\xi|} \frac{1}{\jap{\xi_1}^{2\mu}} \left(\int \m 1_{|\Phi-\alpha| \le M} d\Phi \right)^{1-\mu/3} d\xi_1 \lesssim M^{1-\mu/3}.
		\end{align*}
		%
		
		\bigskip
		
		\emph{Case 5. Near the bisectors $\xi_3=\pm\xi_2$, with small $\xi_1$.}
		We work on
		\[
		D_5^\pm=D\cap \{(\xi_1,\xi_2)\in \R^2 : |\xi_1|\le d|\xi_3|, |\xi_2\pm\xi_3|\le d|\xi_3| \}.
		\]
		We split the domain further into
		\[
		D_{51}^\pm=D_5\cap \{(\xi_1,\xi_2)\in \R^2 : |\xi_1|\ge d|\xi| \},\quad D_{52}^\pm=D_5^\pm\setminus D_{51}^\pm.
		\]
		In $D_{51}^\pm$, $\left|\frac{\partial \Phi}{\partial \xi_1}\right|\sim \xi_2^2$ and $|\Phi|\lesssim |\xi_2|^3$. Performing the change of variables $\xi_1\mapsto \Phi$,
		\begin{align}
			I_{51}^\pm & := \iint_{D_{51}^\pm} \m 1_{|\xi_1| \le |\xi_2| \le |\xi_3|}  \m 1_{|\Phi-\alpha| \le M} \frac{\jap{\xi}^{\mu}|\xi_3|}{\jap{\xi_1}^\mu \jap{\xi_2}^\mu} d\xi_1 d\xi_2 \lesssim \iint_{|\xi_2|\gtrsim |\Phi|^{1/3}}  \frac{|\xi_3|}{\jap{\xi_2}^\mu}\mathbbm{1}_{|\Phi-\alpha|<M}\frac{d\Phi d\xi_2}{|\xi_2|^2}\\& \lesssim \int \mathbbm{1}_{|\Phi-\alpha|<M} \left(\int_{|\xi_2|\gtrsim |\Phi|^{1/3}}\frac{1}{|\xi_2|^{1+\mu}}d\xi_2  \right)d\Phi\lesssim M^{1-\mu/3}.
		\end{align}
		In $D_{52}^\pm$, we have $|\xi_3|\simeq |\xi_2| \gtrsim |\xi| \ge |\xi_1|/d$, which implies 
		\[
		\left|\frac{\partial \Phi}{\partial \xi_2}\right|\simeq |\xi\xi_2| \quad \text{and} \quad |\xi\Phi|\simeq |\xi\xi_2^2(\xi-\xi_1)|\simeq \left|\frac{\partial \Phi}{\partial \xi_2}\right|^2.
		\]
		For $\xi_1$ fixed, this implies that
		\[
		\int |\xi_2|^{1-\mu}\mathbbm{1}_{|\Phi-\alpha|<M}d\xi_2 \lesssim 	\int \frac{1}{|\xi|\jap{\xi}^\mu}\mathbbm{1}_{|\Phi-\alpha|<M}d\Phi\lesssim \frac{M}{|\xi|\jap{\xi}^\mu},
		\]
		and, using \eqref{eq:singular},
		\begin{align*}
			\int |\xi_2|^{1-\mu}\mathbbm{1}_{|\Phi-\alpha|<M}d\xi_2 &\lesssim 	\int |\xi_2|^{1-\mu}\mathbbm{1}_{|\Phi-\alpha|<M}\frac{1}{\left|\frac{\partial \Phi}{\partial \xi_2}\right|}d\Phi \\&\lesssim \int |\xi_2|^{1-\mu}\mathbbm{1}_{|\Phi-\alpha|<M}\frac{1}{|\xi\xi_2|^{1-\mu}|\xi\Phi|^{\mu/2}}d\Phi\lesssim \left(\frac{M}{|\xi|}\right)^{1-\mu/2}.
		\end{align*}
		Interpolating the above estimates, we conclude that
		\begin{equation}
			\int |\xi_2|^{1-\mu}\mathbbm{1}_{|\Phi-\alpha|<M}d\xi_2 \lesssim \left(\frac{M}{|\xi|\jap{\xi}^\mu}\right)^{1/3}\left(\frac{M}{|\xi|} \right)^{2/3-\mu/3}.
		\end{equation}
		Therefore,
		\begin{align*}
			I_{52}^\pm&:=\iint_{D_{52}^\pm} \m 1_{|\xi_1| \le |\xi_2| \le |\xi_3|}  \m 1_{|\Phi-\alpha| \le M} \frac{\jap{\xi}^{\mu}|\xi_3|}{\jap{\xi_1}^\mu \jap{\xi_2}^\mu} d\xi_1 d\xi_2 \lesssim \iint \frac{\jap{\xi}^\mu}{\jap{\xi_1}^\mu}|\xi_2|^{1-\mu}\mathbbm{1}_{|\Phi-\alpha|<M}d\xi_2 d\xi_1 \\
			& \lesssim \left(\int_{|\xi_1|\le d |\xi|} \frac{\jap{\xi}^\mu}{\jap{\xi_1}^\mu} d\xi_1\right)\left(\frac{M}{|\xi|\jap{\xi}^\mu}\right)^{1/3}\left(\frac{M}{|\xi|} \right)^{2/3-\mu/3} \lesssim M^{1-\mu/3}. \qedhere
		\end{align*}
	\end{proof}
	\begin{prop}\label{prop:fre_psi}
		For $M \ge 1$,
		\begin{equation}\label{eq:BalphaMmultiplier}
			\sup_{\xi,\alpha}\int \max\left(|\xi|,|\eta|\right)\frac{\jap{\xi}^{\mu}}{\sqrt{\eta}\jap{\xi-\eta}^\mu}\mathbbm{1}_{|\Psi-\alpha|<M}d\eta \lesssim \sqrt{M}.
		\end{equation}
	\end{prop}
	
	
	\begin{proof}

		Before we proceed, it is useful to normalize the frequencies. Write $p=\eta/\xi$ and $\psi(p)=-\frac{3}{4}p(p-2)^2$, so that
		$$
		\Psi=\xi^3\psi(p),\quad \frac{\partial \psi}{\partial p} = -\frac{3}{4}(p-2)(3p-2),\quad \frac{\partial^2 \psi}{\partial p^2} = -\frac{9}{2}p+6.
		$$
		In particular, the stationary points $p=2,2/3$ are nondegenerate.
		We assume, without loss of generality, that $\xi>0$ and split the integration into the regions near the stationary points of $\Psi$ ($\eta=2\xi,2\xi/3$), the singular points ($\eta=0,\xi$) and the remaining domain. 
		%
		%
		%
		%
		
		\bigskip
		
		\emph{Case 1. Near the stationary points $2\xi$ and $2\xi/3$.}
		
		Let us focus on the case $\eta\simeq 2\xi$.
		Since $p=2$ is a nondegenerate critical point of $\psi$, there exist $c>0$ small, a segment $D_1 \subset [-1,1]$ and a $\q C^1$ diffeomorphism $\varphi_1: D_1 \to [2-c,2+c]$ such that for all $q \in D_1$, $\psi(\varphi_1(q)) = -q^2$. Given $\eta \in [(2-c) \xi, (2+c) \xi]$, we have
		\[ \max\left(|\xi|,|\eta|\right)\frac{\jap{\xi}^{\mu}}{\sqrt{\eta}\jap{\xi-\eta}^\mu} \lesssim |\xi|^{1/2}. \]
		Therefore, using \eqref{eq:quadratica}
		\begin{align*}
			J_1 & : =  \int_{[(2-c) \xi, (2+c) \xi]} \max\left(|\xi|,|\eta|\right)\frac{\jap{\xi}^{\mu}}{\sqrt{\eta}\jap{\xi-\eta}^\mu} \m{1}_{|\Psi-\alpha|<M}d\eta  \lesssim |\xi|^{1/2} \int_{[(2-c) \xi, (2+c) \xi]} \m{1}_{|\Psi-\alpha|<M} d\eta \\
			& \lesssim |\xi|^{3/2} \int_{2-c}^{2+c} \m 1_{|\xi^3 \psi(p) - \alpha| < M} dp  \lesssim |\xi|^{3/2} \int_{-1}^1 \m 1_{|\xi^3 q^2  + \alpha| < M} dq \lesssim \sqrt{M}.
		\end{align*}
		If $\eta\simeq 2\xi/3$, for $c>0$ small, there exists a segment $D_2 \subset [-1,1]$ and a $\q C^1$ diffeomorphism $\varphi_2: D_2 \to [2/3-c,2/3+c]$ such that for all $q \in D_2$, $\psi(\varphi_2(q)) = q^2$. The estimate in the interval $[2/3-c,2/3+c]$ follows from computations analogous to the $\eta\simeq 2\xi$ case.

		\bigskip
		
		\emph{Case 2. Near the singularity at $0$.}
		
		In this region, we use the fact that $\psi$ is a diffeomorphism on $[-2/3+c,2/3-c]$ onto its image, with $\psi'(p) \ge c$ on this interval. In particular, as $\psi(0)=0$, we also have $|\psi(p)| \lesssim |p|$. Furthermore, for $|\eta| \le (2/3-c)|\xi|$, $\jap{\xi-\eta} \gtrsim |\xi|$. Performing the change of variables $p\mapsto \psi$ and using \eqref{eq:singular},
		\begin{align*}
			J_3 & : =  \int_{|\eta| \le (2/3-c)|\xi|} \max\left(|\xi|,|\eta|\right)\frac{\jap{\xi}^{\mu}}{\sqrt{\eta}\jap{\xi-\eta}^\mu}\m{1}_{|\Psi-\alpha|<M}d\eta  \lesssim |\xi| \int_{|\eta| \le (2/3-c)|\xi|} \m{1}_{|\Psi-\alpha|<M} \frac{d\eta}{\sqrt{|\eta}|} \\
			& \lesssim |\xi|^{3/2} \int_{|p| \le 2/3-c} \m 1_{|\xi^3 \psi(p) - \alpha| < M} \frac{dp}{\sqrt{|p|}} \lesssim |\xi|^{3/2} \int \m 1_{\left| \psi  - \frac{\alpha}{\xi^3} \right| < \frac{M}{|\xi|^3}} \frac{d\psi}{\sqrt{|\psi|}} \lesssim  |\xi|^{3/2} \left(\frac{M}{|\xi|^3} \right)^{1/2} \\
			& \lesssim \sqrt{M}.
		\end{align*}
		
		\bigskip
		
		\emph{Case 3. Near $\xi$.}
		
		As for the previous case, $\psi$ is a diffeomorphism on $[2/3+c,2-c]$ onto its image, with $\psi'(p) \ge c$ on this interval. By Cauchy-Schwarz inequality,
		\begin{align*}
			J_4 & : =  \int_{[(2/3+c)\xi,(2-c)\xi]} \max\left(|\xi|,|\eta|\right)\frac{\jap{\xi}^{\mu}}{\sqrt{\eta}\jap{\xi-\eta}^\mu}\m{1}_{|\Psi-\alpha|<M}d\eta \\
			& \le \left( \int_{[(2/3+c)\xi,(2-c)\xi]}  |\xi|^2 \m{1}_{|\Psi-\alpha|<M} d\eta  \right)^{1/2} \left( \int_{[(2/3+c)\xi,(2-c)\xi]} \frac{|\xi|^{2\mu-1}}{\jap{\xi-\eta}^{2\mu}} d\eta \right)^{1/2}.
		\end{align*}
		Now, with $\zeta = \xi -\eta$, we bound the second integral as
		\[ \int_{|\eta -\xi| \le c \xi} \frac{|\xi|^{2\mu-1}}{\jap{\xi-\eta}^{2\mu}} d\eta =|\xi|^{2\mu-1} \int_{|\zeta| \le 2|\xi|} \frac{d\zeta}{\jap{\zeta}^{2\mu}}  = O(1). \]
		On the other hand, by \eqref{eq:singular},
		\begin{align*}
			\MoveEqLeft \int_{[(2/3+c)\xi,(2-c)\xi]}  |\xi|^2 \m{1}_{|\Psi-\alpha|<M} d\eta = |\xi|^3 \int_{2/3+c}^{2-c} \m 1_{|\xi^3 \psi(p) - \alpha| \le M} dp \\
			& \lesssim |\xi|^3 \int_{} \m 1_{\left|\psi - \frac{\alpha}{\xi^3} \right| \le \frac{M}{|\xi|^3}} d\psi  \lesssim M.
		\end{align*}
		Gathering these two estimates together yields
		\[ J_4 \lesssim \sqrt M. \]
		
		\bigskip
		
		\emph{Case 4. In the remaining region.}
		
		We conclude with the region  
		\[ 
		R_4= \left]-\infty, (-2/3 + c)\xi \right[ \cup \left](2+c)\xi,\infty \right[.
		\]
		In $R_4$, $|p|^3 \sim |\psi(p)|$, $|\psi'(p)| \gtrsim p^2$. and $\jap{\xi - \eta} \gtrsim |\eta| \gtrsim |\xi|$. In particular,
		\[
		\max\left(|\xi|,|\eta|\right)\frac{\jap{\xi}^\mu}{\sqrt{|\eta|}\jap{\xi-\eta}^\mu}\lesssim |\eta|^{1/2}.
		\]
		Hence
		\begin{align*}
			J_5 &:= \int_{R_5}  \max\left(|\xi|,|\eta|\right)\frac{\jap{\xi}^{\mu}}{\sqrt{|\eta|} \jap{\xi-\eta}^\mu} \m{1}_{|\Psi-\alpha|<M}d\eta \lesssim \int_{R_5}  |\eta|^{1/2}\m{1}_{|\Psi-\alpha|<M}d\eta \\
			& \lesssim |\xi|^{3/2} \int_{|p-1| \ge 1+c}  \m{1}_{|\xi^3 \psi(p)-\alpha|<M} |p|^{1/2}dp.
		\end{align*}
		We now perform the change of variable $p \to \psi$,
		\begin{align*}
			J_5 \lesssim  |\xi|^{3/2} \int  \m{1}_{|\xi^3 \psi-\alpha|<M} \frac{d\psi}{|p|^{3/2}} \lesssim  |\xi|^{3/2} \int  \m{1}_{|\xi^3 \psi-\alpha|<M} \frac{d\psi}{|\psi|^{1/2}}\lesssim \sqrt{M}.  
		\end{align*}
	\end{proof}
	
	Propositions \ref{prop:fre_phi} and \ref{prop:fre_psi} are sufficient to control the nonlinear terms involving $w$. These results could also be applied to the source terms. However, to achieve the polynomial bound in time for $w$, we need a refined control in time on the source terms (see \eqref{eq:decompF1}). The next result will be used to control the term $F_{11}[S,z]$.
	
	\begin{prop}\label{prop:estPhi_source}
		Recall $\gamma$ was set in \eqref{eq:defieta} and let $\beta$ be such that $1-(\nu-\mu-3\gamma)/2<\beta<1$.
		Then, for  $M \ge 1$,
		\begin{equation}\label{eq:TalphaM_D}
			\sup_{\xi,\alpha}\ \jap{\xi}^\mu \int_{D_3} \left(\max_{j=1,2,3}|\xi_j|\right)\fia \frac{d\xi_1 d\xi_2}{\jap{\xi_1}^\nu \jap{\xi_2}^\nu } \lesssim M^{\beta}\tau^\gamma,
		\end{equation}
	\end{prop}
	
	\begin{proof}
		We split into several cases. If either $|\xi_1|\ge \tau^{-1/3}/100$ or $|\xi_2| \ge \tau^{-1/3}/100$, then
		\begin{multline*}
			\jap{\xi}^\mu \int_{D_3} \left(\max_{j=1,2,3}|\xi_j|\right)\fia \frac{d\xi_1 d\xi_2}{\jap{\xi_1}^\nu \jap{\xi_2}^\nu } \\
		 \lesssim \left(\jap{\xi}^\mu \int_{D_3} \left(\max_{j=1,2,3}|\xi_j|\right)\fia \frac{d\xi_1 d\xi_2}{\jap{\xi_1}^\mu\jap{\xi_2}^\mu}\right) \tau^{\frac{\nu-\mu}{3}},
		\end{multline*}
		and the proof follows from Proposition \ref{prop:fre_phi}. Otherwise, $|\xi_1|,|\xi_2|\le \tau^{-1/3}/100$ and the definition of $D_3$ implies that $ |\xi_3| \ge 9 \tau^{-1/3}/100$ and $|\xi| \ge 7 |\xi_1|$. Then $|\partial_{\xi_1}\Phi|\simeq |\xi|^2$. Setting  $1/p=(\nu-\mu-3\gamma)/2$, we have $\nu p>1$ and, by applying Hölder in $\xi_1$,
		\begin{align*}
			|\xi|\jap{\xi}^\mu \int \fia \frac{1}{\jap{\xi_1}^\nu \jap{\xi_2}^\nu }d\xi_1 d\xi_2  &\lesssim 	|\xi|\jap{\xi}^\mu \int \left(\int \fia d\xi_1 \right)^{\frac{1}{p'}} \left(\int \frac{1}{\jap{\xi_1}^{\nu p}}d\xi_1\right)^{\frac{1}{p}}\frac{1}{\jap{\xi_2}^\nu }d\xi_2\\ 
			&\lesssim	\jap{\xi}^{2-\nu+\mu} \left(\int \fia \frac{1}{\xi^2}d\Phi \right)^{\frac{1}{p'}} \lesssim M^{\beta}|\xi|^{-3\gamma}\lesssim M^{\beta} \tau^\gamma. \qedhere
		\end{align*}
	\end{proof}
	%
	%
	We now prove the necessary frequency-restricted estimate in order to handle the $F_{12}[S,S_{\reg},z]$ term.
	\begin{prop}\label{lem:estSregz}
		Fix $t,\tau\in (0,1)$  such that $t \ge \tau/10$ and $M>0$. Then
		\[
		\sup_{\xi,\alpha}\ \jap{\xi}^\mu \int_{D_1} \left(\max_{j=1,2,3}|\xi_j|\right)\fia \frac{d\xi_1 d\xi_2}{\jap{t^{1/3}\xi_2}^{(4/7)^-}\jap{\xi_3}^\nu } \lesssim \left[ (tM)^{1-\mu/3} +  (tM)^{1-\nu/3} \right] t^{-1+\frac{\nu-\mu}{3}}.
		\]
	\end{prop}
	\begin{proof} Observe that, compared to \eqref{eq:TalphaM_multiplier}, the weights in the denominator are slightly stronger. Thus, for large frequencies, we can convert the extra weights directly into powers of $t$ and then apply Proposition \ref{prop:fre_phi}. Indeed, 
		if $|\xi_3| \ge t^{-1/3}/300$,
		\begin{align*}
			\MoveEqLeft \jap{\xi}^\mu \int_{D_1} \left(\max_{j=1,2,3}|\xi_j|\right) \fia \frac{d\xi_1 d\xi_2}{\jap{t^{1/3}\xi_2}^{(4/7)^-}\jap{\xi_3}^\nu } \\
			& \lesssim \jap{\xi}^\mu \int \left(\max_{j=1,2,3}|\xi_j|\right) \fia \frac{d\xi_1 d\xi_2}{\jap{\xi_2}^\mu t^{\mu/3}\jap{\xi_3}^{\mu}} t^{\frac{\nu-\mu}{3}}\\ 
			& \lesssim M^{1-\mu/3}t^{\frac{\nu-2\mu}{3}}\lesssim (tM)^{1-\mu/3}t^{-1+\frac{\nu-\mu}{3}}.
		\end{align*}
		Second, if $|\xi| \ge t^{-1/3}/300$,
		\begin{align*}
			\MoveEqLeft \jap{\xi}^\mu \int_{D_1} \left(\max_{j=1,2,3}|\xi_j|\right)\fia \frac{d\xi_1 d\xi_2}{\jap{t^{1/3}\xi_2}^{(4/7)^-}\jap{\xi_3}^\nu }\\
			& \lesssim \jap{\xi}^\nu  \int  \left(\max_{j=1,2,3}|\xi_j|\right)\fia \frac{d\xi_1 d\xi_2}{\jap{\xi_2}^\nu t^{\nu/3}\jap{\xi_3}^{\nu}} t^{\frac{\nu-\mu}{3}}\\ 
			& \lesssim M^{1-\nu/3}t^{-\frac{\mu}{3}} \lesssim (tM)^{1-\nu/3}t^{-1+\frac{\nu-\mu}{3}} .
		\end{align*}
		Otherwise, $|\xi|,|\xi_3|\le t^{-1/3}/300$, then $|\xi_1+\xi_2+\xi_3|+|\xi_1+\xi_2|\le  t^{-1/3}/100 < \tau^{-1/3}/10$, but this is not possible in $D_1$.
		
	\end{proof}
	
	Finally, we prove a frequency-restricted estimate for the last source term, $L_{K,D(\tau)}[z]$. It looks very similar to the bound \eqref{eq:BalphaMmultiplier} in Proposition \ref{prop:fre_psi}, with a improvement in $\tau$ when the intregal is over $D(\tau)$ only.
	
	\begin{prop}\label{prop:sourceSz}
		For $M \ge 1$,
		\begin{equation}\label{eq:BalphaM_D}
			\sup_{\xi,\alpha} \int_{D(\tau)} \max\left(|\xi|,|\eta|\right)\frac{\jap{\xi}^{\mu}}{\sqrt{\eta}\jap{\xi-\eta}^\nu }\psia d\eta \lesssim \sqrt{M}\tau^{\frac{\nu-\mu}{3}}.
		\end{equation}
	\end{prop}
	\begin{proof}
		
		If $|\xi-\eta|\gtrsim \tau^{-1/3}/100$, then 
		\begin{multline*}
			\int_{D(\tau)} \max\left(|\xi|,|\eta|\right)\frac{\jap{\xi}^{\mu}}{\sqrt{\eta}\jap{\xi-\eta}^\nu }\psia d\eta \\
			\lesssim	\left(\int_{D(\tau)} \max\left(|\xi|,|\eta|\right)\frac{\jap{\xi}^{\mu}}{\sqrt{\eta}\jap{\xi-\eta}^\mu}\psia d\eta\right) \tau^{\frac{\nu-\mu}{3}},
		\end{multline*}
		and the estimate follows from Proposition \ref{prop:fre_psi}. If $|\xi-\eta|\le \tau^{-1/3}/100$, the conditions on $D(\tau)$ imply $|\xi| \simeq |\eta|\gtrsim \tau^{-1/3}$. Proceeding as in Case 3 of the proof of Proposition \ref{prop:fre_psi}, 
		\begin{align*}
			\int_{\eta\simeq \xi} \m{1}_{|\Psi-\alpha|<M}\frac{\jap{\xi}^{\mu+1}}{\sqrt{|\eta|} \jap{\xi-\eta}^\nu }d\eta  \lesssim & \left( \int_{\eta\simeq \xi}  |\xi|^2 \m{1}_{|\Psi-\alpha|<M} d\eta  \right)^{1/2} \left( \int_{\eta\sim \xi} \frac{|\xi|^{2\mu-1}}{\jap{\xi-\eta}^{2\nu}} d\eta \right)^{1/2}\\ \lesssim & \left( \int  \m{1}_{|\Psi-\alpha|<M} d\Psi  \right)^{1/2} \jap{\xi}^{\mu-\nu} \lesssim \sqrt{M}\tau^{\frac{\nu-\mu}{3}}. \qedhere
		\end{align*}
		
	\end{proof}

	Finally, we notice that frequency restricted estimate with constraint of the form $|\Theta - \alpha| < M$  (where $\Theta$ is a phase function), as above, can help derive a frequency restricted estimate with constraint of the form $|\Theta - \alpha| > M$, via the use of a dyadic decomposition on the phase.
	
	\begin{lem} \label{lem:restricted_<>}
		Assume that one has the bound, for $M \ge 1$,
		\[ \sup_{\xi,\alpha} \int_{\Gamma_\xi} m(\xi, \Xi) \m 1_{|\Theta(\xi,\Xi) -\alpha| < M} d\Xi \le C_0 M^\theta, \]
		for some $\theta,C_0 \ge 0$, where $\Theta$ is a phase, $m \ge 0$ is a multiplier and $\Gamma_\xi$ is a domain of integration in $\m R^N$ which may depend on $\xi$.
		
		Let $\rho > \theta$. Then there exists $C = C(\rho,\theta )$ such that, for any $M \ge 1$,
		\[  \sup_{\xi,\alpha} \int_{\Gamma_\xi} \frac{m(\xi, \Xi)}{|\Theta(\xi,\Xi)-\alpha|^\rho} \m 1_{|\Theta(\xi,\Xi) -\alpha| > M} d\Xi \le \frac{C C_0}{M^{\rho- \theta}}. \]
	\end{lem}
	
	\begin{proof}
		We decompose dyadically\footnote{We write ``$M' >M$ is dyadic'' as short for ``There exists $n \in \m N$ such that $M'= 2^n M$''.} depending on the size of $|\Theta-\alpha|$:
		\begin{align*}
			\sup_{\xi,\alpha} \int_\Gamma \frac{m(\xi, \Xi)}{|\Theta-\alpha|^\rho} \m 1_{|\Theta -\alpha| > M} d\Xi
			& \le \sup_{\xi,\alpha} \sum_{\substack{M' > M \\ M' \text{dyadic}}} \int_\Gamma \frac{m(\xi, \Xi)}{|\Theta-\alpha|^\rho} \m 1_{M'/2 \le  |\Theta -\alpha| < M'} d\Xi \\
			& \le 2^\rho  \sum_{\substack{M' > M \\ M' \text{dyadic}}} \sup_{\xi,\alpha} \frac{1}{{M'}^\rho} \int_\Gamma m(\xi, \Xi) \m 1_{|\Theta -\alpha| < M'} d\Xi \\
			& \le  2^\rho C_0 \sum_{\substack{M' > M \\ M' \text{dyadic}}} {M'}^{\theta - \rho} \le C C_0 M^{\theta-\rho}. \qedhere
		\end{align*} 
		
	\end{proof}
	
	\subsection{Pointwise bounds for the nonlinearity} As it will be seen in Section \ref{sec:INFR}, we will perform a bootstrap argument involving $w, \partial_\xi w$ and $\partial_t w$. The infinite normal form reduction will provide \textit{a priori} bounds for 
	\[
	w\quad \text{ and }\quad \partial_\xi w - \frac{3t}{\xi}\partial_t w.
	\]
	In order to close the bootstrap, we need an estimate for $\frac{3t}{\xi}\partial_t w$, which amounts to a pointwise estimate of the nonlinearity.
	
	\begin{lem}\label{lem:controldt}
		Given $0<a<1$ and $0<a'\le b'<1$, set
		\[
		\theta=\frac{1}{3}\min( a, b'-a').
		\]
		Then, for $t\in(0,1)$,
		\begin{equation}
			\left\|\frac{t}{\xi}N[f,g,h] (t)\right\|_{L^\infty_{a'}}\lesssim t^{\theta}\| f \|_{W^{1,\infty}_{0,a'}(\R\setminus\{0\})} \| g \|_{W^{1,\infty}_{0,a'}(\R\setminus\{0\})} \| h \|_{W^{1,\infty}_{a,b'}(\R\setminus\{0\})}.\label{eq:deriv_t_2}
		\end{equation}
	\end{lem}

	\begin{proof}
		Throughout the proof, we omit the spatial domain $\R\setminus\{0\}$ in the norms, keeping in mind that integrations by parts in frequency will produce boundary terms at the zero frequency. Moreover, assume, from now on, unit norms for $f,g$ and $h$. Recall that
		$$
		\frac{1}{\xi}N[f,g,h] = \int e^{it\Phi}{f}(\xi_1){g}(\xi_2){h}(\xi_3)d\xi_1 d\xi_2=:I.
		$$
		Without loss of generality, we suppose that $|\xi_1|\ge |\xi_2|\ge |\xi_3|$. In the domain where all frequencies are smaller than $t^{-1/3}$, the corresponding integral $I_1$ is bounded directly:
		
		\begin{align*}
			|I_1|&\lesssim \left(\int \frac{1}{\jap{\xi_3}^a}d\xi_1d\xi_2 \right) \| f \|_{W^{1,\infty}_{0,a'}} \| g \|_{W^{1,\infty}_{0,a'}} \| h \|_{W^{1,\infty}_{a,b'}}\lesssim \frac{1}{t^{(2-a)/3}}\lesssim \frac{1}{t^{1-a/3}\jap{\xi}}.
		\end{align*}
		We henceforth consider $|\xi_1|\ge t^{-1/3}$. In the domain where $|\xi_1|\ge 10 |\xi|$, then  $|\xi_2|\simeq |\xi_1|\ge t^{-1/3}$ and $|\partial_{\xi_1}\Phi|\gtrsim |\xi_1|^2$. The integration by parts in $\xi_1$ in this domain $D_2$ gives
		\begin{multline}\label{eq:partes13}
			I_2 =\int_{D_2} e^{it\Phi} \partial_{\xi_1}\left( \frac{1}{it\partial_{\xi_1}\Phi} {f}(\xi_1){g}(\xi_2){h}(\xi_3) \right)d\xi_1d\xi_2 \\
			 + \int_{(\xi_1,\xi_2)\in \partial D_2} \left( e^{it\Phi}  \frac{1}{it\partial_{\xi_1}\Phi} {f}(\xi_1){g}(\xi_2){h}(\xi_3)\right)\Big|_{\xi_3=0^\pm}d\xi_2
		\end{multline}
		and
		\begin{align*}
			|I_2| & \lesssim \left(\int_{D_2} \left(\frac{1}{t|\xi_1|^3} + \frac{1}{t\xi_1^2}\left(\frac{1}{\jap{\xi_1}^{a'}\jap{\xi_3}^a}+\frac{1}{\jap{\xi_3}^{b'}}\right) \right)d\xi_1d\xi_2 + \int_{|\xi_2| \ge t^{-1/3}/10} \frac{d\xi_2}{t\xi_2^2} \right) \\
			& \hspace{90mm} \cdot \| f \|_{W^{1,\infty}_{0,a'}} \|g \|_{W^{1,\infty}_{0,a'}} \|h \|_{W^{1,\infty}_{a,b'}} \\&\lesssim \frac{1}{t}\left(\frac{1}{ \jap{\xi}^{a+a'}}+\frac{1}{ \jap{\xi}^{b'}} + \frac{1}{t^{2/3}}\right)\lesssim \left(\frac{1}{t^{1-(b'-a')/3}}+\frac{1}{t^{1-a/3} }\right)\frac{1}{\jap{\xi}^{a'}}.
		\end{align*}
		In the domain $|\xi_1|\lesssim |\xi|$, we decompose in the stationary and nonstationary cases. In the domain $D_3$ where furthermore $|\partial_{\xi_1}\Phi|\simeq |\xi_1^2-\xi_3^2|\ge|\xi|^2/100$, we perform the same integration by parts as in \eqref{eq:partes13}. We bound the corresponding integral $I_3$ by
		\begin{align*}
			|I_3 |\lesssim & \left(\int_{D_3} \left(\frac{1}{t|\xi|^3} + \frac{1}{t\xi^2}\left(\frac{1}{\jap{\xi_1}^{a'}\jap{\xi_3}^a} + \frac{1}{ \jap{\xi_3}^{b'}}\right) \right)d\xi_1d\xi_2 + \int_{(\xi_1,\xi_2) \in \partial D_3} \frac{d\xi_2}{t\xi^2} \right) \\
			& \hspace{90mm} \cdot \| f \|_{W^{1,\infty}_{0,a'}} \|g \|_{W^{1,\infty}_{0,a'}} \|h \|_{W^{1,\infty}_{a,b'}}\\ \lesssim& \frac{1}{t}\left(\frac{1}{\jap{\xi}^{a+a'}}+\frac{1}{\jap{\xi}^{b'}}+\frac{1}{t^{2/3}}\right) \lesssim \left(\frac{1}{t^{1-a/3}}+\frac{1}{t^{1-(b'-a')/3}}\right)\frac{1}{\jap{\xi}^{a'}}.
		\end{align*}
		By symmetry, these computations can also be applied to the case $|\partial_{\xi_2}\Phi|\ge |\xi|^2/100$. We are left with the region $|\xi_1^2-\xi_3^2|, |\xi_2^2-\xi_3^2|\le |\xi|^2/100$. 
		 Setting $p_j=\xi_j/\xi$ and $P=1-p_1^3-p_2^3-p_3^3$, this implies that $(p_1,p_2,p_3)$ is near a stationary point for $P$,
		\[
		\left(\frac{1}{3}, \frac{1}{3}, \frac{1}{3}\right),\quad  (1,1,-1), \quad  (1,-1,1), \quad  (-1,1,1).
		\]
		We consider only the first two possibilities. After a change of variables, we must bound
		\[
		\xi^2\int e^{it\xi^3P}f(\xi p_1)g(\xi p_2)h(\xi p_3)dp_1dp_2.
		\]
		In the region $|p_j - 1/3| \le 1/10$ for $j=1,2,3$, and (without loss of generality) $|p_1-p_3|\ge |p_2-p_3|$, the corresponding integral writes
		\begin{align*}
		I_4=\xi^2\int_{D_4} \partial_{p_1}\left(e^{it\xi^3 P}(p_1-p_3)\right)\frac{1}{2+it\xi^3(p_1-p_3)\partial_{p_1}P}f(\xi p_1)g(\xi p_2)h(\xi p_3) dp_1dp_2.
		\end{align*}
		Writing $q_j=p_j-p_3$, $j=1,2$, the integration by parts in $p_1$ yields
		\begin{align}
			|I_4| & \lesssim \xi^2\int_{|q_2|\le |q_1|\le 1} \left(\frac{t|\xi|^3q_1^2|\xi|^{-a}}{|1+it\xi^3q_1^2|^2} + \frac{|q_1|(|\xi|^{1-a'-a}+|\xi|^{1-b'})}{|1+it\xi^3q_1^2|}\right)dq_1dq_2 + \xi^2\int \frac{|\xi|^{-a}}{|1+it\xi^3|}dq_2  \label{eq:asympt1/t} \\&\lesssim \xi^2\int_{ |q_1|\le 1} \left(\frac{t|\xi|^{3-a}|q_1|^3}{|1+it\xi^3q_1^2|^2} + \frac{|q_1|^2(|\xi|^{1-a'-a}+|\xi|^{1-b'})}{|1+it\xi^3q_1^2|}\right)dq_1 + \frac{1}{t|\xi|^{1+a}} \\&\lesssim 
			\frac{1}{t|\xi|^{1+a}}\int_{|z_1|\le t^{1/3}|\xi|^{3/2}} \frac{|z_1|^3}{|1+iz_1^2|^2} dz_1 + \frac{1}{t|\xi|^{a+a'}} + \frac{1}{t|\xi|^{b'}} + \frac{1}{t|\xi|^{1+a}}
			\\&\lesssim \frac{\log(t\xi^3)}{t|\xi|^{1+a}} + \frac{1}{t|\xi|^{a+a'}} + \frac{1}{t|\xi|^{b'}} \lesssim \frac{1}{t^{1-\theta}\jap{\xi}^{a'}}
		\end{align}
		In the region where $|p_j - 1| \le 1/10$ for $j=1,2$, define $q_j=p_j+p_3$, $j=1,2$. We can assume without loss of generality that $|q_2|\le |q_1| \le 1$, and write for the corresponding integral
		\[
			I_5 =\xi^2\int_{D_5} \partial_{p_1}\left(e^{it\xi^3 P}\right)\frac{1}{it\xi^3\partial_{p_1}P}f(\xi p_1)g(\xi p_2)h(\xi p_3)dp_1dp_2
		\]
		Since $\partial_{p_1}P \simeq \partial^2_{p_1}P \simeq q_1$, the integration by parts in $p_1$ gives
		\begin{align*}
			|I_5 |&\lesssim \xi^2\left|\int_{D_5} e^{it\xi^3 P}\partial_{p_1}\left(\frac{1}{it\xi^3\partial_{p_1}P}f(\xi p_1)g(\xi p_2)h(\xi p_3)\right)dp_1dp_2\right| + \xi^2\int_{|p_2|\le 2} \frac{|\xi|^{-a}}{t|\xi|^3} dp_2 \\
			&\lesssim \xi^2\int_{|q_2|\le |q_1|\le 1} \frac{1}{t\xi^3q_1}(|\xi|^{-a}+ |\xi|^{1-a-a'}+ |\xi|^{1-b'})dq_1dq_2 + \frac{1}{t|\xi|^{1+a}} \\
			& \lesssim \frac{1}{t|\xi|^{1+a}} + \frac{1}{t|\xi|^{a+a'}} + \frac{1}{t|\xi|^{b'}}\lesssim \frac{1}{t^{1-\theta}\jap{\xi}^{a'}}. 
		\end{align*}
		Summing up the bounds for $I_1, \dots, I_5$, we obtain the desired estimate for $I$.
	\end{proof}
	\begin{nb}
		In the above proof, the analysis can be performed with a single integration by parts due to the subcriticality of the norms in  \eqref{eq:deriv_t_2} (in the critical case, two integrations are necessary, see \cite[Lemma 9]{CC22}). The gain in powers of $t$ comes from an extra decay either in the functions involved or in their derivatives.
	\end{nb}
	

	\section{The infinite normal form reduction}\label{sec:INFR}

	We now detail how to control  
	\[ w_n(t) = \chi_n^2 \int_0^{t} (N[S + \chi_n z + w_n] - N[S])  ds\]
	uniformly in $n$. We use the decomposition \eqref{eq:decomp_nonli}, in which we distinguish the ``well-behaved'' terms which can be bounded suitably using the multilinear estimates of the previous section, and the other ``bad-behaved '' terms, which require an extra reduction. For the latter, the idea it to perform an integration by parts in time. When the time derivative falls on the nonlinear terms, many new terms appear: to be able to correctly treat them, we'll associate to each of them a tree. The nodes and leafs of the tree will have a color corresponding to one of $S,z,w$ etc. (this will be made precise below). We actually need some extra information, which leads to consider four different types of trees:
	\begin{itemize}
	    \item resonant, where the integration by parts in time does not improve the estimate;
	    \item boundary, corresponding to the boundary terms after an integration by parts;
	    \item derivative, when the time derivative falls onto $S$;
	    \item $F_2$, when the time derivative falls on $w$ and we replace it by the source term $F_2$ (see \eqref{eq:decompF2}).
	\end{itemize}
	Given a colored admissible tree of a given type, we can provide an estimate of the associated term in a systematic way, which is actually rather automatic: we essentially use pointwise estimates. This is a very agreable feature coming from working in weighted $L^\infty$ spaces.
	
	\bigskip
	
	Observe that, as $\chi_n$ does not depend on time (and $0 \le \chi_n \le 1$), it is not affected by the integration by parts in times, so that it plays essentially no role in the construction of the tree nor in the estimates. For simplicity of the exposition, we therefore choose to drop all $n$ and $\chi_n$ in the section: equivalently, this corresponds to derive \textit{a priori} estimates on a solution to
	\begin{equation}\label{eq:integralw}
	    	 w(t)  = \int_0^t N[S+z+w] - N[S] dt,
	\end{equation}
	and it will be clear when doing the computations that all bounds also hold for $w_n$, with the same constants (independent of $n$).
	
	\bigskip

	In view of the frequency-restricted estimates of the previous section, at each step $j$, the integration occurs on domains with some conditions on the phase. We therefore give ourselves a sequence of frequency thresholds $(N_j)_j$, to be precised later.
	
	\bigskip
	
	As it will be clear, we will need to incorporate a dependence on time in the normal form reduction itself, ideally by taking $N_j=N_j(t)$. However, as we perform integrations by parts in time, this dependence would yield new delicate terms. 
	
	Instead, from now on, we choose $T>0$ small (to be fixed later) and we work separately on the dyadic time intervals $I_k=[t_{k+1},t_k]=[T/2^{k+1}, T/2^k]$ ($k \in \m N$). This will allow us to freely interchange $t\in I_k$ with $t_{k+1}$ (it plays the role of the variable $\tau$ of the previous section). For each $k$, we perform the infinite reduction to the equation starting at $t=t_{k+1}$:
	\begin{equation} \label{eq:w_tk}
		w(t)=w(t_{k+1}) + \int_{t_{k+1}}^t N[S(t')+z+w(t')]-N[S(t')]dt',
	\end{equation}
	and our goal is to obtain estimates for $t \in [t_{k+1}, t_k]$. Actually the admissible trees do not depend on $k$ or the $(N_j)_j$, only the estimates do, and we will of course track the dependence in $k$ and $(N_j)_j$.

	In the next paragraph, we describe the algebraic procedure corresponding to the infinite normal form reduction. The precise analytic structure and the derivation of \textit{a priori} bounds will be discussed in the following subsection, starting with an example.
	
	\subsection{The normal form algorithm}\label{sec:INFR_algorithm} 
	
	In this procedure, $k \in \m N$ is fixed, and we recall that we are given a set of frequencies $(N_j)_j$. We also give ourselves a time $t \in [t_{k+1},t_k]$.
	
	In the equation \eqref{eq:w_tk} for $w$, at each iteration of the algorithm, the right-hand side is made of an (increasing number of) terms. To go to the next iteration, we split between well-behaved terms (for which the tree building proceduce stops), and bad-behaved terms, for which we will increase the number of factors (and so, the size of the tree).
	
	All terms will have have the form
	\begin{equation}\label{eq:genericform}
		\int_{t_{k+1}}^t \int_\Gamma e^{is\Theta}m\prod_{l=1}^{l_0}f_l(s,\xi_l) d\Xi ds,
	\end{equation}
	(except for boundary terms for which there is no time integration) where $\Theta$ is a phase function, $\Gamma$ is a convolution surface, $m$ is a space-frequency multiplier, $l_0$ is the degree of the nonlinear term and $f_l$ are one of the \textbf{intervening functions} $w,z,S_0,S_{\reg}$, $K_0$, where
	\[
	K_0(t,\xi):=\frac{1}{t^{1/3}}K(S_0,S_0)(t^{1/3}\xi). 
	\]
	($K_0$ is required to make retain the special cancellations specific to the self-similar solution), the derivatives $\partial_t S_0$, $\partial_t S_{\reg}$, $\partial_t K_0$ ($z$ does not depend on time, and we replace $\partial_t w$ using the equation, which increments the step), or $F_2$ (which cannot be treated by INFR, but fortunately is well-behaved). We will give two equivalent descriptions: first, an algorithm to produce the terms of the INFR; second, the procedure to construct colored trees (which can help to visualize the algorithm, and is similar to the description given in \cite{KOY20}).

	The algorithm is as follows. In the first step $J=0$, the well-behaved term is just $F_2$, and the bad behaved terms are the other ones appearing in \eqref{eq:decomp_nonli}.
	To construct the terms at step $J \ge 1$,  for each of the badly-behaved terms at step $J-1$, which are of the type \eqref{eq:genericform}, we perform the following \emph{term-loop}:
	\begin{enumerate}
		\item We split the frequency domain into the regions $|\Theta|<N_J$ and $|\Theta|\ge N_J$. The first corresponds to a \textit{resonant term} and is a well-behaved term.
		\item For the nonresonant term, we use the relation $\ds \partial_t(e^{it\Theta})= \frac{e^{it\Theta}}{i\Theta}$ to integrate by parts in time: this generates one boundary term and $l_0$ integral terms, depending on where the time derivative falls onto which $f_l$; in each of these terms there is a gain of $\Theta^{-1}$ with respect to the previous step.
		\item The resulting \textit{boundary terms} are well-behaved.
		\item If the time derivative fell onto $S_0, S_{\reg}$ or $K_0$, the term is called a \textit{derivative term} and it is well-behaved.
		\item If the time derivative fell onto $w$, we replace $\partial_t w$ using \eqref{eq:eq_w}. The terms with $F_2$ become well-behaved. The remaining ones are part of the badly-behaved terms of step $J$ and can be written in the form \eqref{eq:genericform}; we emphasize that among these remaining terms, those corresponding to source terms (with no $w$) have restricted domains of integration $D_i$, according to the decomposition in \eqref{eq:decompF1}.
	\end{enumerate}
	We emphasize that a well-behaved term at step $J$ is not considered at step $J+1$. 
	
	For each $J\ge 1$, let $\mathcal{N}_{\text{res}}^J$ be the sum of all resonant terms appearing at step $J$. Analogously, we define $\mathcal{N}_{\text{bd}}^J$, $\mathcal{N}_{\text{dt}}^J$ and $\mathcal{N}_{F_2}^{J+1}$ for the boundary, derivative and $F_2$ terms. Finally, $\mathcal{R}^J(t)$ is the sum of all badly-behaved terms at step $J$. This allows us to write the equation for $w$ as
	\begin{equation}\label{eq:nfe}
		w(t)=\sum_{j=1}^J \left(\mathcal{N}_{\text{res}}^j(t)+\mathcal{N}_{\text{bd}}^j(t)+\mathcal{N}_{\text{dt}}^j(t)+\mathcal{N}_{F_2}^{j-1}(t) \right)\ +\mathcal{N}_{F_2}^J(t)+ \mathcal{R}^{J+1}(t).
	\end{equation}
	
	\subsection{Tree representation of the terms in the INFR}
	
	As mentioned above, it is convenient to describe a term appearing in the INFR by a tree, which we call admissible.  We now explain how to construct the corresponding set of admissible trees at each step.
	
	Each node (internal or leaf) is colored with one color among 
	\[ \{w,z,S_0,S_{\reg},K_0,\partial_tS_0, \partial_tS_{\reg}, \partial_tK_0, F_2\}. \]
	An elementary tree is be a tree with a single root, colored $w$, and 2 or 3 leafs, satisfying either one of the following:
	\begin{itemize}
		\item It has three leafs, where at least two of them are colored either $w$ or $z$ and the third is colored $w,z,S_0$ or $S_{\reg}$. This corresponds to the quadratic and cubic terms in $w$ or $z$.
		\item It has three leafs, one of them colored $S_{\reg}$, another colored $S_0$ or $S_{\reg}$ and the third colored $w$ or $z$. This corresponds to the linear terms in either $w$ or $z$ which have a trilinear structure.
		\item It has two leafs, one colored $K_0$ and the other colored either $w$ or $z$. This corresponds to the linear terms in $w$ or $z$ with a bilinear structure related to the self-similar solution.
	\end{itemize}

	An admissible tree of size $J$ will be a colored tree with $J$ parents, together with a numbering $\#:\{\text{parents}\}\to \{1,\dots,J\}$ such that:
	\begin{itemize}
		\item For each parent node, the subtree made of this node and its children is an elementary tree.
		\item If one parent $x$ is an ancestor of another parent $y$, then $\#(x)<\#(y)$.  
	\end{itemize}
	Analogously to the normal form procedure, we will build resonant, boundary, derivative and $F_2$ trees.
	The parallel algorithm for admissible trees is the following: given an admissible tree of length $J$, the \emph{tree-loop} is as follows:
	\begin{enumerate}
		\item Put a copy of it in the resonant trees set.
		\item Put a copy of it in the boundary trees set.
		\item For each leaf not colored $z$, create a copy of the tree, which will be extended as follows.
		\item If the chosen leaf is colored either $S_0, S_{\reg}$ or $K_0$, change its color to $\partial_tS_0, \partial_tS_{\reg}$ or $\partial_tK_0$, respectively, and put the tree in the derivative trees set.
		\item If the chosen leaf is colored $w$, 
		\begin{enumerate}
			\item create one copy where this leaf is replaced with  $F_2$ (and put it in the $F_2$-trees set),
			\item create copies where the leaf is replaced with each of the elementary trees. In the latter, the parent of the elementary tree (which replaced the leaf) is numbered $J+1$.
		\end{enumerate}
	\end{enumerate}
	
	Notice that all numbered nodes are parents, and so are colored $w$. Also, by direct induction, derivatives or $F_2$ factor can only occur at most once in a given admissible term or tree, as a child of the final elementary subtree (i.e whose parent is numbered $J$ in a tree of size $J$). 
	
	\bigskip
	
	The term and tree generation algorithms are equivalent: there is a one-to-one correspondence between terms at step $J$ of the INFR and admissible trees of length $J$. Here is how to proceed.
	
	\smallskip
	Given a term (at step $J$) by the algorithm, the associated tree is constructed inductively on $j =1, \dots, J$: at each step $j$, we consider the choice made in step (5) of the term-loop. The factor where the derivative falls corresponds to a terminal node (leaf) which we replace by an elementary tree, whose parent is labeled $j$ and whose leafs are colored according to the choice made in step (5) of the term-loop.
	
	\smallskip
	Reciprocally, given an admissible tree, we can construct the associated term, inductively on the index $j$ of the numbering $\#$. At the $j^{th}$ step: consider the node $x$ numbered $\#(x) = j$. Denote $\xi^j$ the corresponding frequency. Then in the $(j-1)^{th}$ term, one replaces the factor $w(t,\xi^j)$ corresponding to $x$ with $N[f,g,h](t,\xi^j)$ where the intervening function $f$, $g$, $h$ are chosen according to colors of the children of $x$, except when we are facing a source term (that is no children is colored $w$), in which case we replace it with the adequate component of $F_1[S,z](t,t_k, \xi^j)$, as  detailed in \eqref{eq:decompF1}. Notice that these source terms come all with frequency restrictions; and we emphasize that the complementary frequency restrictions are taken care of by the $F_2$ term, which appeared \emph{at an earlier step} (that of the parent of $x$).
		
	\smallskip
	
	\begin{figure}[ht]
	\flushleft For example, the trees
		
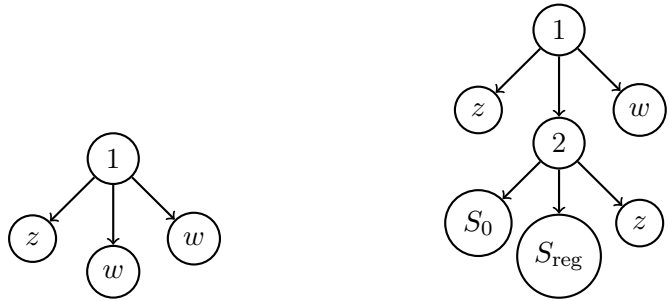
\begin{center}
		\begin{tikzpicture}[node distance={15mm}, thick, main/.style = {draw, circle}] 
				\node[main] (1) {1}; 
				\node[main] (2) [below left of=1] {$z$}; 
				\node[main] (3) [below of=1] {$w$}; 
				\node[main] (4) [below right of=1] {$w$};
				\draw[->] (1) -- (2); 
				\draw[->] (1) -- (3); 
				\draw[->] (1) -- (4);
	\end{tikzpicture}
	\hspace{7em}	
	\begin{tikzpicture}[node distance={15mm}, thick, main/.style = {draw, circle}] 
				\node[main] (1) {1}; 
				\node[main] (2) [below left of=1] {$z$}; 
				\node[main] (3) [below of=1] {2}; 
				\node[main] (4) [below right of=1] {$w$};
				\node[main] (5) [below left of=3] {$S_0$};
				\node[main] (6) [below of=3] {$S_{\text{reg}}$};
				\node[main] (7) [below right of=3] {$z$};
				\draw[->] (1) -- (2); 
				\draw[->] (1) -- (3); 
				\draw[->] (1) -- (4);
				\draw[->] (3) -- (5); 
				\draw[->] (3) -- (6); 
				\draw[->] (3) -- (7);
	\end{tikzpicture}
	\end{center}
	\end{figure}
	correspond, respectively, to 
		\[
	 \iint_{\xi=\xi_1+\xi_2+\xi_3} e^{i t \Phi} z(\xi_1) w(t,\xi_2)w(t,\xi_3)  d\xi_1 d\xi_2.
	\]
	and
	\begin{multline*}
	\iint_{\xi=\xi_1+\xi_2+\xi_3}e^{i t \Phi}z(\xi_1)\left[\iint_{\substack{D_1(t_{k+1},\xi^j)\\\xi_2=\xi_{21}+\xi_{22}+\xi_{23}}} e^{i t \Phi(\xi_2,\xi_{21},\xi_{22}.\xi_{23})} S_0(t,\xi_{21}) S_{\text{reg}}(t,\xi_{22}) z(\xi_{23}) d\xi_{21} d\xi_{22}\right] \\
	w(t,\xi_3)d\xi_1d\xi_2.
	\end{multline*}
	
	\bigskip

	In the next subsection, this identification will help us derive appropriate bounds on the normal form expansion.
	
	\bigskip
	
	Since the elementary trees are at most ternary, an admissible tree of length $J$ will have at most $2J+1$ terminal nodes. Furthermore, there are $18$ different elementary trees (corresponding to the number of nonlinear terms in the equation for $w$); at every step of the algorithm, each node can create at most $18$ new admissible trees. Therefore, denoting ${AT}_J$ the set of admissible trees of length $J$,  the total number of admissible trees of length $J$ is bounded by
	\begin{equation} \label{est:card_AT}
		\text{Card}({AT}_J) \le 18^{J+1} \cdot (2J+1)!
	\end{equation}

	\subsection{Bounds on the infinite normal form equation}\label{sec:INFR_bounds}
	
	%

	We now define the frequency thresholds sequence $(N_j)_{j \in \m N}$. To that end, we fix
	\begin{equation}\label{eq:defibeta}
		\max\left\{ 1-\frac{\mu}{3}, 1-\frac{\nu-\mu-3\gamma}{2}  \right\}<\beta<1.
	\end{equation}
	With this choice of $\beta$, in every single frequency-restricted estimate of Section \ref{sec:multi}, the integral is controlled by $M^\beta$. 
	
	Let $(c_j)_{j \ge 1}$  be an increasing sequence, with $c_1=1$ and such that there exists $C>0$ for which
	\begin{equation}\label{eq:cj}
		\forall J \ge 1, \quad \text{Card}(AT_J) \cdot c_J^\beta\prod_{j=1}^{J-1}c_j^{\beta-1}\le C.
	\end{equation}
	For example, one may choose $c_j=j^{3/(1-\beta)}$.  We then define
	\begin{equation}\label{eq:Nj}
		N_j=c_j/t_{k+1}.
	\end{equation}
	Let $\epsilon >0$ small, to be fixed at the end of the argument. We will bound the terms appearing in the normal form expansion by using a bootstrap argument based on the assumption that 
	\begin{equation}
		\label{eq:bootstrapw}
		\forall t \in [0, T], \quad \|w(t)\|_{L^\infty_\mu}\lesssim \epsilon t^{\gamma}.
	\end{equation}
	Through the infinite normal form expansion, we will prove that 
	\[ \forall k\ \forall t \in [t_{k+1}, t_k], \quad \|w(t) - w(t_{k+1}) \|_{L^\infty_\mu}\lesssim \epsilon^3 t^{\gamma}, \]
	which allows us to recover
\begin{equation}\label{eq:recover}
    	 \forall t \in [0,T], \quad \|w(t) \|_{L^\infty_\mu}\lesssim \epsilon^3 t^{\gamma}. 
\end{equation}
	
	The gain $\epsilon^2$ will come from the smallness of the nonlinear terms, which are at least cubic. From the bootstrap assumption, $w$ is small. By hypothesis, the self-similar solution $S$ is small, namely it satisfies
	\begin{gather} \label{bd:S}
		\forall t >0, \quad \| S \|_{W^{1,\infty}_{0,1}(\m R \setminus \{ 0 \})} \le \epsilon,
	\end{gather}
	and so
	\begin{gather} \label{bd:K0}
		\forall t >0, \quad \| \xi^{1/2} K_0 \|_{L^\infty} \lesssim \epsilon^2 t^{ -1/2}, \quad \| \partial_t S \|_{L^\infty} \lesssim \e t^{-1}.
	\end{gather}
	We need to ensure that $z$ is also small in some space-time norm. Since $z$ is not small in frequency, we need to use weights in time to force the smallness, and this will make our choice of $T>0$, which is the last free parameter. Fix
	\begin{equation}\label{eq:defirho}
		0<\rho<\min\left\{\frac{1-\beta}{3}, \frac{\nu-\mu-3\gamma}{9},  \frac{\nu- 3 \gamma}{12}   \right\} \quad \text{and} \quad T _0 ':= \frac{\epsilon}{\|z\|_{L^\infty_\nu}^{1/\rho}}. 
	\end{equation}
	The idea is to write
	\begin{equation} \label{bd:z}
		\forall t \in [0,T_0'], \quad \|  z\|_{L^\infty_{\nu}} \lesssim \epsilon t^{-\rho}.
	\end{equation}
	This gives smallness, up to the loss in time $t^{-\rho}$  which must be recovered in the estimates for the normal form equation. Also, from  Proposition \ref{prop:control_F2} (and in view of \eqref{bd:S}, \eqref{bd:K0}), 
	\begin{equation} \label{bd:F2}
		\forall t \in [0,T_0'], \quad \| F_2[S,z](t,t_{k+1}) \|_{L^\infty_\mu} \lesssim \epsilon^3 t^{-1+ \frac{\nu-\mu}{3} - 3\rho} \lesssim  \epsilon^3 t^{-1+ \gamma}. 
	\end{equation}
	
	Throughout this section, we will thus suppose that we have the following

\medskip

{\bfseries Bootstrap Assumption}. $w$ is defined on $[0,T]$ with  $T<T_0'$, satisfies \eqref{eq:integralw} and the assumption \eqref{eq:bootstrapw} holds.

\medskip

After recovering estimate \eqref{eq:recover}, the bootstrap argument is done in Proposition \ref{prop:uniform}.

From now on, we fix $k \in \m N$ and we work on the interval $[t_{k+1}, t_k]$, with the equation \eqref{eq:w_tk}. 
	The precise definition of each well-behaved term in the INFR derived from \eqref{eq:w_tk} will depend on $k$, but for most of this section, this dependence is not relevant here (since $k$ is fixed), and we will omit it here. We however emphasize that in the following, the implicit constants do not depend on $k$ (nor $n$).

	\bigskip
	
	In order to explain how one finds \textit{a priori} bounds through the INFR, we need to introduce some notation to make the form \eqref{eq:genericform} more precise. To fix some ideas, let us consider an element of resonant type, with an admissible tree of length $J$. If the $j^{th}$-subtree (the elementary tree whose root is the node numbered $j$) is ternary, 
	\begin{itemize}
		\item $\Xi^j$ will denote the frequency variables corresponding to two children (they are called the represented frequencies); the choice of children is not important, as one can use the convolution relation in the integrals that follow to pass from one choice to the other.  
		We call the frequency corresponding to the unchosen child the \textbf{unrepresented frequency} (corresponding to the dependent variable in the integration).	
		Furthermore, $\overline{\Xi}^j=\bigcup_{m\le j} \Xi^j$ represents the \textbf{total variables}.
		\item 
		It is convenient to number the frequency via word indices: $\xi_\emptyset = \xi$ and if $\xi_{a}$ represents the parent node, the frequencies for its children will be denoted $\xi_{a 1}, \xi_{a_2}$ and $\xi_{a 3}$.
		\item We write the convolution hyperplane as 
		\[ \Gamma^j = H_{\xi_a} = \{ (\xi_{a 1}, \xi_{a 2}, \xi_{a 3}) :  \xi_{a} = \xi_{a 1}+ \xi_{a 2}+ \xi_{a 3} \}. \]
		The total convolution surface, determined by the convolution relations at each step $k\le j$, is 
		\[ \overline{\Gamma}^j = \{ (\overline{\xi}, \xi_{a1}, \xi_{a2}, \xi_{a3}) : \overline{\xi} \in \overline{\Gamma}^{j-1},  (\xi_{a1}, \xi_{a2}, \xi_{a3}) \in {\Gamma}^j \}. \]
		\item the corresponding phase function as $\Theta^j = \Phi(\xi_a,\xi_{a 1}, \xi_{a 2}, \xi_{a 3})$. 
		\item the\textbf{ total phase function} of the $j^{th}$-subtree as $\overline{\Theta}^j=\sum_{j'\le j} \Theta^{j'}$.
		\item $m^j$ will denote the multiplier introduced by the subtree. More precisely, it is the parent frequency (coming from the derivative loss), together with any possible restrictions in the domain of integration.
		\item if a child is a  { terminal}  node  (leaf) of the full tree, then it is associated to $f$, where $f$ is one of the intervening functions. If $\zeta$ is the frequency associated to this node, we define its \textbf{frequency weight}  as
		\begin{equation}\label{eq:frequencyweights}
			\begin{cases}
				t^{-\gamma} \jap{\zeta}^\mu & \quad \text{if }f=w,\\
				t^\rho \jap{\zeta}^\nu & \quad \text{if }f=z, \\
				1 &\quad \text{if } f=S_0, \\
				\jap{t^{1/3}\zeta}^{(4/7)^-} & \quad \text{if }f=S_{\reg}, \\
				t^{1/2}|\zeta|^{1/2} & \quad \text{if }f=K_0.\\
			\end{cases}
		\end{equation}
		Notice that the weights are chosen so that 
		\[\| f \cdot \text{frequency weight} \|_{L_{t,\zeta}^\infty}\lesssim \begin{cases} \epsilon & \text{if }f\neq K_0, \\
			\epsilon^2 & \text{if }f= K_0.
		\end{cases} \]
		(see Proposition \ref{prop:selfsimilar}, Lemma \ref{lem:estK}, Proposition \ref{prop:control_F2} and \eqref{eq:bootstrapw}). 
		
		\item For non-terminal nodes (parents), the associated weight will always be $\jap{\zeta}^\mu$ (the same weight in frequency as for $w$, with no weight in time).
	\end{itemize}
	
	In the binary case, still denoting $\xi_a$ the frequency of the parent, the child colored $K_0$ has frequency $\xi_{a1}$ and the other has frequency $\xi_{a2}$. The frequency set $\Xi^j$ is the variable of a single child: the dependent frequency is the other one. The convolution hyperplane is simply the line $\Gamma^j = \{ \xi_a = \xi_{a1} + \xi_{a2} \}$ in $\m R^2$ on which we define the phase $\Theta_j = \Psi(\xi_{a},\xi_{a1})$. The other notations are modified accordingly.
	
	\bigskip

	\textbf{A concrete example. }Let us exemplify the INFR procedure and the notations above with a particular case. Let us consider one of the bad behaved term at step 1, namely
	\[
	\int_{t_{k+1}}^t \int_{H_\xi} e^{is\Phi}\xi S(t,\xi_1)z(\xi_2)w(t,\xi_3)d\xi_1d\xi_2ds.
	\]
	According to the notations above, $\Gamma^1$ is the hyperplane $H_\xi = \{ (\xi_1,\xi_2,\xi_3) \in \m R^3 :  \xi=\xi_1+\xi_2+\xi_3 \}$, $\Xi^1=(\xi_1,\xi_2)$, $\xi_3$ is the unrepresented frequency, the phase function is $\Theta^1=\Phi(\xi,\xi_1,\xi_2)$ and the multiplier $m^1$ is simply $\xi$.
	
	According to the INFR algorithm, we split the frequency domain into resonant and nonresonant regions:
	\begin{align*}
		\MoveEqLeft \int_{t_{k+1}}^t \int_{\Gamma^1} e^{is\Phi}\xi Szwd\xi_1d\xi_2ds  \\
		& = \int_{t_{k+1}}^t \int_{\Gamma^1} e^{is\Phi}\xi\mathbbm{1}_{|\Phi|<N_1} Szwd\xi_1d\xi_2ds + \int_{t_{k+1}}^t \int_{\Gamma^1} e^{is\Phi}\xi\mathbbm{1}_{|\Phi|>N_1} Szwd\xi_1d\xi_2ds.
	\end{align*}
	In the next computations, we systematically use \eqref{eq:bootstrapw}, \eqref{bd:z}, \eqref{bd:S} and \eqref{bd:K0}. Using the frequency-restricted estimate of Proposition \ref{prop:fre_phi}, we are able to bound directly the resonant term. Indeed,
	\begin{align*}
		\left\|\int_{\Gamma^1} e^{is\Phi}\xi\mathbbm{1}_{|\Phi|<N_1} Szwd\xi_1d\xi_2\right\|_{L^\infty_\mu} &\lesssim \sup_\xi \int_{\Gamma^1} \frac{|\xi|\jap{\xi}^\mu }{\jap{\xi_2}^\mu \jap{\xi_3}^\mu }\mathbbm{1}_{|\Phi|<N_1}d\xi_1d\xi_2 \cdot \|S\|_{L^\infty}\|z\|_{L^\infty_\mu} \|w\|_{L^\infty_\mu} \\
		&\lesssim N_1^{1-\mu/3}s^{\gamma-\rho}\epsilon^3 \lesssim c_1^{1-\mu/3} \epsilon^3 s^{-1+\gamma+\mu/3-\rho}.
	\end{align*}
	Integrating in $s \in[t_{k+1},t]$ yields (using that $\rho \le 1-\beta/3 \le \mu/3$)
	\begin{equation}
		\left\|\int_{t_{k+1}}^t\int_{\Gamma^1} e^{is\Phi}\xi\mathbbm{1}_{|\Phi|<N_1} Szwd\xi_1d\xi_2ds\right\|_{L^\infty_\mu} \lesssim  c_1^{1-\mu/3} \epsilon^3\int_{t_{k+1}}^t s^{-1+\gamma+\mu/3-\rho}ds \lesssim  c_1^{1-\mu/3} \epsilon^3t^{\gamma}.
	\end{equation}
	For the nonresonant term, we integrate by parts in time:
	\begin{align*}
		\int_{t_{k+1}}^t \int_{\Gamma^1} e^{is\Phi}\xi\mathbbm{1}_{|\Phi|>N_1} Szwd\xi_1d\xi_2ds &= \left[ \int_{\Gamma^1} \frac{e^{is\Phi}}{i\Phi}\xi\mathbbm{1}_{|\Phi|>N_1} Szwd\xi_1d\xi_2 \right]_{s={t_{k+1}}}^{s=t} \\
		& \quad - 	\int_{t_{k+1}}^t \int_{\Gamma^1} \frac{e^{is\Phi}}{i\Phi}\xi\mathbbm{1}_{|\Phi|>N_1} \partial_s (Szw)d\xi_1d\xi_2ds.
	\end{align*}
	The bound for the boundary term follows from Lemma \ref{lem:restricted_<>} together with \eqref{eq:TalphaM_multiplierxi}: for either $s=t$ or $t_{k+1}$, there holds
	\begin{align*}
		\MoveEqLeft \left\|\int_{\Gamma^1} \frac{e^{is\Phi}}{i\Phi}\xi\mathbbm{1}_{|\Phi|>N_1} Szwd\xi_1d\xi_2 \right\|_{L^\infty_\mu} \lesssim \sup_\xi \int_{\Gamma^1} \frac{1}{|\Phi|} \frac{|\xi|\jap{\xi}^\mu }{\jap{\xi_2}^\mu \jap{\xi_3}^\mu }\mathbbm{1}_{|\Phi|>N_1} d\xi_1d\xi_2\cdot  \|S\|_{L^\infty}\|z\|_{L^\infty_\mu} \|w\|_{L^\infty_\mu} \\
		&\lesssim  N_1^{-\mu/3}  \epsilon^3 s^{\gamma-\rho} \lesssim c_1^{-\mu/3} \epsilon^3 s^{\mu/3+\gamma-\rho} \lesssim c_1^{ -\mu/3} \epsilon^3 s^\gamma.
	\end{align*}
	Regarding the integral term, we distribute the time derivative. As $z$ does not depend on $t$, there are two cases. If the derivative falls on $S$, we obtain the derivative term
	\[
	\int_{t_{k+1}}^t \int_{\Gamma^1} \frac{e^{is\Phi}}{i\Phi}\xi\mathbbm{1}_{|\Phi|>N_1} (\partial_tS)zwd\xi_1d\xi_2ds.
	\]
	The estimate for this term is very similar to that of the boundary term, the difference being that the loss of a power of $t$, coming from $\partial_t S = O(\epsilon t^{-1})$, is then compensated with the time integration:
	\begin{align*}
		\left\|\int_{\Gamma^1} \frac{e^{is\Phi}}{i\Phi}\xi\mathbbm{1}_{|\Phi|>N_1} (\partial_t S)zwd\xi_1d\xi_2 \right\|_{L^\infty_\mu} &\lesssim \sup_\xi \int_{\Gamma^1} \frac{1}{|\Phi|} \frac{|\xi|\jap{\xi}^\mu }{\jap{\xi_2}^\mu \jap{\xi_3}^\mu }\mathbbm{1}_{|\Phi|>N_1} d\xi_1d\xi_2\cdot  \|\partial_t S\|_{L^\infty}\|z\|_\mu\|w\|_\mu \\
		&\lesssim N_1^{-\mu/3}  \cdot s^{-1+\gamma-\rho}\epsilon^3 \lesssim c_1^{-\mu/3}  \epsilon^3 s^{\mu/3+\gamma-\rho-1},
	\end{align*}
	and so, integrating in $s \in [t_{k+1},t]$, 
	\[
	\left\|\int_{t_{k+1}}^t \int_{\Gamma^1} \frac{e^{is\Phi}}{i\Phi}\xi\mathbbm{1}_{|\Phi|>N_1} (\partial_t S)zwd\xi_1d\xi_2 ds\right\|_{L^\infty_\mu} \lesssim c_1^{1-\mu/3}  \epsilon^3  \int_{t_{k+1}}^t s^{\mu/3+\gamma-\rho-1} ds \lesssim c_1^{-\mu/3} \epsilon^3 t^\gamma.
	\]
	If the time derivative falls on $w$, we use the equation \eqref{eq:w_tk} and the decomposition \eqref{eq:decomp_nonli} to write
	\begin{align}
		\MoveEqLeft \int_{t_{k+1}}^t \int_{\Gamma^1} \frac{e^{is\Phi}}{i\Phi}\xi\mathbbm{1}_{|\Phi|>N_1} Sz(\partial_tw)d\xi_1d\xi_2ds\\
		& = \int_{t_{k+1}}^t \int_{\Gamma^1} \frac{e^{is\Phi}}{i\Phi}\xi\mathbbm{1}_{|\Phi|>N_1} Sz\Big( F_{11}[S,z] +F_{12}[S_0,S_{\reg},z] + F_2[S,z] \\
		&\qquad+ L_{K,D(t_{k+1})}[z] + L_K[w] 
		+ L_2[S_0,S_{\reg},z,w] + Q[S,z,w] \Big)d\xi_1d\xi_2ds. \label{eq:substitui}
	\end{align}
	(In the last factor of the integrand, the terms are evaluated at $(t,t_{k+1},\xi_3)$). As mentioned, the term with $F_2[S,z]$ is well-behaved: by Lemma \ref{lem:restricted_<>}, \eqref{eq:TalphaM_multiplierxi} and \eqref{bd:F2}, 
	\begin{align*}
		\MoveEqLeft \left\| \int_{\Gamma^1} \frac{e^{is\Phi}}{i\Phi}\xi\mathbbm{1}_{|\Phi|>N_1} SzF_2[S,z]d\xi_1d\xi_2 \right\|_{L^\infty_\mu} \\
		& \lesssim \sup_{\xi} \int_{\Gamma^1} \frac{1}{|\Phi|} \frac{|\xi|\jap{\xi}^\mu }{\jap{\xi_2}^\mu \jap{\xi_3}^\mu } \mathbbm{1}_{|\Phi| > N_1} d\xi_1d\xi_2 \cdot \|S\|_{L^\infty}\|z\|_{L^\infty_\mu} \|F_2[S,z]\|_{L^\infty_\mu} \\
		&\lesssim N_1^{-\mu/3} \epsilon^5 s^{-1 + \frac{\nu-\mu}{3} - 4\rho} \lesssim c^{-\mu/3}_1 \epsilon^5 s^{-1+\nu/3 - 4\rho} \lesssim c^{-\mu/3}_1 \epsilon^5 s^{-1+\gamma}.
	\end{align*}
	Then, the integration in time gives the bound $c_1^{-\mu/3} \epsilon^5 t^{\gamma}$, which is $\epsilon^2$ more than what was necessary.
	
	The remaining terms in \eqref{eq:substitui} are badly-behaved terms, and are left to the second step of the expansion. 
	
	\bigskip
	
	In order to give an insight on how the frequency-restricted estimates propagate throughout the INFR expansion, let us consider one of them, say:
	\begin{align*}
		\MoveEqLeft[0] \int_{t_{k+1}}^t \int_{\Gamma^1} \frac{e^{is\Phi}}{i\Phi}\xi\mathbbm{1}_{|\Phi|>N_1} S(s,\xi_1)z(\xi_2)L_K[w](s,\xi_3)d\xi_1d\xi_2ds\\ 
		& =	\int_{t_{k+1}}^t \int_{\Gamma^1} \frac{e^{is\Phi}}{i\Phi}\xi\mathbbm{1}_{|\Phi|>N_1} S(s,\xi_1)z(\xi_2)\left(\int_{\xi_3=\xi_{31}+\xi_{32}}\xi_3e^{is\Psi(\xi_3,\xi_{31})}K_0(s,\xi_{31})w(s,\xi_{32})d\xi_{31}\right)d\xi_1d\xi_2ds\\ 
		& = \int_{t_{k+1}}^t \int_{\overline{\Gamma}^2} \frac{e^{is\overline{\Theta}^2}}{i\Theta^1}m^1m^2\mathbbm{1}_{|\Theta^1|>N_1} S(s, \xi_1) z (\xi_2) K_0(s,\xi_{31}) w(s,\xi_{32}) d\xi_1d\xi_2d\xi_{31}ds,
	\end{align*}
	where $\Xi_2 = \{ \xi_{31} \}$, the dependent frequency is $\xi_{32}$,
	\[
	\Theta^1=\Phi(\xi,\xi_1,\xi_2, \xi_3),\quad \overline{\Theta}^2=\Phi(\xi,\xi_1,\xi_2,\xi_3) + \Psi(\xi_3,\xi_{31}),\quad m^1=\xi,\quad m^2=\xi_3,
	\]
	and
	\[
	\overline{\Gamma}^2=\{(\xi,\xi_1,\xi_2,\xi_3,\xi_{31},\xi_{32})\in \R^6:\xi=\xi_1+\xi_2+\xi_3,\ \xi_3=\xi_{31}+\xi_{32}  \}.
	\]
	The associated admissible tree is represented in Figure \ref{figure1}.
	\begin{figure}[ht]
		\begin{center}
			\begin{tikzpicture}[node distance={15mm}, thick, main/.style = {draw, circle}] 
				\node[main] (1) {1}; 
				\node[main] (2) [below left of=1] {$S$}; 
				\node[main] (3) [below of=1] {$z$}; 
				\node[main] (4) [below right of=1] {$w$};
				\node[main] (12) [right=2cm of 4] {$S$};
				\node[main] (11) [above right of=12]{1};
				\node[main] (13) [below of=11] {$z$}; 
				\node[main] (14) [below right of=11] {2}; 
				\node[main] (15) [below of=14] {$K_0$}; 
				\node[main] (16) [below right of=14] {$w$}; 
				\draw[->] (1) -- (2); 
				\draw[->] (1) -- (3); 
				\draw[->] (1) -- (4);
				\draw[->] (11) -- (12); 
				\draw[->] (11) -- (13); 
				\draw[->] (11) -- (14);
				\draw[->] (14) -- (15);
				\draw[->] (14) -- (16);
			\end{tikzpicture}
		\end{center}
		\caption{The replacement, in the tree on the left, of the $w$ node with an elementary tree yields an admissible tree associated to a badly-behaved term at Step 2.}
		\label{figure1}
	\end{figure}
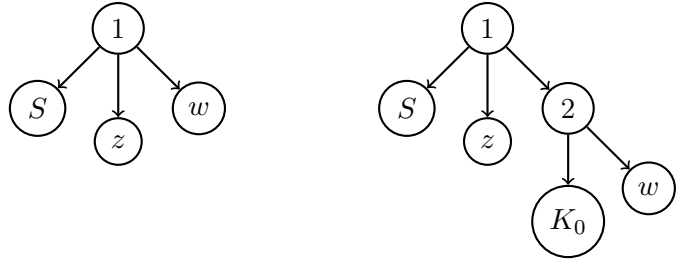

	As before, we split the integral into the resonant region $|\overline{\Theta}^2|<N_2$ and the nonresonant region ${|\overline{\Theta}^2|>N_2}$. Let us perform the estimate for the resonant term. First,
	\begin{align*}
		\MoveEqLeft \left\|\int_{\overline{\Gamma}^2} \frac{e^{is\overline{\Theta}^2}}{i\Theta^1}m^1m^2\mathbbm{1}_{|\Theta^1|>N_1}\mathbbm{1}_{|\overline{\Theta}^2|<N_2} SzK_0w d\xi_1d\xi_2d\xi_{31}\right\|_{L^\infty_\mu} \\
		 & \lesssim \sup_\xi\ \jap{\xi}^\mu \int_{\overline{\Gamma}^2} \frac{1}{|\Theta^1|}\frac{|\xi\xi_3|}{\jap{\xi_2}^\mu |\xi_{31}|^{1/2}\jap{\xi_{32}}^\mu }\mathbbm{1}_{|\Theta^1|>N_1}\mathbbm{1}_{|\overline{\Theta}^2|<N_2} d\xi_1d\xi_2d\xi_{31} \\
		 & \hspace{80mm}  \cdot \|S\|_{L^\infty}\|z\|_{L^\infty_\mu}\|K_0|\xi|^{1/2}\|_{L^\infty}\|w\|_{L^\infty_\mu} \\
		& \lesssim \sup_\xi \int_{\overline{\Gamma}^2} \frac{1}{|\Theta^1|}\frac{|\xi| \jap{\xi}^\mu }{\jap{\xi_2}^\mu \jap{\xi_3}^\mu }\mathbbm{1}_{|\Theta^1|>N_1}\frac{|\xi_3| \jap{\xi_3}^\mu }{|\xi_{31}|^{1/2}\jap{\xi_{32}}^\mu }\mathbbm{1}_{|\overline{\Theta}^2|<N_2} d\xi_1d\xi_2d\xi_{31} \cdot \epsilon^5 s^{-\rho-1/2+\gamma}.
	\end{align*}
	Using the convolution relations, the frequency associated with the last parent can be assumed to be represented. In the specific case above, this can be achieved by replacing $\xi_1$ with $\xi_3$. This we are left to bound
	\[
	\sup_\xi \int_{\overline{\Gamma}^2} \frac{1}{|\Theta^1|}\frac{\xi\jap{\xi}^\mu }{\jap{\xi_2}^\mu \jap{\xi_3}^\mu }\mathbbm{1}_{|\Theta^1|>N_1}\frac{\xi_3\jap{\xi_3}^\mu }{|\xi_{31}|^{1/2}\jap{\xi_{32}}^\mu }\mathbbm{1}_{|\overline{\Theta}^2|<N_2} d\xi_2d\xi_3d\xi_{31}.
	\]
	Now observe that the only frequency that depends on $\xi_{31}$ is his unrepresented brother, $\xi_{32}$. This means we can integrate in $\xi_{31}$ first, leaving the factors corresponding to the first step of the INFR algorithm on the outside:
	\begin{align*}
		\MoveEqLeft[1] \sup_\xi \int_{\Gamma^1} \frac{1}{|\Theta^1|}\frac{\xi\jap{\xi}^\mu }{\jap{\xi_2}^\mu \jap{\xi_3}^\mu }\mathbbm{1}_{|\Theta^1|>N_1}\left(\int_{\Gamma^2}\frac{\xi_3\jap{\xi_3}^\mu }{|\xi_{31}|^{1/2}\jap{\xi_{32}}^\mu }\mathbbm{1}_{|\overline{\Theta}^2|<N_2}d\xi_{31}\right) d\xi_2d\xi_3 \\
		& \lesssim \left[ \sup_\xi \int_{\Gamma^1} \frac{1}{|\Theta^1|}\frac{\xi\jap{\xi}^\mu }{\jap{\xi_2}^\mu \jap{\xi_3}^\mu }\mathbbm{1}_{|\Theta^1|>N_1}d\xi_2d\xi_3\right] \cdot  \left[ \sup_{\xi} \sup_{(\xi_1,\xi_2,\xi_3) \in \Gamma^1} \int_{\Gamma^2}\frac{\xi_3\jap{\xi_3}^\mu }{|\xi_{31}|^{1/2}\jap{\xi_{32}}^\mu }\mathbbm{1}_{|\overline{\Theta}^2|<N_2}d\xi_{31} \right] \\
		& \lesssim  \left[ \sup_\xi \int_{\Gamma^1} \frac{1}{|\Phi|}\frac{\xi\jap{\xi}^\mu }{\jap{\xi_2}^\mu \jap{\xi_3}^\mu }\mathbbm{1}_{|\Phi|>N_1}d\xi_2d\xi_3\right] \cdot  \left[ \sup_{\xi_3,\alpha} \int_{\Gamma^2}\frac{\xi_3\jap{\xi_3}^\mu }{|\xi_{31}|^{1/2}\jap{\xi_{32}}^\mu }\mathbbm{1}_{|\Psi-\alpha|<N_2}d\xi_{31} \right].
	\end{align*}
	This procedure has decoupled the integrals corresponding to steps 1 and 2. Applying the frequency-restricted estimates,
	\begin{align*}
		\sup_{\xi_3,\alpha} \int_{\Gamma^2}\frac{\xi_3\jap{\xi_3}^\mu }{|\xi_{31}|^{1/2}\jap{\xi_{32}}^\mu }\mathbbm{1}_{|\Psi-\alpha|<N_2}d\xi_{31} \lesssim N_2^{1/2},
	\end{align*}
	and, with Lemma \ref{lem:restricted_<>},
	\begin{align*}
		\sup_\xi \int_{\Gamma^1} \frac{1}{|\Theta^1|}\frac{\xi\jap{\xi}^\mu }{\jap{\xi_2}^\mu \jap{\xi_3}^\mu }\mathbbm{1}_{|\Theta^1|>N_1}d\xi_2d\xi_3  \lesssim N_1^{-\mu/3}.
	\end{align*}
	Hence (we use again $\rho \le \mu/3$)
	\begin{multline*}
		\left\|\int_{t_{k+1}}^t\int_{\overline{\Gamma}^2} \frac{e^{is\overline{\Theta}^2}}{i\Theta^1}m^1m^2\mathbbm{1}_{|\Theta^1|>N_1}\mathbbm{1}_{|\overline{\Theta}^2|<N_2} SzK_0w d\xi_1d\xi_2d\xi_{31}ds\right\|_{L^\infty_\mu} \\
		\lesssim \int_{t_{k+1}}^t N_1^{-\mu/3}N_2^{1/2}s^{-1/2-\rho+\gamma}\epsilon^5 ds \lesssim c_1^{-\mu/3} c_2^{1/2} \epsilon^5 t^{\gamma}.
	\end{multline*}
	The analysis of the nonresonant term can be made following the general approach described below.
	
	\begin{nb}
		Here, one can observe the critical behavior of the linearized operator $L_K$. Indeed, the gain of a full power of $t$ coming from the $1/M_1$ factor is compensated exactly by a loss of $t_{k+1}^{-1/2} \approx t^{-1/2}$ coming from $N_2^{1/2}$ and another loss of $t^{-1/2}$ coming from the bound on $K_0$.
	\end{nb}
	
	\bigskip
	
	\textbf{The general case. } We now proceed to the estimates of a generic well-behaved term in the INFR procedure. 
	
	\emph{Case 1.} We begin with a general resonant term,
	\begin{align} \label{eq:formares}
		\cal N_{\text{res}} &= \int_{t_{k+1}}^{t}N(s,\xi)ds\\ \quad \text{where} \quad
		N(t,\xi)&=\int_{\overline{\Gamma}^J}e^{it\overline{\Theta}^J} \left(\prod_{j=1}^{J-1}\frac{m^j}{i\overline{\Theta}^j}\mathbbm{1}_{|\overline{\Theta}^j|>N_j}\right) m^J\mathbbm{1}_{|\overline{\Theta}^J|<N_J}\cdot  F(t,\xi,\overline{\Xi}^J) d\overline{\Xi}^J,
	\end{align}
 and $F$ is a product of intervening functions: counting double the $K_0$ factors (and one for each of other factors), there is a total of $2J+1$ factors, as can be seen by immediate induction. 
	
	Our goal is to derive an $L^\infty_\mu$ bound for this term, in terms of some weighted $L^\infty$ bounds on the intervening functions. We multiply and divide each of the factors $f$ in $F$ by the  corresponding frequency weight according to \eqref{eq:frequencyweights}, so that the numerator $g$ can is bounded in $L^\infty_{t,\xi_a}$ by $\epsilon$ (or $\epsilon^2$ if $f$ involves $K_0$). Denote $G$ the products of the $g$: then
	\[ \| G \| _{L^\infty_{t,\xi,\overline{\Xi}^J}} \lesssim \e^{2J+1}, \]
	and the tree structure allows to us to rewrite:
	\begin{equation} \label{def:weight_G}
	\jap{\xi}^\mu \left(\prod_{j=1}^J m^j\right)F =  \left(\prod_{j=1}^J \mathcal{M}^j\right)G,
	\quad \text{where} \quad
	\mathcal{M}^j=\frac{m^j \times \text{weight for node } j}{\prod \text{ weights for its children}}. \end{equation}
	Then
	\begin{align}
		\|N(t)\|_{L^\infty_\mu}&\lesssim \sup_\xi\ \jap{\xi}^\mu \int_{\overline{\Gamma}^J} \left(\prod_{j=1}^{J-1} \frac{|m^j|}{|\overline{\Theta}^j|} \mathbbm{1}_{|\overline{\Theta}^j|>N_j}\right) |m^J| \mathbbm{1}_{|\overline{\Theta}^J|<N_J}\times F(t,\xi,\overline{\Xi}^J) d\overline{\Xi}^J \\ 
		&\lesssim \sup_\xi\ \int_{\overline{\Gamma}^J} \left(\prod_{j=1}^{J-1} \frac{|\mathcal{M}^j|}{|\overline{\Theta}^j|} \mathbbm{1}_{|\overline{\Theta}^j|>N_j}\right) |\mathcal{M}^J| \mathbbm{1}_{|\overline{\Theta}^J|<N_J} d\overline{\Xi}^J\times \|G\|_{L^\infty} \\
		&=: I^J\times\|G\|_{L^\infty}.\label{eq:est_termoJ}
	\end{align}
	We now decompose the estimate for $I^J$ (which can be seen as an $L^\infty_\xi(L^1_{\overline{\Xi}^J})$ estimate) into $J$ estimates of the same type, each one involving a single subtree. This comes from the iterative scheme, which induces a lower triangular dependence of the phase functions $\Theta^j$ on the frequencies $\Xi^j$. Indeed, the crucial observation is that, after ensuring that the last parent is represented, the only frequency which depends on $\Xi^J$ is the unrepresented frequency of the $J^{th}$-subtree. In particular, in $I^J$, only $\mathcal{M}^J$ and $\overline{\Theta}^J$ depend on $\Xi^J$. This allows us to split the integration as follows:
	\begin{align*}
		I^J & = \sup_\xi\ \int_{\overline{\Gamma}^{J-1}} \left(\prod_{j=1}^{J-1} \frac{|\mathcal{M}^j|}{|\overline{\Theta}^j|} \mathbbm{1}_{|\overline{\Theta}^j|>N_j}\right)\left(\int_{\Gamma^J} |\mathcal{M}^J| \mathbbm{1}_{|\overline{\Theta}^J|<N_J}d\Xi^J\right) d\overline{\Xi}^{J-1}\\
		&\lesssim \left[\sup_\xi\ \int_{\overline{\Gamma}^{J-1}} \left(\prod_{j=1}^{J-1}\frac{|\mathcal{M}^j|}{|\overline{\Theta}^j|} \mathbbm{1}_{|\overline{\Theta}^j|>N_j}\right) d\overline{\Xi}^{J-1}\right] \cdot \sup_{\xi,\overline{\Xi}^{J-1}}\left(\int_{\Gamma^J} |\mathcal{M}^J| \mathbbm{1}_{|\overline{\Theta}^J|<N_J}d\Xi^J\right) .
	\end{align*}
	Now, as $\overline{\Theta}^J = \overline{\Theta}^{J-1} + \Theta^J$ and $\overline{\Theta}$ only depends on $\Xi^{J-1}$,
	\[  \sup_{\xi,\overline{\Xi}^{J-1}}\left(\int_{\Gamma^J} |\mathcal{M}^J| \mathbbm{1}_{|\overline{\Theta}^J|<N_J}d\Xi^J\right)  \le \sup_{\xi^J,\alpha}\left(\int_{\Gamma^J} |\mathcal{M}^J| \mathbbm{1}_{|\Theta^J-\alpha|<N_J}d\Xi^J\right). \]
	Therefore,
	\begin{align}
		I^J &\lesssim \bar I ^{J-1} \cdot \sup_{\xi^J,\alpha}\left(\int_{\Gamma^J} |\mathcal{M}^J| \mathbbm{1}_{|\Theta^J-\alpha|<N_J}d\Xi^J\right)\label{eq:uppertriang} \\&
		\text{with} \quad  \bar I ^{J-1}: = \sup_\xi\ \int_{\overline{\Gamma}^{J-1}} \left(\prod_{j=1}^{J-1} \frac{|\mathcal{M}^j|}{|\overline{\Theta}^j|} \mathbbm{1}_{|\overline{\Theta}^j|>N_j}\right)d\overline{\Xi}^{J-1}.
	\end{align}
	As $\bar I^{J-1}$ has the same structure as $I^J$, we can repeat (a small variation of) the above argument:
	\begin{align*}
		\bar I^{J-1} & =  \sup_\xi\ \int_{\overline{\Gamma}^{J-2}} \left(\prod_{j=1}^{J-2} \frac{|\mathcal{M}^j|}{|\overline{\Theta}^j|} \mathbbm{1}_{|\overline{\Theta}^j|>N_j}\right) \left(\int_{\Gamma^{J-1}} \frac{|\mathcal{M}^{J-1}|}{|\overline{\Theta}^{J-1}|} \mathbbm{1}_{|\overline{\Theta}^{J-1}| > N_{J-1} }d\Xi^{J-1} \right)  d\overline{\Xi}^{J-2} \\
		& \le  \sup_\xi\ \int_{\overline{\Gamma}^{J-2}} \left(\prod_{j=1}^{J-2} \frac{|\mathcal{M}^j|}{\overline{\Theta}^j} \mathbbm{1}_{|\overline{\Theta}^j|>N_j}\right)   d\overline{\Xi}^{J-2} \cdot \sup_{\xi, \overline \Xi^{J-2}} \left(\int_{\Gamma^{J-1}} \frac{|\mathcal{M}^{J-1}|}{|\overline \Theta^{J-1}|} \mathbbm{1}_{|\overline{\Theta}^{J-1}| > N_{J-1} }d\Xi^{J-1} \right) \\
		& \le \bar I^{J-2 } \cdot \sup_{\xi, \overline \Xi^{J-2}} \left(\int_{\Gamma^{J-1}} \frac{|\mathcal{M}^{J-1}|}{| \Theta^{j-1} + \overline \Theta^{J-2}|} \mathbbm{1}_{|\Theta^{J-1} + \overline{\Theta}^{J-2}| > N_{J-1} }d\Xi^{J-1} \right) \\
		& \le  \bar I^{J-2 } \cdot \sup_{\xi,\alpha} \left(\int_{\Gamma^{J-1}} \frac{|\mathcal{M}^{J-1}|}{| \Theta^{J-1} -\alpha|} \mathbbm{1}_{|\Theta^{J-1} -\alpha| > N_{J-1} }d\Xi^{J-1} \right).
	\end{align*}
	An inductive application of this procedure yields
	\[ \bar I^{J-1} \le \prod_{j=1}^{J-1} \sup_{\xi,\alpha} \left(\int_{\Gamma^{j-1}} \frac{|\mathcal{M}^{j-1}|}{| \Theta^{j-1} -\alpha|} \mathbbm{1}_{|\Theta^{j-1} -\alpha| > N_{j-1} }d\Xi^{j-1} \right). \]
	Combining this with \eqref{eq:est_termoJ} and \eqref{eq:uppertriang}, 
	\begin{multline} \label{est:Nres}
		|\cal N_{\text{res}}| \lesssim \int_{t_{k+1}}^t      \left( \prod_{j=1}^{J-1}\sup_{\xi^j,\alpha}\ \int_{{\Gamma}^{j}} \frac{|\mathcal{M}^j|}{|\Theta^j-\alpha|} \mathbbm{1}_{|\Theta^j-\alpha|>N_j}d\Xi^{j}\right) \\
		\cdot \sup_{\xi^J,\alpha}\left(\int_{\Gamma^J} |\mathcal{M}^J| \mathbbm{1}_{|\Theta^J-\alpha|<N_J}d\Xi^J\right) ds \| G \|_{L^\infty(I_k)}. 
	\end{multline}
	
	So, due to Lemma \ref{lem:restricted_<>}, it suffices to estimate quantities of the form
	\[ \int_{\Gamma^j} |\mathcal{M}^j| \mathbbm{1}_{|\Theta^j-\alpha|< M}d\Xi^j \]
	for $j \le J$, which will follow from a direct application of the results of the previous section. 
	
	For boundary, derivative or $F_2$-terms, one can proceed as above. Indeed:
	\begin{itemize}
		\item A boundary term of length $J$ can be written as
		\[
		\mathcal{N}(t,\xi)=\left[ N(s,\xi) \right]_{s=t_{k+1}}^{s=t},
		\]
		where
		\begin{equation}\label{eq:formabdry}
			N(t,\xi) =\int_{\overline{\Gamma}^J}e^{it\overline{\Theta}^J} \left(\prod_{j=1}^{J}\frac{m^j}{i\overline{\Theta}^j}\mathbbm{1}_{|\overline{\Theta}^j|> N_j}\right) F(t,\xi,\overline{\Xi}^J) d\overline{\Xi}^J.
		\end{equation}
		The differences with the resonent term are the lack of integration in time and the nonresonance condition at the last subtree. Therefore, we get
		\begin{equation} \label{est:Nbd}
			|\cal N_{\text{bd}}| \lesssim \left[\prod_{j=1}^{J}\sup_{\xi^j,\alpha}\ \int_{{\Gamma}^{j}} \frac{|\mathcal{M}^j|}{|\Theta^j-\alpha|} \mathbbm{1}_{|\Theta^j-\alpha|>N_j}d\Xi^{j}\right] \cdot \| G \|_{L^\infty(I_k)}. 
		\end{equation}
		
		\item We can write derivative or $F_2$-terms of length $J$ as 
		\[
		\int_{t_{k+1}}^tN(s,\xi)ds \quad \text{with }  N \text{ as in } \eqref{eq:formabdry},
		\]
		with the feature that in the last subtree, one of the children is colored $\partial_t S_0,\partial_t S_{\reg}$ or $\partial_tK_0$ (for derivative term); or $F_2$ (for $F_2$ term). According to Propositions \ref{prop:selfsimilar}, \ref{prop:control_F2} and Lemma \ref{lem:estdK}, the correct frequency weight associated to this node is
		\begin{equation} \label{eq:frequencyweights_2}
			\begin{cases}
				t & \text{if } f=\partial_t S_0, \\
				t\jap{t^{-1/3}\zeta}^{(4/7)^-} & \text{if } f=\partial_t S_{\reg}, \\
				t^{3/2}|\zeta|^{1/2} & \text{if } f=\partial_t K_0, \\
				t^{1-\gamma}\jap{\zeta}^\mu & \text{if } f=F_2.
			\end{cases}
		\end{equation}
		Notice that each weight gains a $t$ factor compared to the corresponding weight for the non derivated function in \eqref{eq:frequencyweights}, and they are chosen such that
		\[ \| f \cdot \text{frequency weight} \|_{L^\infty_{t,\zeta}} \lesssim \begin{cases}
			\e & \text{if } f = \partial_t S_0 \text{ or } \partial_t S_{\text{reg}}, \\
			\epsilon^2 & \text{if } f  = \partial_t K_0, \\
			\epsilon ^3 & \text{if } f = F_2.
		\end{cases} \]
		The above discussion can be summarized as
		\begin{equation}
			\label{est:Ndt_NF2}
			|\cal N_{\text{dt}}|, \ |\cal N_{F_2}| \lesssim \int_{t_{k+1}}^t \left[\prod_{j=1}^{J}\sup_{\xi^j,\alpha}\ \int_{{\Gamma}^{j}} \frac{|\mathcal{M}^j|}{|\Theta^j-\alpha|} \mathbbm{1}_{|\Theta^j-\alpha|>N_j}d\Xi^{j}\right]  ds \cdot \| G \|_{L^\infty(I_k)}.
		\end{equation}
		
		\end{itemize}

		We emphasize that in the following, implicits constants are independent of $j$ and $k$.
		
		We first consider the case of an intermediate subtree, that is when $j <J$.
		\begin{lem}\label{lem:estsubtree}
			Let $t \in [t_{k+1},t_k]$ and, in a term corresponding to a tree of size $J$, consider the factor given by the $j^{th}$ elementary subtree,  for some $j <J$. Then
			\[ \forall M \ge 1, \ \forall \alpha \in \m R, \quad 
			\int_{\Gamma^j} |\mathcal{M}^j| \mathbbm{1}_{|\Theta^j-\alpha|< M}d\Xi^j  \lesssim t^{-1+\beta} M^\beta. \]
			As a consequence, 
			\[ \quad \sup_{\xi^j, \alpha}  \int_{\Gamma^j} \frac{|\mathcal{M}^j|}{|\Theta^j - \alpha|} \mathbbm{1}_{|\Theta^j-\alpha| \ge N_j} d\Xi^j  \lesssim (t N_j)^{-1+\beta}. \]
		\end{lem}
		
		
		\begin{proof}
			In a non-terminal subtree, the frequency weights are among those in \eqref{eq:frequencyweights}.
			
			a) We start with the case $\Theta^j=\Phi$ and denote $\xi=\xi^j$, $\Xi^j= ( \xi_1,\xi_2 )$.  Then the children are colored among $S_0, S_{\reg}, z,w$, and there is at most one $S_0$ and at least one $w/z$.
			
			We split the analysis on the various possibilities of weights for the $j^{th}$-subtree:
			\begin{itemize}
				\item at least two children have weights corresponding to $w$ and/or $z$. Then, the third child has frequency weight $\gtrsim t^{\rho}$, and up to permutations on the frequencies,
				\[ 
				|\mathcal{M}^j| \lesssim \frac{|\xi|\jap{\xi}^\mu }{\jap{\xi_1}^\mu \jap{\xi_2}^\mu }t^{-3\rho}. \]
				Then it follows from  \eqref{eq:TalphaM_multiplierxi} that
				\[ \sup_{\xi,\alpha} \int {|\mathcal{M}^j|}\mathbbm{1}_{|\Phi-\alpha|<M_j}d\xi_1d\xi_2\lesssim t^{-3\rho}M_j^{1-\mu/3}. \]
				\item one child has a $w$-weight, one has a $S_{\reg}$-weight and the remaining has a weight corresponding to either $S_{\reg}$ or $S_0$. In this case, 
				\[
				|\mathcal{M}^j| \lesssim \frac{|\xi|\jap{\xi}^\mu }{\jap{\xi_1}^\mu \jap{t^{1/3}\xi_2}^{(4/7)^-}} \lesssim \frac{|\xi|\jap{\xi}^\mu }{t^{\mu/3}\jap{\xi_1}^\mu \jap{\xi_2}^{\mu}}, \]
				and
				\[ 
				\sup_{\xi,\alpha} \int |\mathcal{M}^j|\mathbbm{1}_{|\Phi-\alpha|<M}d\xi_1d\xi_2\lesssim \frac{1}{t^{\mu/3}}M^{1-\mu/3}. \]
				\item one child has a $z$-weight, one has a $S_{\reg}$-weight and the remaining has a weight corresponding to either $S_{\reg}$ or $S_0$. This corresponds to a source term: the construction of the term from the tree tells us that the the region of integration is restricted to $D_1 = D_1(t_{k+1},\xi)$ (it is an $F_{12}$-type source term, see \eqref{eq:decompF1}). Therefore
				\[ 
				|\mathcal{M}^j| \lesssim \frac{|\xi|\jap{\xi}^\mu }{t^{\rho}\jap{\xi_1}^\nu \jap{t^{1/3}\xi_2}^{(4/7)^-}}\mathbbm{1}_{D_1}. \]
				Proposition \ref{lem:estSregz} then implies
				\[
				\sup_{\xi,\alpha} \int |\mathcal{M}^j| \mathbbm{1}_{|\Phi-\alpha|<M}d\xi_1d\xi_2\lesssim (tM)^{\beta}t^{-1+\gamma}. \]
			\end{itemize}
			
			b) We now turn to the case $\Theta^j=\Psi$, and denote $\xi^j=\xi$ and $\Xi^j= \zeta $. There are only two possibilities:
			\begin{itemize}
				\item one child has a $K_0$-weight and the other has a $w$-weight, then
				\[
				|\mathcal{M}^j| \lesssim \frac{|\xi|\jap{\xi}^\mu }{t^{1/2}|\zeta|^{1/2}\jap{\xi-\zeta}^{\mu}}
				\]
				and Proposition \ref{prop:fre_psi} implies
				\[ \sup_{\xi,\alpha} \int |\mathcal{M}^j| \mathbbm{1}_{|\Psi-\alpha|<M}d\xi_1d\xi_2 \lesssim t^{-1/2} M^{1/2}. \] 
				\item one child has a $K_0$-weight and the other has a $z$ weight. Then
				\[ 
				|\mathcal{M}^j| \lesssim \frac{|\xi|\jap{\xi}^\mu }{t^{1/2+\rho}|\zeta|^{1/2}\jap{\xi-\zeta}^{\nu}}\mathbbm{1}_{D(t_{k+1})}
				\] 
				and Proposition \ref{prop:sourceSz} gives that
				\[ \sup_{\xi,\alpha} \int |\mathcal{M}^j| \mathbbm{1}_{|\Psi-\alpha|<M}d\xi_1d\xi_2 \lesssim t^{-1/2-\rho + \frac{\nu-\mu}{3}} M^{1/2}. \]
			\end{itemize}
			In all cases, in view of the definitions of $\beta$ \eqref{eq:defibeta} and $\rho$ \eqref{eq:defirho}, the bound $t^{-1+\beta} M^\beta$ holds.
			
			The second bound claim is an immediate application of Lemma \ref{lem:restricted_<>}.
		\end{proof}

	In the above lemma, we have bounded all but the last subtree of any given tree of length $J$, uniformly in time. However, our goal is to gain a power $t^\gamma$, in order to bootstrap the estimates for $w$. This gain is achieved at the last subtree:
	
	\begin{lem} [Final subtree]\label{lem:estfinalsubtree}
		Let $t\in [t_{k+1},t_k]$, and consider a term corresponding to a tree of size $J$.
		
		a) For resonant terms, the contribution of the terminal subtree is bounded by
		\begin{equation}\label{eq:resonanttree}
			\sup_{\xi^J,\alpha}\  \int_{{\Gamma}^{J}} |\mathcal{M}^J|\mathbbm{1}_{|\Theta^J-\alpha|<N_J}d\Xi^{J}\lesssim (tN_J)^\beta t^{-1+\gamma}.
		\end{equation}
		b)	For boundary terms, the corresponding bound is
		\begin{equation}\label{eq:bdrytree}
			\sup_{\xi^J,\alpha}\  \int_{{\Gamma}^{J}} \frac{|\mathcal{M}^J|}{|\Theta^j-\alpha| } \mathbbm{1}_{|\Theta^J-\alpha| >N_J}d\Xi^{J}\lesssim (tN_J)^{-1+\beta} t^{\gamma}.
		\end{equation}
		c)	For $F_2$ or derivative terms, the corresponding bound is
		\begin{equation}\label{eq:derivativetree}
			\sup_{\xi^J,\alpha}\  \int_{{\Gamma}^{J}} \frac{|\mathcal{M}^J|}{|\Theta^J-\alpha| } \mathbbm{1}_{|\Theta^j-\alpha| >N_J}d\Xi^{J} \lesssim (tN_J)^{-1+\beta} t^{-1+\gamma}.
		\end{equation}
	\end{lem}
	
	\begin{proof}
		The argument follows closely the steps in the previous proof. However, it is necessary to split between the source terms (where no child has a $w$-weight, and frequencies are restricted according to \eqref{eq:decompF1}) and the remaining ones. 
		
		a) We begin with the case of a resonant term. First consider the case $\Theta^j=\Phi$ and denote $\xi = \xi^J$ the frequency of the parent $J$, and $\xi_1$ and $\Xi^J = (\xi_1,\xi_2)$.

		If at least one child has a  $w$-weight, there are two possiblities:
		\begin{itemize}
			\item another child has a $w$ or $z$-weight: then as in the previous lemma, we don't take into account the last weight (which is $\gtrsim t^\rho$) and
			\begin{equation}\label{eq:M1}
				\mathcal{M}^J\lesssim \frac{|\xi|\jap{\xi}^\mu }{\jap{\xi_1}^\mu \jap{\xi_2}^\mu }t^{\gamma-2\rho}.
			\end{equation}
			We apply \eqref{eq:TalphaM_multiplierxi}, it yields
			\[
			\sup_{\xi^J,\alpha}\  \int_{{\Gamma}^{J}} {\mathcal{M}^J}\mathbbm{1}_{|\Theta^J-\alpha|< N_J}d\Xi^{J}\lesssim (N_J)^{\beta}t^{\gamma-2\rho} \lesssim (t N_J)^\beta t^{\gamma-2\rho-\beta}\lesssim (t N_J)^\beta t^{-1+\gamma}.
			\]
			(we used the fact that $-2\rho-\beta > -1$).
			\item another child has a $S_{\reg}$-weight: in this case, 
			\begin{equation}\label{eq:M2}
				\mathcal{M}^J\lesssim \frac{|\xi|\jap{\xi}^\mu }{\jap{\xi_1}^\mu \jap{t^{1/3}\xi_2}^{(4/7)^-}}t^\gamma\lesssim \frac{|\xi|\jap{\xi}^\mu }{\jap{\xi_1}^\mu \jap{\xi_2}^{\mu}}t^{\gamma-\mu/3}.
			\end{equation}
			The estimates follows from \eqref{eq:TalphaM_multiplierxi}.
		\end{itemize}
		Otherwise, if no child has a $w$-weight (which corresponds to a source term), we consider two alternatives:
		\begin{itemize}
			\item  two children have a $z$-weight, then the frequency set is restricted to $D_3$:
			\begin{equation}\label{eq:M3}
				\mathcal{M}^J\lesssim \frac{|\xi|\jap{\xi}^\mu }{\jap{\xi_1}^\nu \jap{\xi_2}^\nu t^{2\rho}}\mathbbm{1}_{D_3},
			\end{equation}
			and the estimates follow from Proposition \ref{prop:estPhi_source}.

			\item one child has a $z$-weight and another has a $S_{\reg}$-weight. Then the frequency set is restricted to $D_1$:
			\begin{equation}\label{eq:M4}
				\mathcal{M}^J\lesssim \frac{|\xi|\jap{\xi}^\mu }{t^\rho\jap{\xi_1}^\nu \jap{t^{1/3}\xi_2}^{(4/7)^-}}\mathbbm{1}_{D_1},
			\end{equation}
			and we apply Proposition \ref{lem:estSregz}.
		\end{itemize}
		
		We now consider the case $\Theta=\Psi$, and denote $\xi = \xi^J$ and $\Xi^J = \zeta$. There are two cases.
		\begin{itemize}
			\item One child has a $w$-weight and the other has a $K_0$-weight. Then
			\begin{equation}\label{eq:M5}
				\mathcal{M}^J\lesssim \frac{|\xi|\jap{\xi}^\mu }{\jap{\xi-\zeta}^\mu t^{1/2}|\zeta|^{1/2}}t^\gamma,
			\end{equation}
			and we use Proposition \ref{prop:fre_psi}.
			\item  One child has a $z$-weight and the other has a $K_0$-weight. Then the frequency set is restricted to $D(t_{k+1})$:
			\begin{equation}\label{eq:M6}
				\mathcal{M}^J\lesssim \frac{|\xi|\jap{\xi}^\mu }{\jap{\xi-\zeta}^\nu t^{1/2}|\zeta|^{1/2}}\mathbbm{1}_{D(t_{k+1})},
			\end{equation}
			and the estimate is then a consequence of Proposition \ref{prop:sourceSz}.
		\end{itemize}
		b) For boundary terms, we can bound using the same argument as in case a) to obtain, in all configurations and for any $M \ge 1$,
		\[  \sup_{\xi^J,\alpha}  \int_{{\Gamma}^{J}} {\mathcal{M}^J}\mathbbm{1}_{|\Theta^J-\alpha|< M}d\Xi^{J} \lesssim  M^\beta t^{\beta-1+\gamma}. \]
		Using Lemma \ref{lem:restricted_<>}, we infer that
		\[  \sup_{\xi^J,\alpha}v \int_{{\Gamma}^{J}} {\mathcal{M}^J}\mathbbm{1}_{|\Theta^J-\alpha| \ge N_J }d\Xi^{J} \lesssim N_J^{\beta-1} t^{\beta-1+\gamma} = (t N_J)^{-1+\beta} t^ \gamma. \]
		
		c) For derivative terms, the computations are the same as in b), taking into account that one of the children is a derivative, so that its frequency weight is given by \eqref{eq:frequencyweights_2}. Compared with the corresponding weight for the non derivated intervening function in \eqref{eq:frequencyweights}, we see that there is a $t$ extra factor. Therefore in this case
		\begin{equation} \label{est:final_dr}
			\sup_{\xi^J,\alpha}  \int_{{\Gamma}^{J}} {\mathcal{M}^J}\mathbbm{1}_{|\Theta^J-\alpha|< M}d\Xi^{j} \lesssim  M^\beta t^{\beta-2+\gamma}.
		\end{equation}
		Hence
		\begin{equation} \label{est:final_dr_2}
			\sup_{\xi^J,\alpha} \int_{{\Gamma}^{J}} {\mathcal{M}^J}\mathbbm{1}_{|\Theta^j-\alpha| \ge N_J }d\Xi^{J} \lesssim N_J^{\beta-1} t^{\beta-1+\gamma} = (t N_J)^{-1+\beta} t^{-1 + \gamma}.
		\end{equation}
		
		For an $F_2$ term, the frequency weight is the one of $w$, with an extra $t$ factor. So we obtain as for the derivative tree, the bounds \eqref{est:final_dr} and \eqref{est:final_dr_2}.
	\end{proof}
	
	\begin{nb}
		It is in the above proof that one sees the optimality in the frequency-restricted estimates. Indeed, it is necessary to match the power in $M$ with the singular powers of $t$ coming from the self-similar solution in order to recover the exact polynomial growth $t^\gamma$.
	\end{nb}

	\begin{cor}\label{cor:boundsinfr}
		Given $J\ge 1$ and $t \in [t_{k+1},t_k]$,
		\[
		\left\|\mathcal{N}_{\textnormal{res}}^{J,k}(t)\right\|_{L^\infty_\mu} + 	\|\mathcal{N}_{\textnormal{bd}}^{J,k}(t)\|_{L^\infty_\mu} + \|\mathcal{N}_{\textnormal{dt}}^{J,k}(t)\|_{L^\infty_\mu} +  \|\mathcal{N}_{F_2}^{J-1,k}(t)\|_{L^\infty_\mu} \lesssim Ct_k^\gamma \epsilon^{2J+1}, 
		\]
		where $C$ is the absolute constant defined in \eqref{eq:cj}.
	\end{cor}
	
	\begin{proof}
		First consider $\mathcal{N}_{\textnormal{res}}^{J,k}(t)$. It is the sum of at most $\text{Card}(AT_J)$ terms, each of which can be bounded in $L^\infty_\mu$ norm, due to \eqref{est:Nres}, by 
		\begin{align*}
			\MoveEqLeft \int_{t_{k+1}}^t \prod_{j=1}^{J-1} (sN_j)^{-1+\beta} \cdot (sN_J)^\beta s^{-1+\gamma} ds \cdot \| G \|_{L^\infty_{x,t}}  \lesssim  \prod_{j=1}^{J-1} c_j ^{\beta-1}\cdot c_J^\beta \cdot  \int_{t_{k+1}}^t s^{-1+\gamma} ds \cdot \e^{2J+1} \\
			& \lesssim  \prod_{j=1}^{J-1} c_j ^{\beta-1}\cdot c_J^\beta \cdot \e^{2J+1} t^\gamma.
		\end{align*}
		(We used \eqref{eq:Nj} and Lemmas \ref{lem:estsubtree} and \ref{lem:estfinalsubtree}).  From \eqref{eq:cj} and \eqref{est:card_AT}, we infer that
		\[ \left\| \mathcal{N}_{\textnormal{res}}^{J,k}(t) \right\|_{L^\infty_\mu}  \lesssim \text{Card}(AT_J)  \prod_{j=1}^{J-1} c_j ^{\beta-1}\cdot c_J^\beta \cdot  \e^{2J+1} t^\gamma \lesssim C \e^{2J+1} t^\gamma. \]
		
		For $\mathcal{N}_{\textnormal{bd}}^{J,k}(t)$, in view of \eqref{est:Nbd} (and using Lemmas \ref{lem:estsubtree} and \ref{lem:estfinalsubtree}, and \eqref{eq:Nj}), each of the terms can be bounded by
		\begin{align*}
			\MoveEqLeft \left( \sum_{s \in \{ t_{k+1}, t \} } \prod_{j=1}^{J-1} (sN_j)^{-1+\beta} \cdot (sN_J)^{-1+\beta} s^{\gamma} \right) \cdot \| G \|_{L^\infty_{x,t}}  \lesssim  \prod_{j=1}^{J-1} c_j ^{\beta-1}\cdot c_J^{-1+\beta}   t^\gamma \cdot \e^{2J+1} \\
			& \lesssim  \prod_{j=1}^{J-1} c_j ^{\beta-1}\cdot c_J^\beta \cdot \e^{2J+1} t^\gamma.
		\end{align*}
		(because $c_J \ge 1$). As before, there are at most $\text{Card}(AT_J)$ boundary terms, and so we obtain the same bound as for resonant terms: 
		\[ \left\| \mathcal{N}_{\textnormal{res}}^{J,k}(t) \right\|_{L^\infty_\mu} \lesssim C \epsilon^{2J+1} t^\gamma. \]
		
		For $\mathcal{N}_{\textnormal{dt}}^{J,k}(t)$, in view of \eqref{est:Ndt_NF2} (and using Lemmas \ref{lem:estsubtree} and \ref{lem:estfinalsubtree}, and \eqref{eq:Nj}), each of the terms can be bounded by
		\begin{align*}
			\MoveEqLeft \int_{t_{k+1}}^t \left( \prod_{j=1}^{J-1} (sN_j)^{-1+\beta} \cdot (sN_J)^{-1+\beta} s^{-1+\gamma} \right) ds \cdot \| G \|_{L^\infty_{x,t}}  \lesssim  \prod_{j=1}^{J-1} c_j ^{\beta-1}\cdot c_J^{-1+\beta} \cdot  \int_{t_{k+1}}^t s^{-1+\gamma} ds \cdot \e^{2J+1} \\
			& \lesssim  \prod_{j=1}^{J} c_j ^{\beta-1} \cdot \e^{2J+1} t^\gamma.
		\end{align*}
		(again we used $c_J \ge 1$) and the computations are the same as for $\mathcal{N}_{\textnormal{res}}^{J,k}$.
		
		For $\mathcal{N}_{F_2}^{J-1,k}(t)$, the computations are the same as for the derivative terms $\mathcal{N}_{\text{dt}}^{J-1,k}(t)$, except that the $F_2$ factor (multiplied by its weight) contributes $\e^3$ instead of $\e$, so that the final elementary tree (with parent numbered $J-1$) contributes $c_{J-1}^{\beta-1} \e^5 t^\gamma$. Therefore,
		\[ \|\mathcal{N}_{F_2}^{J-1,k}(t)\|_{L^\infty_\mu} \lesssim C \epsilon^{2(J-2)+1} \epsilon^5 t^\gamma \lesssim C  \epsilon^{2J+1} t^\gamma. \qedhere \]
	\end{proof}
	
	\begin{nb}\label{rmk:abs}
		As a consequence of the above analysis, each of the space-time integrals corresponding to well-behaved terms in the INFR expansion is absolutely convergent (independently of $n$).
	\end{nb}

	We recall that from Proposition \ref{prop:exist_aprox_0}, $w = w_n$ is defined in a suitable space on a time interval $(0,T_0)$ which \textit{a priori} depends on $n$. The analysis performed above translates directly to $w_n$, so that if $k \in \m N$ and $t \in [t_{k+1},t_k]$,
\begin{multline} \label{eq:wn_INFR}
w_n(t)=w_n(t_{k+1})+ \chi_n^2 \left( \sum_{j=1}^J \left(\mathcal{N}_{\text{res},n}^{j,k}(t)+\mathcal{N}_{\text{bd},n}^{j,k} (t)+\mathcal{N}_{\text{dt},n}^{j,k}(\tau)+\mathcal{N}_{F_2,n}^{j,k}(\tau) \right)  \right. \\
\left. \vphantom{\sum_j^J} + \cal N_{F_2,n}^{J+1,k}(t) + \mathcal{R}^{J+1,k}_n(t) \right), 
\end{multline}
 where the bounds  of Corollary \ref{cor:boundsinfr} hold:
\begin{equation} \label{est:wn_INFR} 
\| \mathcal{N}_{\text{res},n}^{j,k}(t) \|_{L^\infty_\mu} + \| \mathcal{N}_{\text{bd},n}^{j,k} (t) \|_{L^\infty_\mu} + \| \mathcal{N}_{\text{dt},n}^{j,k}(t) \|_{L^\infty_\mu} +\| \mathcal{N}_{F_2,n}^{j,k}(t) \|_{L^\infty_\mu} \lesssim t_k^\gamma \e^{2j+1}.
\end{equation}
	We have now all the tools to prove existence and bound on $w_n$ independently of $n$ (which is why we write the index $n$ in the next statement).
	
	\begin{prop}[A priori estimate for $w_n$]\label{prop:estimatew}
	Suppose that, for some $T<\min\{T_0^n,T_0'\}$,
	\[ \forall t\in[0,T], \quad 
	\|\chi_n^{-1} w_n(t)\|_{L^\infty_\mu} \le \epsilon t^\gamma.
	\]
	Then, given $J,k\ge 0$,
	\[ \forall t\in[t_{k+1},t_k], \quad \| \mathcal{R}^{J+1,k}_n(t) \|_{L^\infty_\mu} \lesssim_n t^{-1/2} \epsilon^{2J}.
	\]
	As a consequence,
	\[ \forall t\in[0,T], \quad 
	\|\chi_n^{-1} w_n(t)\|_{L^\infty_\mu} \lesssim \epsilon^3 t^\gamma
	\]
	where the implicit constant does not depend on $n$.
		

	\end{prop}
	
	\begin{proof}
		a) Observe that $\mathcal{R}^{J+1,k}_n(t)$ has exactly the same structure as a boundary term of step $J$, except one terminal node is colored $\partial_t w$ instead of $w$ and one also integrates in time. By \eqref{eq:dtwbounded},
		\[
		t^{-\gamma}\|\partial_t w(t)\|_{L^\infty_\mu}\lesssim_{n,S,z} \frac{1}{t^{1/2+\gamma}}.
		\]
		Proceeding as in Corollary \ref{cor:boundsinfr}, we get
		\[
		\|\mathcal{R}^{J+1,k}_n (t)\|_{L^\infty_\mu}\lesssim_{n,S,z}  \int_{t_{k+1}}^t s^{-3/2}ds\cdot \epsilon^{2J} \lesssim_{n,S,z} t^{-1/2}\epsilon^{2J}, \]
		as claimed.

		b) We now consider the $w_n$ estimate, where the bound is independent on $n$. 
		Let $\kappa \in \m N$, and $\tau \in[t_{\kappa+1},t_\kappa]$. Due to equation \eqref{eq:wn_INFR} with the bounds \eqref{est:wn_INFR} and a), we can estimate
		\begin{align*}
			\MoveEqLeft \|\chi_n^{-1}({w}_n(\tau)-w_n(t_{\kappa+1}))\|_{L^\infty_\mu} \\
			&\lesssim \sum_{j=1}^\infty \left( \|\mathcal{N}_{\text{res},n}^{j,k}(\tau)\|_{L^\infty_\mu} + \|\mathcal{N}_{\text{bd},n}^{j,k}(\tau)\|_{L^\infty_\mu} + \|\mathcal{N}_{\text{dt},n}^{j,k}(\tau) \|_{L^\infty_\mu} + \|\mathcal{N}_{F_2,n}^{j-1,k}(\tau)\|_{L^\infty_\mu} \right) \\
			&\lesssim  \sum_{j= 1}^J t_\kappa^\gamma \epsilon^{2j+1} + C_n  \tau^{-1/2} \epsilon^{2J} \lesssim \epsilon^3 t_\kappa^\gamma  + C_n  t_\kappa^{-1/2} \epsilon^{2J}.
		\end{align*}
		Now, let $J \to +\infty$ (while $n$, $\kappa$ and $\tau$ are fixed) to infer
		\[ \|\chi_n^{-1}({w}_n(\tau)-w_n(t_{\kappa+1}))\|_{L^\infty_\mu} \lesssim \epsilon^3 t_k^\gamma. \]
		Finally, fix $t \in (0,T_0)$ and let $k \ge 1$ such that $t \in [t_{k+1},t_k]$. The above estimate holds for all $\kappa \ge k+1$ and for $\tau = t$ or $\tau = t_{\kappa}$, hence
		\begin{align*}
			\|\chi_n^{-1}w_n(t)\|_{L^\infty_\mu} & \lesssim \|\chi_n^{-1}(w_n(t)-w_n(t_{k+1}))\|_{L^\infty_\mu} + \sum_{\kappa \ge k+1} \|\chi_n^{-1}(w_n(t_{\kappa})-w_n(t_{\kappa+1}))\|_{L^\infty_\mu}\\
			& \lesssim \epsilon^3\left(t_k^\gamma + \sum_{\kappa \ge k+1}t_{\kappa}^\gamma\right)  \lesssim \epsilon^3 t_k^\gamma \lesssim \epsilon^3 t^\gamma. \qedhere
		\end{align*}
	\end{proof}
	
	\begin{nb}
		The sucessful derivation of low-regularity estimates through the INFR hinges on the first bound on $\mathcal{R}$. Indeed, even though the remainder $\mathcal{R}^{J+1}_n(t)$ grows as $n\to \infty$ (or as $t \to 0$), one can first take the limit in $J$ to drop the remainder and arrive, \emph{for each fixed} $n$ and $t>0$, to an infinite sum of well-behaved terms.
	\end{nb}
	
	\begin{nb} \label{nb:pb_INFR_L2}
		The problem with the strategy presented in \cite{KOY20} is the reduction to a trilinear operator estimate (see the proof of Lemma 3.10 therein). Indeed, it is not true that a general multilinear term in the infinite expansion may be written as a successive composition of a trilinear operator, as one cannot separate properly the dependence in the various frequencies. The difference in our approach is the replacement of an \textit{operator bound} with a \textit{multiplier bound} - the frequency-restricted estimates. These multiplier bounds, when inserted in the infinite equation, reveal a lower triangular dependence structure (see \eqref{eq:uppertriang}). As a result, the estimate for a term at step $J$ is reduced to the product of $J$ frequency-restricted estimates and the \textit{a priori} estimates follow.
	\end{nb}
	
	\subsection{A priori estimates on \texorpdfstring{$\Lambda w_n$}{Lambda wn} and uniform time of existence}
	
	Recall the scaling operator \eqref{def:dilation_op}
	\begin{equation}
		\Lambda= \partial_\xi - \frac{3t}{\xi}\partial_t
	\end{equation}
	Due to the self-similar structure of $S$, the homogeneity of $\Phi$ and $\Lambda$ being an order-one operator, we have the identities\footnote{For the identity involving $\Lambda N$, one can either work in physical space (cf. \cite{CC22}), or use the Euler identity $\xi D_\xi \Phi = 3\Phi$ on $\Gamma_\xi$.}:
	\begin{gather*}
		 \xi \Lambda S =0, \quad \Lambda z= \partial_\xi z, \quad \text{and} \\
		[\partial_t, \Lambda] u = - \frac{3}{\xi} \partial_t u, \quad  \Lambda N[u] = \frac{3}{\xi}N[\xi\Lambda u, u, u] - \frac{3}{\xi}N[u].
	\end{gather*}
	Hence, for the solution $w_n\in Y_{n,T}$ to \eqref{eq:approx} given by Proposition \ref{prop:exist_aprox_1}, 
	\begin{equation}\label{eq:Lambdaw}
		\begin{cases}
			\ds \partial_t \Lambda w_n = \frac{3\chi_n^2}{\xi}N[\xi\Lambda (w_n+z),v_n,v_n]+2\chi_n\chi_n'(N[v_n]-N[S]), \quad \text{with } v_n=S+z_n+w_n, \\
			\Lambda w_n (t=0)=0.
		\end{cases}
	\end{equation} 
	As we have done for $w_n$, we want to perform a bootrap argument for $\Lambda w_n$,  starting from \begin{equation}
		\label{eq:bootstrapLw}
		\|\Lambda w_n(t)\|_{L^\infty_\mu}\lesssim \epsilon t^{\gamma}
	\end{equation}
	and recovering, through the INFR,
	\[ \|\Lambda w_n(t)\|_{L^\infty_\mu}\lesssim \epsilon^3 t^{\gamma}. \]
	When compared with the normal form algorithm for $w$, we see that few changes arise:
	\begin{enumerate}
		\item the set of intervening functions is now $w,z,S_0,S_{\text{reg}}, K_0, \Lambda w$ and $\partial_\xi z$.
		\item As done in \eqref{eq:defirho} and \eqref{bd:z}, define $T_1' = \epsilon\| z \|_{W^{1,\infty}_\nu}^{-1/\rho}$, so that
		\[
		\|t^{\rho}\jap{\xi}^{\nu}\partial_\xi z\|_{L^\infty_{ \xi}}\lesssim \epsilon,\quad \forall t<T_1'.
		\]
		\item the frequency weights for terminal nodes associated with an intervening function $f$ are
\[
		\begin{cases}
			t^{-\gamma} \jap{\zeta}^{\mu} &\quad \text{if }f=w \text{ or } \Lambda w, \\
			t^\rho \jap{\zeta}^{\nu} &\quad \text{if }f=z\text{ or } \partial_\xi z, \\
			1 &\quad \text{if } f=S_0,\\
			\jap{t^{1/3}\zeta}^{(4/7)^-} & \quad \text{if }f=S_{\reg}, \\
			t^{1/2}|\zeta|^{1/2},&\quad \text{if } f=K_0.
		\end{cases}
\]
		\item when distributing time derivatives, one may find either $\partial_tw$ or $\partial_t \Lambda w$, which one then replaces with the corresponding equation.
	\end{enumerate}
	
	In other words, one may view the problem as a coupled system for $(w,\Lambda w)$. The nonlinear terms in both equations have exactly the same algebraic structure, allowing for the application of the algorithm described in Section \ref{sec:INFR_algorithm}. Handling $\Lambda w$ (resp. $\partial_\xi z$) as if it were $w$ (resp. $z$), by inspection of  Section \ref{sec:INFR_bounds}, the \textit{a priori} bounds derived for Proposition \ref{prop:estimatew} hold for the expansion of the equation for $\Lambda w$, and we have the following result:

	\begin{prop}[A priori estimate for $\Lambda w$]\label{prop:estimateLambda}
		Suppose that, for some $T<\min\{T_1^n,T_1'\}$,
	\[  \forall t\in[0,T], \quad
	\|\chi_n^{-1} w_n(t)\|_{L^\infty_\mu} +   \|\chi_n^{-1} \Lambda w_n(t)\|_{L^\infty_\mu} \le \epsilon t^\gamma.
	\]
	Then
	\[ \forall t \in [0,T], \quad
	\|\chi_n^{-1} w_n(t)\|_{L^\infty_\mu} + \|\chi_n^{-1} \Lambda w_n(t)\|_{L^\infty_\mu} \lesssim \epsilon^3 t^\gamma,\quad \forall t\in[0,T],
	\]
	where the implicit constant does not depend on $n$.

	\end{prop}

	Having the \textit{a priori} bounds for $w$ (Proposition \ref{prop:estimatew}), $\Lambda w$ (Proposition \ref{prop:estimateLambda}) and $\partial_t w$ (Lemma \ref{lem:controldt}), we are now in position of closing the bootstrap argument for $w$. We recall the definition of the $Y_n$ norm \eqref{def:Yn}.
	
\begin{prop}[Uniform local time of existence]\label{prop:uniform}
Under the conditions of Theorem \ref{thm:exist}, there exists $T=T(\epsilon, \|z\|_{W^{1,\infty}_\nu})>0$, independent of $n$, such that the solution $w_n$ defined in Proposition \ref{prop:exist_aprox_1} is defined up to a time $T_1(S,z,n) > T$ and
\begin{equation} \label{est:ubd_wn_Yn} 
\forall t \in [0,T], \quad \| w_n \|_{Y_n(t)}  \lesssim \epsilon^3 t^\gamma.
\end{equation}
	\end{prop}

	\begin{proof}
	Let 
\[  \gamma < \vartheta < \frac{1}{3} \min \left( \mu, \nu-\mu \right) \quad \text{and} \quad T = \min \left(T_1, T_1', \left( \frac{ \e}{ \| z \|_{W^{1,\infty}_\nu}} \right)^{\frac{3}{\vartheta- \gamma}} \right). \]
Recall $T_1'\le T_0'\le 1 $ are given by Proposition \ref{prop:estimatew} and Proposition \ref{prop:estimateLambda}), and $T$ is chosen so that
\begin{equation} \label{est:T_z_e}
T^{\vartheta -\gamma} (\epsilon^2 + \| z \|_{W^{1,\infty}_\nu}^2) \| z \|_{W^{1,\infty}_\nu} \le 2 \epsilon^3.
\end{equation}
 Fix $n \in \m N$.  For $\tau \in [0,T)$ define
\[
		A(\tau)=\sup_{t \in [0,\tau]} t^{-\gamma} \| {w}_n\|_{Y_n(t)} ,  \quad
\text{and} \quad 		\tau^*=\sup\{\tau \in [0,T) : A(\tau)\le \epsilon\}.
		\]
		By Proposition \ref{prop:exist_aprox_1}, $\tau^*>0$.  Let $t \in [0,\tau^*)$. Then, since $\chi_n \le 1$,
\[ \| w_n(t) \|_{W^{1,\infty}_\mu} \lesssim \| \chi^{-1}_n w_n(t) \|_{L^\infty_\mu} + \| \chi^{-1}_n \partial_\xi w_n \|_{L^\infty_\mu} \lesssim \epsilon t^\gamma \lesssim \epsilon. \]

We claim that
\begin{equation} \label{est:ubd_N_1}
 \left\| \frac{t}{\xi} (N[v_n]-N[S_n]) \right\|_{L^\infty_{\mu}} \lesssim \epsilon^3 t^\gamma.
 \end{equation}
We use  Lemma \ref{lem:controldt} to control the trilinear terms obtained when expanding $N[v_n]-N[S_n]$, distinguishing two cases:
\begin{itemize}
\item If there is one factor $z_n$, it takes the place of $h$ and we choose $a=b' = \nu$ and $a' = \mu$ so that $\vartheta \le \frac{1}{3} \min(a,b'-a')$. Keeping in mind \eqref{est:T_z_e}, the term is then bounded by
\begin{align*}
\MoveEqLeft t^{\vartheta}\left( \|S(t)\|_{W^{1,\infty}_{0,\mu}(\R\setminus\{0\})}^2 + \| z_n(t) \|_{W^{1,\infty}_{0,\mu}}^2 + \| w_n(t) \|_{W^{1,\infty}_{0,\mu}}^2\right) \| z_n \|_{W^{1,\infty}_\nu} \\
& \lesssim  t^{\vartheta}\left(\|S(t)\|_{W^{1,\infty}_{0,1}(\R\setminus\{0\})}^2 + \| z \|_{W^{1,\infty}_{\nu}}^2 + \| w_n(t)\|_{W^{1,\infty}_\mu}^2\right) \| z \|_{W^{1,\infty}_\nu} \\
& \le t^\vartheta  \| z \|_{W^{1,\infty}_\nu}  ( \epsilon^2 + \| z \|_{W^{1,\infty}_\nu} ^2) \lesssim \epsilon^3 t^\gamma.
\end{align*}
\item If there is no factor $z_n$, there is at least one $w_n$ which takes the place of $h$; we choose $a=b'=\mu$ and $a'= \mu - 3\vartheta>0$ so that $\vartheta \le \frac{1}{3} \min(a,b'-a')$. The term is then bounded by
\begin{align*}
\MoveEqLeft t^{\vartheta}\left( \|S(t)\|_{W^{1,\infty}_{0,\mu/2}(\R\setminus\{0\})}^2 + \| w_n(t) \|_{W^{1,\infty}_{0,\mu}}^2\right) \| w_n (t) \|_{W^{1,\infty}_\mu} \\
& \lesssim  t^{\vartheta}\left(\|S(t)\|_{W^{1,\infty}_{0,1}(\R\setminus\{0\})}^2 + \| w_n(t)\|_{W^{1,\infty}_\mu}^2\right) \| w_n(t) \|_{W^{1,\infty}_\mu} \\
& \le \epsilon^3 t^\vartheta \lesssim \epsilon^3 t^\gamma.
\end{align*}
\end{itemize}
This proves \eqref{est:ubd_N_1}. Hence, for $t \in [0,\tau^*)$, there holds
\[ \left\| \chi_n^{-1}\frac{t\partial_t {w}_n(t)}{|\xi|} \right\|_{L^\infty_{\mu}} \le \left\|\frac{t}{\xi}(N[v_n]-N[S_n])\right\|_{L^\infty_{\mu}} \lesssim \epsilon^3 t^\gamma, \]
and, using as well the bound from Proposition \ref{prop:estimateLambda},
\[ \left\| \chi_n^{-1}\partial_\xi{w}_n(t) \right\|_{L^\infty_{\mu}} \lesssim \left\| \chi_n^{-1}\frac{t\partial_t {w}_n(t)}{|\xi|} \right\|_{L^\infty_{\mu}} + \|\chi_n^{-1}\Lambda{w}_n(t)\|_{L^\infty_{\mu}} \lesssim  \epsilon^3t^\gamma. \]
This proves that
\[ \forall t \in [0,\tau^*), \quad \|\chi_n^{-1}{w}_n(t)\|_{L_\mu^\infty}+\left\| \chi_n^{-1}\frac{t\partial_t {w}_n(t)}{|\xi|} \right\|_{L^\infty_{\mu}} + \left\| \chi_n^{-1}\partial_\xi{w}_n(t) \right\|_{L^\infty_{\mu}} \le C_1 \epsilon^3 t^\gamma. \]
(the absolute $C_1$ does not depend on $n$). Choose $\epsilon >0$ so small that $C_1 \epsilon^2 \le 1/2$. A continuity argument gives that $\tau^* = T$.
Due to the blow-up criterion of Proposition \eqref{prop:estimateLambda}, we infer that 
\[ T_1(S,z,n) > T = \min\left( \frac{\e}{\| z \|_{W^{1,\infty}_{\nu}}^{1/\rho}},  \left( \frac{ \e}{ \| z \|_{W^{1,\infty}_\nu}} \right)^{3/\gamma} \right), \]
and  \eqref{est:ubd_wn_Yn} holds.
\end{proof}

\section{Proof of Theorem \ref{thm:exist}} \label{sec:proof_th1}
	
	\subsection{Existence} Before we move to the proof of the existence of $w$, we need some very mild control in physical space for the linear evolutions whose Fourier transform lies in $W^{1,\infty}_{\mu}$.
	
	\begin{lem}[Boundedness in physical space]\label{lem:physical}
		
		Let $\mu >0$ and $w\in W^{1,\infty}_{0,\mu}$. Define $v=e^{-t\partial_x^3}w^\vee$. 
		Then, for $0<t_0<t_1$ and $R\ge 1$,
		\begin{gather*}
		\forall t \in [t_0,t_1], \quad \|v(t)\|_{W^{1,\infty}([-R,R])} \lesssim_{t_0,t_1} R^{2} \|w\|_{W^{1,\infty}_{0,\mu}}.
		\end{gather*}
	\end{lem}
	\begin{proof}
		By linearity, we can assume $\| w \|_{W^{1,\infty}_{0,\mu}} =1$. Then we write
		\begin{align*}
			v(t,x)&=\frac{1}{2\pi}\int_\R e^{ix\xi + it\xi^3} w(\xi) d\xi = \frac{1}{2\pi}\int_\R \xi e^{ix\xi + it\xi^3}\partial_\xi\left(\frac{w(\xi)}{1+i\xi(x+3t\xi^2)}\right) d\xi\\&=\frac{1}{2\pi}\int_\R \xi e^{ix\xi + it\xi^3}\left(\frac{i(x+9t\xi^2)w(\xi)}{(1+i\xi(x+3t\xi^2))^2}+\frac{\partial_\xi w(\xi)}{1+i\xi(x+3t\xi^2)}\right) d\xi.
		\end{align*}
		Since $w$ and $\partial_\xi w$ are bounded by $1$,
		\begin{align*}
			|v(t,x)|&\lesssim \int_\R  \left|\frac{(x+9t\xi^2)\xi}{(1+i\xi(x+3t\xi^2))^2}\right|+\left|\frac{\xi}{1+i\xi(x+3t\xi^2)}\right| d\xi
		\end{align*}
We split the integral between different regions. In the region $|\xi| \le 1$,  the integral is $O(1)$.

In the region $|x+3t\xi^2| \le t_0 \xi^2$, it is bounded by
\[ \int_{|x| \ge 2t \xi^2} (1+|x|+t_1|\xi|^2) |\xi| d\xi \lesssim_{t_1} 1+| x|^{3/2}. \]
In the region $|x+3t\xi^2| \ge t_0 \xi^2$, we bound it by
\[ \int_{|\xi| \ge 1} \frac{(1+|x|+t_1 \xi^2)|\xi|}{t_0^2 \xi^6} d\xi \lesssim_{t_0,t_1} 1+|x|. \]
For the derivative, we have
		\begin{align*}
			\partial_xv(t,x)&=\frac{1}{2\pi}\int_\R i\xi e^{ix\xi + it\xi^3} w(\xi) d\xi = \frac{1}{2\pi}\int_\R \xi e^{ix\xi + it\xi^3}\partial_\xi\left(\frac{i\xi w(\xi)}{1+i\xi(x+3t\xi^2)}\right) d\xi\\&=\frac{1}{2\pi}\int_\R \xi e^{ix\xi + it\xi^3}\left(\frac{i(x+9t\xi^2)i\xi w(\xi)}{(1+i\xi(x+3t\xi^2))^2}+\frac{i\xi\partial_\xi w(\xi) + iw(\xi)}{1+i\xi(x+3t\xi^2)}\right) d\xi.
		\end{align*}
		Thus
		\begin{align*}
			|\partial_xv(t,x)|&\lesssim \int_\R   \left|\frac{\xi^2(x+9t\xi^2) \jap{\xi}^{-\mu}}{(1+i\xi(x+3t\xi^2))^2}\right|+\left|\frac{ |\xi|^2\jap{\xi}^{-\mu} + |\xi|\jap{\xi}^{-\mu}}{1+i\xi(x+3t\xi^2)}\right| d\xi.
		\end{align*}
We use the same region decomposition as earlier, and we bound the integral in a similar fashion, with an extra $|x|^{1/2}$: this yields the desired estimate.
		\end{proof}
		
We can now prove the existence part of Theorem \ref{thm:exist}.

		\begin{prop}[Existence of $w$ at $t=0$]\label{prop:exist}
			Under the conditions of Theorem \ref{thm:exist}, there exist a time $T=$ $T(\epsilon, \mu,\nu, \|z\|_{W^{1,\infty}_{\nu}})>0$ and $w\in \q C_b([0,T] \times \m R)$ such that  $u=e^{-t\partial_x^3}(S+z+w)^\vee$ is a distributional solution of \eqref{mKdV} on $(0,T)\times \R$ and
\[ \forall t \in [0,T], \quad
			\|w(t)\|_{W^{1,\infty}_{\mu}} + \left\| \frac{t\partial_t w(t)}{\xi}\right\|_{L^\infty_\mu}\lesssim \epsilon^3 t^\gamma. \]
		\end{prop}

		The proof follows similar arguments to that of \cite[Proposition 13]{CCV20}, with some changes in the details which we provide for the convenience of the reader.
		
		\begin{proof}
			\emph{Step 1. Convergence of the approximating sequence.} For each $n\in\mathbb{N}$, let $w_n$ be the solution of \eqref{eq:approx}. By Proposition  \ref{prop:uniform}, there exists $T= T(\e, \|z\|_{W^{1,\infty}_{\nu}})$ such that $w_n$ is defined on $[0,T]$ and, as $\chi_n \le 1$,
\[ 
\forall n \in \m N, \forall t \in [0,T], \quad
			\|{w}_n(t)\|_{W^{1,\infty}_\mu}+\left\| \frac{t\partial_t {w}_n(t)}{|\xi|} \right\|_{L^\infty_{\mu}}  \lesssim \epsilon^3t^\gamma. 
\]
In particular, given $R>0$, $(w_n)_n$ is equibounded and equicontinuous on $[0,T]\times [-R,R]$. By Ascoli-Arzelà theorem (and a diagonal extraction), there exists $w\in W^{1,\infty}_{\text{loc}}((0,T]\times \R)$ such that (up to a subsequence)
\[
\forall R >0, \quad	w_n \to w \quad \text{in } \q C ([0,T]\times [-R,R]).
\]
Moreover,
\[ \forall  t \in [0,T], \quad
			\|{w}(t)\|_{W^{1,\infty}_\mu }+\left\| \frac{t\partial_t {w}(t)}{|\xi|} \right\|_{L^\infty_{\mu}} \le \liminf_n \left(\|{w}_n(t)\|_{W^{1,\infty}_\mu }+\left\|\frac{t\partial_t {w}_n(t)}{|\xi|} \right\|_{L^\infty_{\mu}} \right) \lesssim \epsilon^3t^\gamma.
\]
The uniform decay in $\xi$ implies $w_n\to w$ in\footnote{In the sense that for all $\varphi \in \cal S(\m R)$, $\sup_{t \in [0,T]} |\langle w_n (t) - w(t), \varphi \rangle| \to 0$ as $n \to +\infty$.} $\q C([0,T], \mathcal{S}'(\R))$. As such, if we define, for $t \in [0, T]$ and $n \in \m N$,
\[
			u_n(t) := e^{-t\partial_x^3}\left(S(t)+z_n(t)+w_n(t)\right)^\vee \quad \text{and} \quad  u(t):=e^{-t\partial_x^3}\left(S(t)+z+w(t)\right)^\vee, 
\]
we can conclude that $u_n \to u$ in $\q C((0,T),\mathcal{S}'(\R))$. We need to improve the convergence in the physical space in order to pass to the limit in the nonlinear term. For $t\in(0,T)$ and $R\ge 0$ fixed, Lemma \ref{lem:physical} implies that $(u_n(t))_n$ is equibounded and equicontinuous on $[-R,R]$. By Ascoli-Arzelà, there exists a subsequence $(u_{n_k}(t))_k$ (which may depend on the chosen time $t$) such that
\[
u_{n_k}(t)\to h(t) \quad \text{ in }L^\infty_{\text{loc}}(\R).
\]
Since, by Lemma \ref{lem:physical}, $\|\jap{x}^{-2}u_n(t)\|_{L^\infty}\lesssim 1$, the convergence must hold in $\mathcal{S}'(\R)$ and thus $h(t)=u(t)$. As the limit $h$ is uniquely determined, we can conclude that the whole sequence converges to $h=u$, and so
\begin{equation}\label{eq:convergencia}
u_n(t)\to u(t) \quad \mbox{in }L^\infty(\jap{x}^{-2^+}dx).
\end{equation}

			
\emph{Step 2. \eqref{mKdV} is satisfied in distributional sense.} Define for $v\in\mathcal{S}'(\R)$, $\Pi_n v = \left(\chi_n^2(\xi)\hat{v}(\xi)\right)^\vee$: this is a continuous operator on $\cal S'(\m R)$. Then we may rewrite \eqref{eq:approx} as
			\begin{equation}\label{eq:eqfisico}
				(\partial_t+\partial_x^3)u_n = \Pi_n \partial_x (u_n^3).
			\end{equation}
			As $u_n \to u$ in $\q D'((0,T) \times \m R)$, we have $(\partial_t+\partial_x^3)u_n  \to (\partial_t+\partial_x^3)u$ in $\q D'((0,T) \times \m R)$. For the nonlinear term, first observe that,
			by \eqref{eq:convergencia}, 
\[ \forall t \in [t_0,T], \quad u_n^3(t) \to u^3(t) \quad \text{in } L^\infty(\jap{x}^{-6^+}dx) \]
and we have the bound uniform in $t \in [t_0,T]$, 
\[ \| \langle x \rangle^{-6} u_n^3(t) \|_{L^\infty},  \| \langle x \rangle^{-6} u^3(t) \|_{L^\infty} \lesssim 1. \]

Fix $\phi \in \q C^\infty_c((0,T) \times \m R)$, with $\Supp \phi \subset [t_0,T] \times \m R$ for some $t_0>0$. As $\Pi_n \phi(t) \to \phi(t)$ in $\cal S(\m R)$ for all $t \in [t_0,T]$, and that for sufficiently large integer $a \ge 8$, one has the uniform (in $n$) bound
\[ \forall \varphi \in \cal S(\m R), \quad \|  \langle x \rangle^{6} \Pi_n \partial_x \varphi \|_{L^1} \lesssim \| \langle x \rangle^a (1-\partial_x^2)^a \varphi \|_{L^\infty}, \]
we can conclude by dominated convergence that
\[  \langle \Pi_n \partial_x (u_n^3), \phi \rangle_{\mathcal{D}'\times \mathcal{D}} = \int_{t_0}^T \int_{\m R} u_n^3 (t) \Pi_n \partial_x \phi(t) dxdt \to \int_{t_0}^T \int_{\m R} u^3 (t) \partial_x \phi(t) dxdt =  \langle  \partial_x (u^3), \phi \rangle_{\mathcal{D}'\times \mathcal{D}}, \]
and so $\Pi_n \partial_x (u_n^3) \to \partial_x(u^3)$ in $\q D'((0,T) \times \m R)$. Taking the limit $n \to +\infty$ in \eqref{eq:eqfisico}, we see that
$u$ solves the (mKdV) equation in the distributional sense.
		\end{proof}
	
	\begin{nb}\label{nb:Linfty}
		Through the INFR procedure, one can actually prove that $(w_n)_{n\in\mathbb{N}}$ is a Cauchy sequence in $L^\infty_\mu$. This allows for the extraction of a limit $w$ and could allow us to lower the assumptions on the perturbation, not requiring anything on the derivatives. However, there are two caveats: first, the control of $F_2$ depends on integration by parts in frequency, which uses information on $\partial_\xi z$; second, the proof that the limit $w$ satisfies \eqref{mKdV} uses the control on physical space given by Lemma \ref{lem:physical}, which depends once again on a control of $\partial_\xi w$.
	\end{nb}
	
		\begin{nb} \label{rem:Picard}
		Another possibility to construct the solution $w$ of Theorem \ref{thm:exist} could be to use a Picard iteration scheme. Setting
\[
		\tilde w_0=0,\quad \text{for } n \in \m N, \quad \partial_t \tilde w_{n+1}=N[S+z_n+ \tilde w_n]-N[S],
\]
		the bounds of the INFR would show that $(\tilde w_n)_n$ is a Cauchy sequence in weighted $L^\infty_\mu$ spaces. However, we do not know how to justify the normal form algorithm, as it requires a control on the time derivative of $\tilde w_n$, see the proof of Proposition \ref{prop:nfevale}.
		
		On the other hand, one could define the Picard iteration for the expanded equation \eqref{eq:nfe2} directly. In this case, the bounds given through the INFR do show that the Picard iterations form a Cauchy sequence\footnote{One can even apply a fixed-point argument in weighted $L^\infty$ spaces.}, producing a limit object $w'$. The problem is to identify $w'$ as a distributional solution to \eqref{mKdV}. This is in fact the case \textit{a posteriori}: since the solution given through Theorem \ref{thm:exist} also satisfies \eqref{eq:nfe2}, it must coincide with $w'$.
		
		In our approach, using the approximating equation \eqref{eq:approx}, the normal form reduction can be rigorously justified to produce the required \textit{a priori} bounds for each fixed $n\in\mathbb{N}$, circunventing the obstructions explained above. We pay the price later on: instead of proving that the approximating sequence $w_n$ is a Cauchy sequence in $W^{1,\infty}_\mu$, we prove only that it is bounded in this space and then resort to Ascoli-Arzelà theorem to show the weaker convergence in $L^\infty_\mu$. Fortunately, this is sufficient to prove that the limit satisfies \eqref{mKdV}.
		
		In conclusion, we have two equations, \eqref{mKdV} (the one we are trying to solve but which behaves badly) and \eqref{eq:nfe2} (which is better behaved, but equivalent only under additional assumptions). Our approach permits us to work on both equations at the same time, while the Picard iteration scheme only works on \eqref{eq:nfe2}.
	\end{nb}
	
\subsection{Uniqueness and continuous dependence}\label{sec:uniq}

Given $S$ and $z$ of size $\epsilon$ in $W^{1,\infty}_{0,1}$ and $W^{1,\infty}_\nu$, respectively, let $w \in L^\infty([0,T],W^{1,\infty}_\mu)$ be any function that satisfy the condition of Theorem \ref{thm:exist}: it is a solution to \eqref{eq:eq_w} such that 
\[ \forall t \in [0,T], \quad \| w(t) \|_{W^{1,\infty}_\mu} + \left\|\frac{t}{\xi} \partial_tw(t)  \right\|_{L^\infty_\mu}\lesssim \e^3 t^\gamma. \]
		
As it was done in \cite{KOY20}, uniqueness and continuous dependence in Theorem \ref{thm:exist} will follow from proving that $w$ actually satisfies the infinite normal form equation
\begin{equation}\label{eq:nfe2}
	w(t)=\sum_{j=1}^\infty \left(\mathcal{N}_{\text{res}}^j(t)+\mathcal{N}_{\text{bd}}^j(t)+\mathcal{N}_{\text{dt}}^j(t)+\mathcal{N}_{F_2}^{j-1}(t) \right).
\end{equation}
To do this, we need to prove that, for any $J\ge1$ and $t\in [t_{k+1},t_k]$ (where $t_k = T/2^k$ and $k \in \m N$)
\begin{equation}\label{eq:nfeJ}
	w(t)=w(t_{k+1})+\sum_{j=1}^J \left(\mathcal{N}_{\text{res}}^{j,k}(t)+\mathcal{N}_{\text{bd}}^{j,k}(t)+\mathcal{N}_{\text{dt}}^{j,k}(t)+\mathcal{N}_{F_2}^{j-1,k}(t) \right) +  \mathcal{N}_{F_2}^{J,k}(t) + \mathcal{R}^{J+1,k}(t).
\end{equation}
and that, when $J\to \infty$, $\mathcal{N}_{F_2}^{J,k}(t)$ and  $\mathcal{R}^{J+1,k}(t)$ tend to zero in some appropriate norm.

\begin{lem}\label{lem:restodecai}
For any $J\ge 1$ and $k \in \m N$,
\begin{equation}
	\forall t \in [t_{k+1},t_k], \quad 	\left\| \frac{1}{|\xi|} \left( \mathcal{N}_{F_2}^{J,k}(t) + \mathcal{R}^{J+1,k}(t) \right) \right\|_{L^\infty_\mu} \lesssim \epsilon^{2J+1}t^\gamma.
\end{equation}
The bounds of Corollary \ref{cor:boundsinfr} also hold for $w$: for $[t_{k+1},t_k]$,
\begin{equation} \label{est:w_INFR_2}
\left\|\mathcal{N}_{\textnormal{res}}^{J,k}(t)\right\|_{L^\infty_\mu} + \|\mathcal{N}_{\textnormal{bd}}^{J,k}(t)\|_{L^\infty_\mu} + \|\mathcal{N}_{\textnormal{dt}}^{J,k}(t)\|_{L^\infty_\mu} + \|\mathcal{N}_{F_2}^{J-1,k}(t)\|_{L^\infty_\mu} \lesssim t_k^\gamma \epsilon^{2J+1}.
\end{equation}
\end{lem}

The choice of this norm is essentially guided by the control of $\Lambda w$ and \eqref{est:td_tw}
		
\begin{proof}
We use a sequence $(c_j)_j$ as in \eqref{eq:cj} and set $N_j = c_j/t_{k+1}$, as in \eqref{eq:Nj}.
		
\textit{Step 1. Adaptation of the INFR bounds.} Observe that $\mathcal{N}_{F_2}^{J,k}(t)+\mathcal{R}^{J+1,k}(t)$ corresponds to the nonresonant terms at step $J$ where, in the last subtree, one child is colored $\partial_t w$. 
More precisely, each term will be of the form
\begin{equation}\label{eq:formabadbhaved}
	\int_{t_{k+1}}^tN(s,\xi)ds=\int_{t_{k+1}}^t\int_{\overline{\Gamma}^J}e^{is\overline{\Theta}^J} \left(\prod_{j=1}^{J}\frac{m^j}{i\overline{\Theta}^j}\mathbbm{1}_{|\overline{\Theta}^j|>N_j}\right) \cdot F(s,\xi,\overline{\Xi}^J) d\overline{\Xi}^J,
\end{equation}
where $F$ is a product of the intervening functions $w,z,S_0,S_{\reg}$ and/or $K_0$, and exactly one $\partial_tw$. We now derive a bound for this term based on the analysis performed in Section \ref{sec:INFR_bounds}.
			
Consider the associated tree, where every node is colored either $w,z,S_0,S_{\reg}$ or $K_0$, except one terminal node, which is colored $\partial_t w$. Define the weights as follows:
\begin{itemize}
    \item for a terminal node colored $f$ (see \eqref{eq:frequencyweights}), its weight is \begin{equation}
        \begin{cases}
		    t^{-\gamma}\jap{\zeta}^{\mu} & \quad \text{if } f=w, \\
    		t^\rho \jap{\zeta}^{\nu} &\quad \text{if } f=z, \\
    		1 &\quad \text{if }f=S_0, \\
	    	\jap{t^{1/3}\zeta}^{(4/7)^-} & \quad \text{if } f=S_{\reg}\\
		    t^{1/2}|\zeta|^{1/2} &\quad \text{if } f=K_0, \\
		    t^{1-\gamma}|\zeta|^{-1} \jap{\zeta}^{\mu} &\quad \text{if } f=\partial_t w.
	    \end{cases}
    \end{equation}
	\item for a node which is an ancestor of the $\partial_tw$ node (so it is non-terminal), the weight will be $\jap{\zeta}^{\mu}|\zeta|^{-1}$. 
	\item for the remaining (non-terminal) nodes, we set the weight to be $\jap{\zeta}^{\mu}$.
\end{itemize}
Then, defining (analogously to \eqref{def:weight_G})
\[ 
	\tilde{\mathcal{M}}^j=\frac{m^j \times \mbox{ weight for the parent}}{\mbox{ weights for the children}},
	\]
the procedure described in Section \ref{sec:INFR_bounds} gives
\begin{equation}\label{eq:provaresto}
    \left\|\frac{1}{|\xi|}\int_{t_{k+1}}^t N(s,\xi)ds \right\|_{L^\infty_\mu} \lesssim \int_{t_{k+1}}^t \prod_{j=1}^{J}	\left[ \sup_{\xi^j,\alpha}\  \int_{{\Gamma}^{j}} {\left|\tilde{\mathcal{M}}^j\right|}\mathbbm{1}_{|\Theta^j-\alpha|\ge N_j}d\Xi^{j}\right]ds \cdot \epsilon^{2J+1}.
\end{equation}

\medskip

\emph{Step 2}. We claim that Lemma \ref{lem:estsubtree} holds for $w$, that is, given $j < J$ and $t \in [t_{k+1},t_k]$, one has
\[ \forall M \ge 1, \ \forall \alpha \in \m R, \quad 
    \int_{{\Gamma}^{j}} {\tilde{\mathcal{M}}^j}\mathbbm{1}_{|\Theta^j-\alpha|<M}d\Xi^{j}\lesssim t^{-1+\beta} M^\beta, \]
and
\[
	 \sup_{\xi^j,\alpha}\  \int_{{\Gamma}^{j}} {\tilde{\mathcal{M}}^j}\mathbbm{1}_{|\Theta^j-\alpha| \ge N_j}d\Xi^{j}\lesssim
			(tN_j)^{-1+\beta}.
\]
Indeed, let us we repeat the proof of Lemma \ref{lem:estsubtree}: we compare the multipliers $\mathcal{M}^j$ and $\tilde{\mathcal{M}}^j$.
\begin{itemize}
	\item If the $j^{th}$ parent is not an ancestor of the $\partial_tw$ node, then $\mathcal{M}^j$ and $\tilde{\mathcal{M}}^j$ are the same and Lemma \ref{lem:estsubtree} holds.
	\item If the $j^{th}$ parent is an ancestor of the $\partial_tw$ node, then the change between $\mathcal{M}^j$ and $\tilde{\mathcal{M}}^j$ is the extra factor $|\zeta|^{-1}$ coming from the parent and one of its children. In the proof of Lemma \ref{lem:estsubtree}, this corresponds to replacing the $|\xi|$ factor with the slightly stronger weights $\max_{j=1,2,3}|\xi_j|$ (for ternary trees) or $\max(|\xi|,|\eta|)$ (for binary trees). Fortunately, the frequency-restricted estimates of Section \ref{sec:frequencyrestricted} hold for these multipliers and thus Lemma \ref{lem:estsubtree} is still valid.
\end{itemize}

\bigskip
         
\emph{Step 3}. We have the following analogue of Lemma \ref{lem:estfinalsubtree}: the contribution corresponding to the terminal subtree is bounded,  uniformly in $[t_{k+1}, t_k]$, by
\[
	 \sup_{\xi^J,\alpha}\  \int_{{\Gamma}^{J}} {\tilde{\mathcal{M}}^J}\mathbbm{1}_{|\Theta^J-\alpha| \ge N_j}d\Xi^{J}\lesssim (tN_J)^{-1+\beta}t^{-1+\gamma}.
\]
Indeed, this term is analog to that of case b) in Lemma \ref{lem:estfinalsubtree}, and as in the previous step, we repeat the proof. The change between $\tilde{\mathcal{M}}^J$ and $\mathcal{M}^J$ is now twofold: the extra $|\zeta|^{-1}$ coming from the parent and the $\partial_t w$ child (which, as in the previous step, is harmless); the $\frac{1}{t}$ factor coming from the $\partial_t w$-weight. Therefore \eqref{eq:bdrytree} holds with an extra $t^{-1}$ factor and the claim follows.
			
\bigskip
			
\noindent\emph{Step 4. Conclusion.} We argue as in Corollary \ref{cor:boundsinfr}. Inserting the estimates from Steps 2 and 3 into \eqref{eq:provaresto} and recalling that $N_j=c_j/t_{k+1}$, we see that each term in $\mathcal{N}_{F_2}^{J,k}+\mathcal{R}^{J+1,k}$ is bounded by
\begin{equation}
	\left\|\frac{1}{|\xi|}\int_{t_{k+1}}^t N(s,\xi)ds \right\|_{L^\infty_\mu} \lesssim \int_{t_{k+1}}^t s^{-1+\gamma}\prod_{j=1}^{J} (sN_j)^{-1+\beta}ds\cdot\epsilon^{2J+1} \lesssim \epsilon^{2J+1} t^\gamma \prod_{j=1}^J c_j.
\end{equation}
Therefore, as there is at most $\text{Card}(AT_j)$ terms involved, 
\begin{equation}
\forall t \in [t_{k+1}, t_k], \quad	\left\|\frac{1}{|\xi|} \left( \mathcal{N}_{F_2}^{J,k}(t)+\mathcal{R}^{J+1,k}(t) \right) \right\|_{L^\infty_\mu} \lesssim \epsilon^{2J+1}t^\gamma,
\end{equation}
as desired. The proof of \eqref{est:w_INFR_2} is similar and the same as that of Corollary \ref{cor:boundsinfr}.
\end{proof}
		
		\begin{nb}\label{rmk:abs2}
			As a byproduct of the above proof, we see that, for each $\xi\neq 0$, the integrals \eqref{eq:formabadbhaved} are absolutely convergent.
		\end{nb}
		
		\begin{prop}\label{prop:nfevale}
			Let $w$ satisfy the conditions of Theorem \ref{thm:exist}, in particular \eqref{est:td_tw}. Then, for any $J\ge1$, $k \in \m N$ and $t\in [t_{k+1},t_k]$,
			\begin{equation}\label{eq:nfe3}
				w(t)=w(t_{k+1})+\sum_{j=1}^\infty \left(\mathcal{N}_{\textnormal{res}}^{j,k}(t)+\mathcal{N}_{\textnormal{bd}}^{j,k}(t)+\mathcal{N}_{\textnormal{dt}}^{j,k}(t)+\mathcal{N}_{F_2}^{j-1,k}(t) \right),
			\end{equation}
			(with convergence in $L^\infty(\jap{\zeta}^\mu |\zeta|^{-1} d\zeta)$).	In particular, if one sets
			\begin{equation}\label{eq:termosinfinitos}
				\mathcal{N}^j_\ast(t) =  \mathcal{N}^{j,k}_\ast(t) + \sum_{k'> k} \mathcal{N}^{j,k'}_\ast(t_{k'}) \quad \text{for } \ast\in\{ \textnormal{res, bd, dt, } F_2\},
			\end{equation} 
			(with convergence in $L^\infty_\mu$), then the relation \eqref{eq:nfe2} holds (with convergence in $L^\infty_\mu$).
		\end{prop}
		
\begin{proof}
	The result is a consequence of the INFR expansion as long as one can rigorously apply the procedure described in Section \ref{sec:INFR_algorithm} at the level of regularity of $w$. More specifically, we need to justify the integration by parts in time, the product rule for the time derivative and the decay in $J$ of the badly-behaved terms.
			
	First, notice that the condition \eqref{est:td_tw} on $w$ implies after integration that 
	$$w(0)=0 \quad \mbox{ and }\quad  w \in \q C^{0,\gamma}([0,T], L^\infty(\jap{\xi}^\mu|\xi|^{-1} d\xi)).$$
			
	\bigskip
			
	\emph{Step 1. The integration by parts in time are justified.} Indeed, observe that both $S(\cdot,\xi)$ and $w(\cdot,\xi)$ belong to $W^{1,\infty}((0,T))$ for any $\xi\in\R$ a.e. Therefore the product rule for the time derivative (which is applied for fixed frequencies) is valid.
			
	Then, we focus for example on the integration by parts for nonresonant terms in the generic form \eqref{eq:genericform}. We claim that the following sequence of equalities holds:
	\begin{align*}
		\MoveEqLeft[0.5] \int_{t_{k+1}}^t \int_\Gamma {e^{is\Theta}} m \mathbbm{1}_{|\Theta|>N} \prod_{l=1}^{l_0}f_l(s,\xi_l) d\Xi ds = \int_{t_{k+1}}^t \int_\Gamma \partial_t\left(\frac{e^{is\Theta}}{i\Theta}\right)m\mathbbm{1}_{|\Theta|>N}\prod_{l=1}^{l_0}f_l(s,\xi_l) d\Xi ds \\
		& = \int_{t_{k+1}}^t \int_\Gamma \partial_t\left(\frac{e^{is\Theta}}{i\Theta}m\mathbbm{1}_{|\Theta|>N}\prod_{l=1}^{l_0}f_l(s,\xi_l)\right) d\Xi ds - \int_{t_{k+1}}^t \int_\Gamma \frac{e^{is\Theta}}{i\Theta}\partial_t\left(m\mathbbm{1}_{|\Theta|>N}\prod_{l=1}^{l_0}f_l(s,\xi_l)\right) d\Xi ds \\ 
		& = \int_\Gamma\int_{t_{k+1}}^t \partial_t\left(\frac{e^{is\Theta}}{i\Theta}m\mathbbm{1}_{|\Theta|>N}\prod_{l=1}^{l_0}f_l(s,\xi_l)\right)ds d\Xi  - \int_{t_{k+1}}^t \int_\Gamma \frac{e^{is\Theta}}{i\Theta}\partial_t\left(m\mathbbm{1}_{|\Theta|>N}\prod_{l=1}^{l_0}f_l(s,\xi_l)\right) d\Xi ds \\
		& = \int_\Gamma \left[ \frac{e^{is\Theta}}{i\Theta}m\mathbbm{1}_{|\Theta|>N}\prod_{l=1}^{l_0}f_l(s,\xi_l)\right]_{t_{k+1}}^t d\Xi  - \int_{t_{k+1}}^t \int_\Gamma \frac{e^{is\Theta}}{i\Theta}\partial_t\left(m\mathbbm{1}_{|\Theta|>N}\prod_{l=1}^{l_0}f_l(s,\xi_l)\right) d\Xi ds.
	\end{align*}
	The first equality is immediate. By Remarks \ref{rmk:abs}, \ref{rmk:abs2}, every integral appearing in the INFR iteration is absolutely convergent (in time and frequency). In particular, the second equality holds: one applies the product rule in time and then splits the integral. The third equality is just the application of Fubini's theorem, which is again a consequence of the absolute convergence of the integral. In the last equality, we use once again the fact that, for each $\Xi$ fixed, the integrand is $W^{1,\infty}([t_{k+1},t])$ and so absolutely continuous. Thus the integration by parts can be made rigorous at the current level of regularity.
	
	\bigskip
	
	\emph{Step 2. Equation \eqref{eq:nfe3} holds for $w$.} Let us write the equation for $w$ at the $J^{th}$ step:
	\begin{equation}
		w(t)=w(t_{k+1})+\sum_{j=1}^J \left(\mathcal{N}_{\text{res}}^{j,k}(t)+\mathcal{N}_{\text{bd}}^{j,k}(t)+\mathcal{N}_{\text{dt}}^{j,k}(t)+\mathcal{N}_{F_2}^{j,k}(t) \right)\ + \ \mathcal{R}^{J+1,k}(t).
	\end{equation}
	Then, by Lemma \ref{lem:restodecai},
	\begin{align*}
		\MoveEqLeft \left\| \frac{1}{|\xi|}\left( w(t)-\left[ w(t_{k+1})+\sum_{j=1}^J \left(\mathcal{N}_{res}^{j,k} (t) +\mathcal{N}_{bd}^{j,k} (t) + \mathcal{N}_{dt}^{j,k} (t) +\mathcal{N}_{F_2}^{j-1,k}(t) \right)\right] \right) \right\|_{L^\infty_\mu} \\
		& = \left\| \frac{1}{|\xi|} \left( \mathcal{N}_{F_2}^{J,k}(t)+\mathcal{R}^{J+1,k}(t) \right) \right\|_{L^\infty_\mu}\to 0\quad\text{as }J\to \infty.
	\end{align*}
	from which \eqref{eq:nfe3} follows. 
	
	\bigskip
	
	\emph{Step 3. Equality \eqref{eq:nfe2} holds for $w$.}  Due to \eqref{est:w_INFR_2}, the series below are absolutely convergent: for $* \in \{ \text{res, bd, dt}, F_2 \}$,
	\[
		\left\|\mathcal{N}^{j,k}_\ast(t)\right\|_{L^\infty_\mu} + \sum_{k'> k} \left\|\mathcal{N}^{j,k'}_\ast(t_{k'})\right\|_{L^\infty_\mu} \lesssim \epsilon^{2j+1}\left( t^\gamma + \sum_{k'>k} t_{k'}^\gamma\right)\lesssim \epsilon^{2j+1} t^\gamma.
	\]
	Hence $\q N_*^j \in L^\infty_\mu$ is well-defined. 
	For any $k' > k+1$,
	\begin{align*}
		w(t)&= w(t) - w(t_{k+1}) + \sum_{\kappa = k+1}^{k'-1} (w(t_{\kappa})-w(t_{\kappa+1})) + w(t_{k'}) \\
		& = \sum_{j=1}^\infty \left( \sum_{* \in \{ \text{res}, \text{bd}, \text{dt} \}} \mathcal{N}_{*}^{j,k}(t)+\mathcal{N}_{F_2}^{j-1,k}(t) \right) \\
		&\qquad+ \sum_{\kappa = k+1}^{k'-1} \sum_{j=1}^\infty  \left( \sum_{* \in \{ \text{res}, \text{bd}, \text{dt} \}} \mathcal{N}_{*}^{j,\kappa}(t_\kappa) + \mathcal{N}_{F_2}^{j-1,\kappa}(t_{\kappa}) \right) + w(t_{k'}) \\
		&=\sum_{j=1}^\infty \left[ \sum_{* \in \{ \text{res}, \text{bd}, \text{dt} \}} \left( \mathcal{N}_{*}^{j,k}(t) + \sum_{\kappa = k+1}^{k'-1} \mathcal{N}_{*}^{j,\kappa}(t_\kappa) \right) + \cal N_{F_2}^{j-1,k}(t) + \sum_{\kappa = k+1}^{k'-1} \cal N_{F_2}^{j-1,\kappa} (t_{\kappa})  \right] + w(t_{k'}).
	\end{align*}
	Let $k'\to +\infty$. Then, in $L^\infty_\mu$, $w(t_{k'}) \to 0$ (due to \eqref{est:td_tw}). Furthermore, the above computations show that the series are summable. We thus get \eqref{eq:nfe}
    and the proof is finished.
	\end{proof}
		
We are now ready to complete the proof of Theorem \ref{thm:exist}.
		
\begin{proof}[Proof of Theorem \ref{thm:exist}]
	Recall that the existence of $w$ has been shown in Proposition \ref{prop:exist}, so we are left with the proof of uniqueness and continuous dependence \eqref{est:cont_dep}: in fact it suffices to show the latter, which we do now.
			
	Let $z_1,z_2\in L^\infty_\nu$ are two perturbations and let solutions $w_1,w_2 \in L^\infty([0,T], W^{1,\infty}_\mu)$ be two corresponding solutions satisfying the conditions of Theorem \ref{thm:exist}, in particular \eqref{est:td_tw}. By Proposition \ref{prop:nfevale}, both satisfy \eqref{eq:nfe2}. We observe that  the r.h.s. of \eqref{eq:nfe2} can be seen as a bounded multilinear operator acting in weighted $L^\infty$ spaces\footnote{This is the main advantage in working with the normal form equation. By replacing the nonlinearity with a sum of well-behaved terms, \emph{a priori} bounds can now be obtained without having to resort to any auxiliary space.}. More precisely, denoting $\q N_*^j(t,w)$ the terms appearing in \eqref{eq:nfe2} related to the expansion of $w$, $\q N_*^j(t,w)$ is a homogeneous polynomial of degree $2J+1$ in its variables $S_0, S_\text{reg}, z, w, K_0$ (counting $K_0$ as quadratic), and estimates as in Corollary \ref{cor:boundsinfr} allow to bound, for $t \in [0,T]$
    \[ 
	    \| \q N_*^j(t,w_1) - \q N_*^j(t,w_2) \|_{L^\infty_\mu} \lesssim \e^{2j} \left( \sup_{s \in [0,t]} s^\gamma \| w_1(s) - w_2(s) \|_{L^\infty_\mu} + t^\rho \| z_1 - z_2 \|_{L^\infty} \right).
	\]
	Indeed, in each difference of a term appearing in $\q N_*^j$, we can factorise either $w_1 - w_2$ or $z_1 - z_2$ (since $S_0$, $S_{\text{reg}}$, $K_0$ do not depend on $z_1,z_2$), and bound the result according to the weights indicated in \eqref{eq:frequencyweights}.
	
	Therefore, similar to Proposition \ref{prop:estimatew}, the difference between the two solutions satisfies, for $t \in [0,T]$
	\begin{equation}
		\sup_{s\in[0,t]} \|w_1(s)-w_2(s)\|_{L^\infty_\mu}\le C_0 \epsilon^2	\left( t^\gamma \sup_{s\in[0,t]}\|w_1(s)-w_2(s)\|_{L^\infty_\mu}+ t^\rho \| z_1 - z_2 \|_{L^\infty_\mu} \right).
	\end{equation}
	Choose $T >0$ so small that $T \le T_0'$ and $C_0 \epsilon^2 T^\gamma \le 1/2$, then \eqref{est:cont_dep} holds on $[0,T]$ (with implicit constant $2C_0$).
	\end{proof}

		\section{Proof of Theorem \ref{thm:lwp}} \label{sec:th2}
		
		In this last section, we present a sketch of the proof of Theorem \ref{thm:lwp}. In fact, it follows from a simplification of the arguments used to prove Theorem \ref{thm:exist}. Before we begin, observe that, through the scaling invariance, it suffices to work on the time interval $(0,1)$ with $0<t_0<1$. We start by writing the equation for $w$,
		\[
		\begin{cases}
		    \partial_t w = N[w+S]-N[S] = Q[S,w] + 6N[S,S_{\reg},w]+L_K[w], \\ w(t_0)=w_0,
		\end{cases}
		\]
		where $Q$ contains the terms which are at least quadratic in $w$; $6N[S,S_{\reg},w]$ corresponds to a linear term in $w$ without two $S_0$'s; $L_K[w]$ is the linear term in $w$ with two $S_0$'s, written as in \eqref{eq:K}.
		
		We then approximate the problem by cutting off large frequencies as done in Section \ref{sec:aprox}:
		\begin{equation}\label{eq:approx2}
			\begin{cases}
			\partial_t w_n = \chi_n^2 (N[S+w_n]-N[S]),\\
			 w_n(0)=\chi_nw_0,
			\end{cases}
		\end{equation}
		and construct, by a fixed-point argument, a solution in $Y_n(I)$ (recall the definition \eqref{def:Yn}).
		 
		\begin{prop}\label{prop:exist_aprox_12}
			Under the conditions of Theorem \ref{thm:lwp}, there exists $0 < T^n_\pm \le 1$ and a unique $w_n\in Y_{n}([T_-^n, T_+^n])$ (integral) solution to \eqref{eq:approx2} for $t\in(T^n_-,T^n_+)$. If $T_+^n<+\infty$, then $\| w \|_{Y_n(t)} \to \infty$ as $t \to T_+^n$, and an analogous alternative holds if $T_-^n>0$.
			
			Moreover, there exists $t^n>0$ such that
\[ \forall t \in [t_0-t^n,t_0+t^n], \quad
			\|w_n \|_{Y_n(t)} \le 2\epsilon.	
\]
		\end{prop}
		The second step is the derivation of \textit{a priori} bounds for $w_n$ and $\Lambda w_n$ through a bootstrap argument, starting from 
		\[
		\|w_n(t)\|_{L^\infty_\mu}, \|\Lambda w_n(t)\|_{L^\infty_\mu}<2\epsilon,\quad t,t_0\in I
		\]
		and reaching
		\[
		\|w_n(t)\|_{L^\infty_\mu}, \|\Lambda w_n(t)\|_{L^\infty_\mu}<\frac{3}{2}\epsilon.
		\]
		As done in Section \ref{sec:INFR}, this will follow from the application of the INFR over dyadic intervals. We consider intervals of the form $[t_{k+1},t_k]=[t_0/2^{k+1},\ t_0/2^k]$ for $k \in \m Z$ and set
\[
		\beta=1-\frac{\mu}{3}>\frac{1}{2}.
\]
		Then we take $c_j$ as in \eqref{eq:cj} and define the frequency thresholds $N_j=c_j/t_{k+1}$ (see \eqref{eq:Nj}). One simplification with respect to Section \ref{sec:INFR} is that only finitely many intervals $[t_{k+1},t_k]$ will play a role.
		
		We can now perform the same argument as in Section \ref{sec:INFR_bounds}: this allows to write the expansion, for $t \in [t_{k+1}, t_k] \cap (T_-^n, T_+^n)$,
\begin{equation} \label{eq:INFR_>0}
w_n(t) = w_n(t_{k+1}) + \chi_n^2 \left( \sum_{j=1}^J \left(\mathcal{N}_{\text{res},n}^{j,k} (t)+\mathcal{N}_{\text{bd},n}^{j,k}(t)+\mathcal{N}_{\text{dt},n}^{j,k}(t) \right) +  \mathcal{R}^{J+1,k}_n (t) \right),
\end{equation}
	(indeed, as there is no longer a source term $z$, there is no $F_2$ term in the INFR here), where the terms in the above formula are as in \eqref{eq:formares} or \eqref{eq:formabdry}.

Having defined the INFR development, let us perform with care the estimates for resonant, boundary and derivative terms. Define the frequency weight for a terminal node colored $f$ as
		\begin{equation}\label{eq:frequencyweights2}
			\begin{cases}
				\jap{\zeta}^\mu ,&\quad \text{if }f=w,\\
				1,&\quad \text{if }f=S_0,\\
				\jap{t^{1/3}\zeta}^{(4/7)^-},&\quad \text{if }f=S_{\reg},\\
				t^{1/2}|\zeta|^{1/2},&\quad \text{if } f=K_0.\\
			\end{cases}
		\end{equation}
		When compared with \eqref{eq:frequencyweights}, the difference is the lack of a weight in time when $f=w$. For non-terminal nodes, we set the weight to be $\jap{\zeta}^\mu $. Arguing as in Section \ref{sec:INFR_bounds}, we are led to the bound \eqref{est:Nres}.
		
		The following results are the analogous versions of Lemmas \ref{lem:estsubtree} and \ref{lem:estfinalsubtree}.
		
		\begin{lem}\label{lem:estsubtree2}
			Let $t \in [t_{k+1},t_k] \cap [0,1]$, and in a term corresponding to a tree of size $J$, consider, for some $j <J$, the factor given by the $j^{th}$ elementary subtree. Then
			\[ \forall M \ge 1, \ \forall \alpha \in \m R, \quad \int_{{\Gamma}^{j}} {\mathcal{M}^j}\mathbbm{1}_{|\Theta^j-\alpha|<M}d\Xi^{j}\lesssim t^{-1+\beta} M^\beta. \]
			As a consequence,
			\[  \sup_{\xi^j,\alpha}\  \int_{{\Gamma}^{j}} {\mathcal{M}^j}\mathbbm{1}_{|\Theta^j-\alpha|\ge N_j}d\Xi^{j}\lesssim (tN_j)^{-1+\beta}.
			\]
		\end{lem}
		
		\begin{proof}
			There are three possibilities. If the subtree is ternary so that $\Theta^j=\Phi$:
			\begin{itemize}
			    \item If at least two children have $w$-weights, then
\[ 
			\mathcal{M}^j\lesssim \frac{|\xi|\jap{\xi}^\mu }{\jap{\xi_1}^\mu \jap{\xi_2}^\mu }.
\]
Due to \eqref{eq:TalphaM_multiplierxi},
\[
			\sup_{\xi,\alpha}\int \mathcal{M}^j \mathbbm{1}_{|\Phi-\alpha|<M}d\xi_1 d\xi_2 \lesssim M^{1-\mu/3} \lesssim t^{-1+\mu/3} M^{1-\mu/3},
\]
			\item Otherwise, there is one child colored $S$ and another colored $S_{\reg}$, so that
\[
			\mathcal{M}^j\lesssim \frac{|\xi|\jap{\xi}^\mu }{\jap{t^{1/3}\xi_1}^{(4/7)^-}\jap{\xi_2}^\mu }\lesssim  \frac{|\xi|\jap{\xi}^\mu }{t^{\mu/3}\jap{\xi_1}^{\mu}\jap{\xi_2}^\mu },
\]
			and applying again \eqref{eq:TalphaM_multiplierxi},
\[
			\sup_{\xi,\alpha}  \int {\mathcal{M}^j}\mathbbm{1}_{|\Phi-\alpha|<M}d\xi_1d\xi_2 \lesssim t^{-1+\mu/3} M^{1-\mu/3}.
\]
			\end{itemize}
Now, we have the case where the tree is binary and $\Theta^j=\Psi$: the only possibility is that a child is colored $K_0$ and the other $w$, so that
\[
			\mathcal{M}^j\lesssim \frac{|\xi|}{t^{1/2}\sqrt{|\eta|}\jap{\xi-\eta}^\mu }.
\]
			The conclusion now follows from \eqref{eq:BalphaMmultiplier}.
			
This concludes the first claim. The second is then a direct application of Lemma \ref{lem:restricted_<>}. 
		\end{proof}
		
		\begin{lem}[Final subtree] \label{lem:estfinalsubtree2}
			Let $t\in [t_{k+1},t_k] \cap (0,1]$, and consider a term corresponding to a tree of size $J$.
			
			a) For resonant terms, the contribution of the terminal subtree is bounded by
			\begin{equation}\label{eq:resonanttree2}
				\sup_{\xi^J,\alpha}\  \int_{{\Gamma}^{J}} {\mathcal{M}^J}\mathbbm{1}_{|\Theta^j-\alpha|<N_j}d\Xi^{j}\lesssim (tN_J)^\beta t^{-1}.
			\end{equation}
			b) For boundary terms, the corresponding bound is
			\begin{equation}\label{eq:bdrytree2}
				 \sup_{\xi^j,\alpha}\  \int_{{\Gamma}^{J}} {\mathcal{M}^J}\mathbbm{1}_{|\Theta^J-\alpha|>N_J}d\Xi^{J}\lesssim (tN_J)^{-1+\beta} t^{\gamma}.
			\end{equation}
			c) For derivative terms, the corresponding bound is
			\begin{equation}\label{eq:derivativetree2}
			 \sup_{\xi^J,\alpha}\  \int_{{\Gamma}^{J}} {\mathcal{M}^J}\mathbbm{1}_{|\Theta^J-\alpha| <N_J}d\Xi^{J}\lesssim (tN_J)^{-1+\beta} t^{-1}.
			\end{equation}
		\end{lem}
		\begin{proof}
			The proof for resonant and boundary trees follows from the arguments of the previous proof. For derivative trees, the multiplier is as for boundary trees, except for an extra $t^{-1}$ factor.
		\end{proof}
		
		The estimates for the subtrees give the analogue of Corollary \ref{cor:boundsinfr}.
		\begin{cor}\label{cor:boundsinfr2}
			In the normal form expansion \eqref{eq:INFR_>0}, given $J \ge 1$ and $t\in[t_{k+1},t_k] \cap (0,1]$
			\[
			\left\|\mathcal{N}_{res}^{J,k}(t)\right\|_{L^\infty_\mu } + 	\|\mathcal{N}_{bd}^{J,k}(t)\|_{L^\infty_\mu } + \|\mathcal{N}_{dt}^{J,k}(t)\|_{L^\infty_\mu } \lesssim  \epsilon^{2J+1}. 
			\]
		\end{cor}
		In particular, if $t\in[t_{k+1},t_k] \cap (0,1]$,
		\begin{align*}
			\MoveEqLeft \| \chi_n^{-1} (w_n(t)-w_n(t_0)) \|_{L^\infty_\mu } \lesssim \| \chi_n^{-1}( w_n(t)-w_n(t_{k+1})) \|_{L^\infty_\mu} + \sum_{k'=0}^{\nu} \|\chi_n^{-1} (w(t_k)-w(t_{k+1})) \|_{L^\infty_\mu} \\
			&\lesssim \|\mathcal{N}_{res}^{J,k}(t) + 	\mathcal{N}_{bd}^{J,k}(t) + \mathcal{N}_{dt}^{J,k}(t)\|_{L^\infty_\mu } + \sum_{k'=0}^{\nu} \|\mathcal{N}_{res}^{J,k}(t_k)+ 	\mathcal{N}_{bd}^{J,k}(t_k) + \mathcal{N}_{dt}^{J,k}(t_k)\|_{L^\infty_\mu } \\&  \lesssim |k|\epsilon^{3} \lesssim  \ln\left|\frac{t}{t_0}\right|\epsilon^{3}.
		\end{align*}
		and therefore, if $t\in (T_-^n,T_+^n) \cap  [t_0e^{-2C\epsilon^2},\min\{1,t_0e^{2C/\epsilon^2}\}]$,
\[
\| \chi_n^{-1}  w_n(t)\|_{L^\infty_\mu} \le \| \chi_n^{-1}  w_n(t_0)\|_{L^\infty_\mu }+ C\ln\left|\frac{t}{t_0}\right|\epsilon^{3} \le 2\epsilon.
\]
		This bound essentially closes the bootstrap argument on $\|\chi_n^{-1}  w_n(t)\|_{L^\infty_\mu}$. The bootstrap on $\|\chi_n^{-1}\Lambda w_n(t)\|_{L^\infty_\mu}$ also follows from the INFR procedure: indeed, the algebraic structure of the nonlinearity in the $\Lambda w_n$ equation and the required frequency-restricted estimates coincide with those for $w_n$. To summarize this discussion, there exist a constant $\tau_\e >1$ (depending on $\epsilon$ only) such that denoting 
		\[ T_-  := t_0/\tau_\e \quad\text{and} \quad T_+ := \min(1,\tau_\e t_0), \]
		 then
\begin{equation} \label{est:bd_wn_2}
\forall n \in \m N, \forall t \in (T_-^n,T_+^n) \cap [T_-,T_+], \quad \| \chi_n^{-1}  w_n (t)\|_{L^\infty_\mu } + \| \chi_n^{-1}  \Lambda w_n(t)\|_{L^\infty_\mu } \le 2 \epsilon.
\end{equation}
		
		The last step that ensures the existence of a uniform time of existence for $w_n$ is a bootstrap argument for $\partial_t w_n$, as in Proposition \ref{prop:uniform}.

		\begin{prop}[Uniform local time of existence]\label{prop:uniform2}
			In the conditions of Theorem \ref{thm:lwp}, there exist $T_-<t_0<T_+$, independent of $n$, such that
\[ \forall t \in [T_-,T_+], \quad 
			\| {w}_n \|_{Y_n(t)} \le 2\epsilon.
\]
			In particular, $[T_-,T_+]\subset (T_-^n,T_+^n)$ (where $T_\pm^n$ are the maximal times of existence of $w_n$ in Proposition \ref{prop:exist_aprox_12}). Furthermore, \eqref{eq:limiteT} holds.
		\end{prop}
		
		\begin{proof}
			Fix $n \in \m N$. Define 
\[ 
			T_+^*=\sup\{ \tau \in [t_0, \min(T_+,T_+^n)) :  \| w_n \|_{Y_n([t_0,\tau])} \le 6\epsilon\}.
\]
			By Proposition \ref{prop:exist_aprox_12}, $T_+^*>t_0$. Given $t \in [t_0, T_+^*)$, the \textit{a priori} estimates \eqref{est:bd_wn_2} given by the INFR yield
\[
			\|\chi_n^{-1}{w}_n(t)\|_{L^\infty_\mu}+\|\chi_n^{-1}\Lambda {w}_n(t)\|_{L^\infty_\mu} \le 2\epsilon.
\]
The bootstrap assumption gives that $\| w_n(t) \|_{W^{1,\infty}_\mu} \le 6\e$ for $t \in [t_0, T_+^*)$. Then we use Lemma \ref{lem:controldt} (with $a=a'=b = \mu$) to produce a control for the time derivative: as $\chi_n \le 1$ and $0 \le t \le T_+^* \le 1$
\begin{align}
	\left\| \chi_n^{-1}\frac{t\partial_t {w}_n(t)}{|\xi|} \right\|_{L^\infty_{\mu}} &\le  	\left\|\frac{t}{\xi}(N[S+w_n]-N[S])\right\|_{L^\infty_{\mu}} \\
	&\lesssim  \| w_n\|_{W^{1,\infty}_\mu}\left(\|S\|_{W^{1,\infty}_{0,1}(\R\setminus\{0\})}^2 + \| w_n\|_{W^{1,\infty}_\mu}^2\right)\\
	&\lesssim \epsilon^3  \le \epsilon,\label{eq:bootstrap2}
\end{align}
for $\epsilon \le \epsilon_0$ small enough: indeed, one argues as in the proof of Proposition \ref{prop:uniform}, with the simplification that there is not $z$ term here. Therefore for $t \in [t_0,T_+^*)$,
\[
			\left\| \chi_n^{-1}\partial_\xi{w}_n(t) \right\|_{L^\infty_{\mu}} \lesssim \left\| \chi_n^{-1}\frac{t\partial_t {w}_n(t)}{|\xi|} \right\|_{L^\infty_{\mu}} + \|\chi_n^{-1}\Lambda{w}_n(t)\|_{L^\infty_{\mu}} \le 3 \epsilon. 
\]
and so $\| w_n \|_{Y_n(t)} \le 5 \epsilon < 6 \epsilon$. A continuity argument (using the blow up criterion of Proposition \ref{prop:exist_aprox_12}) implies that $T_+^* = T_+ < T_+^n$. One argues in a similar fashion for $t < t_0$, and this yields the result.
		\end{proof}

		Having $w_n$ defined over the interval $[T_-,T_+]$, the existence of $w$ is done exactly as in Proposition \ref{prop:exist}; for uniqueness and continuous dependence, one exploits once again the INFR, as in Section \ref{sec:uniq}: we leave the details to the reader. This concludes the proof of Theorem \ref{thm:lwp}.
		
		\begin{nb}\label{rmk:small}
			The smallness of the initial data $w_0$ comes up only in the control of the time derivative. The estimates \eqref{est:bd_wn_2} given by the INFR on the interval $[T_-,T_+]$ hold regardless of size of $\epsilon$. However, there is no time factor in \eqref{eq:bootstrap2} to compensate for a large $\epsilon$. 
			
			One could attempt to work only with $w$ and $\Lambda w$ in order to bypass this difficulty. The limiting procedure through Ascoli-Arzelà would would still give a limit profile $w$, continuous in the self-similar direction $\tau=t^{1/3}\xi$, but not necessarily in the transverse direction. This lack of regularity makes it impossible to prove that the limiting profile gives a distributional solution to (mKdV). 
			
			One could also ask whether it is possible to construct $w$ without using $\Lambda w$, so as to obtain a result for data $w_0 \in L^\infty_\mu$. This approach raises the same concerns as those explained in Remark \ref{rem:Picard}: the INFR bounds allow to show that $w_n$ is a Cauchy sequence, but it is unclear if the limit solves the \eqref{mKdV} equation.
		\end{nb}

		\section{Proof of Theorem \ref{thm:lwpS=0}}\label{sec:thm1}
		Finally, we prove Theorem \ref{thm:lwpS=0}, through the Fourier restriction norm method. Given $b,\mu\in\R$, we define the adapted Bourgain norm
	\[
		\|u\|_{X^{\mu,b}_\infty}:=\|\jap{\tau}^b\jap{\xi}^\mu\mathcal{F}_{t,x}(e^{t\partial_x^3}u)\|_{L^2_\tau L^\infty_\xi}
		\]
		associated to the space
		\[
		X^{\mu ,b}_\infty=\{ u\in \mathcal{S}(\R\times \R) : \|u\|_{X^{\mu ,b}_\infty}<\infty\}.
		\]
		Observe that, for $b>1/2$, $X_\infty^{\mu ,b} \hookrightarrow \q C(\R, \widehat{L^\infty_\mu} (\R))$.
		
		In order to handle the nonlinearity, it is also convenient to introduce the auxiliary space $Y^{\mu ,b}_\infty$, defined through the norm
		\[
		\|u\|_{Y^{\mu ,b}_\infty}:=\|\jap{\tau}^b\jap{\xi}^\mu \mathcal{F}_{t,x}(e^{t\partial_x^3}u)\|_{L^\infty_\xi L^2_\tau }.
		\]
		For the remainder of this section, we fix $\psi\in C_c^\infty(\R)$ with $\text{supp} \psi \subset[-2,2]$ and $\psi\equiv 1$ in $[-1,1]$. For $0<\delta<1$, define $\psi_\delta(t)=\psi(t/\delta)$. Our goal is to prove 
		
		\begin{prop}\label{prop:bourg}
			Fix $\mu >0$. Given $u_0\in \widehat{L^\infty_\mu} $, there exist $\delta>0$ and a unique $u\in X^{\mu ,b}_\infty$ satisfying
			\begin{equation}\label{eq:duhamelmodif}
				u(t)=\psi(t)e^{-t\partial_x^3}u_0 + \psi_\delta(t)\int_0^t e^{-(t-s)\partial_x^3}(u^3)_x ds.
			\end{equation}
		\end{prop}
		This will follow from a standard fixed-point argument. In particular, we derive Theorem \ref{thm:lwpS=0}.
		
		\begin{lem}[Basic properties in $X^{\mu ,b}_\infty$ spaces]\label{lem:Xsb}
			For $\mu \in \R$, $b>1/2$ and $b-1<b'<0$, there exists $C>0$ such that
			\begin{equation}\label{eq:estimativa_linear}
				\|\psi e^{-t\partial_x^3}u_0\|_{X^{\mu ,b}_\infty} \le C\|\hat{u}_0\|_{L^\infty_\mu }, \quad \mbox{for all }u_0\in \widehat{L^\infty_\mu}, 
			\end{equation}
			and
			\begin{equation}\label{eq:estimativa_inhomog}
				\left\|\psi_\delta(t)\int_0^t e^{-(t-s)\partial_x^3}f(s) ds\right\|_{X^{\mu ,b}_\infty} \le  C\delta^{1+b'-b} \|f\|_{Y^{\mu ,b'}_\infty}, \quad \mbox{for all }f\in Y^{\mu ,b'}_\infty.
			\end{equation}
		\end{lem}
		\begin{proof}
			The linear estimate \eqref{eq:estimativa_linear} follows from a direct computation:
			\begin{equation}
				\|\psi e^{-t\partial_x^3}u_0\|_{X^{\mu ,b}_\infty} = \|\jap{\tau}^b\jap{\xi}^\mu \hat{\psi}(\tau)\hat{u}_0(\xi)\|_{L^2_\tau L^\infty_\xi} \lesssim\|\hat{u}_0\|_{L^\infty_\mu }.
			\end{equation}
			For the proof of \eqref{eq:estimativa_inhomog}, we follow closely the computations of \cite[Lemma 2]{G03}. Let $g=e^{t\partial_x^3}f$. First, we write
			\[
			\psi_\delta(t)\int_0^t g(s) ds = c\psi_\delta(t)\int \frac{e^{it\tau'}-1}{i\tau'} (\mathcal{F}_tg)(\tau',x) d\tau' = I_1 + I_2 + I_3,
			\]
			where
			\begin{align*}
			I_1 & = c\psi_\delta(t)\sum_{k\ge 1} \frac{t^k}{k!}\int_{|\tau|\delta \le 1} (i\tau')^{k-1} (\mathcal{F}_tg)(\tau',x) d\tau', \\
			I_2 &= -c\psi_\delta(t) \int_{|\tau|\delta \ge 1} (i\tau')^{-1} (\mathcal{F}_tg)(\tau',x) d\tau', \\
			\text{and} \quad
			I_3 & = c\psi_\delta(t) \int_{|\tau|\delta \ge 1} e^{it\tau}(i\tau')^{-1} (\mathcal{F}_tg)(\tau',x) d\tau'.
			\end{align*}
			We begin with $I_1$. Taking the Fourier transform in $(t,x)$,
			\begin{align*}
				\mathcal{F}_{t,x} \mathbf I = c\sum_{k\ge 1} \frac{\widehat{t^k\psi_\delta}(\tau)}{k!} \int_{|\tau|\delta \le 1} (i\tau')^{k-1} (\mathcal{F}_{t,x}g)(\tau',\xi) d\tau'.
			\end{align*}
			Due to the Cauchy-Schwarz inequality we can bound for each $k$ and $\xi$ the integral in $\tau'$  by
			\begin{align*}
				\left| \int_{|\tau|\delta \le 1} (i\tau')^{k-1} (\mathcal{F}_{t,x}g)(\tau',\xi) d\tau' \right| &\lesssim \left(\int_{|\tau'|\delta\le 1}|\tau'|^{2k-2}\jap{\tau'}^{-2b'} d\tau'\right)\|\jap{\tau}^{b'}\mathcal{F}_{t,x}g(\xi)\|_{L^2_\tau}\\&\lesssim \delta^{\frac{1}{2}+b'-k}\|\jap{\tau}^{b'}\mathcal{F}_{t,x}g\|_{L^\infty_\xi L^2_\tau}.
			\end{align*}
			On the other hand,
			\begin{align*}
				\|\jap{\tau}^b \widehat{t^k\psi_\delta}\|_{L^2_\tau}^2 &= \int \jap{\tau}^{2b}|\hat{\psi}_\delta^{(k)}(\tau)|^2 d\tau = \delta^{2k+2} \int \jap{\tau}^{2b}|\hat{\psi}^{(k)}(\delta \tau)|^2d\tau \\
				&\lesssim \delta^{2k+2-2b-1} \int \jap{\tau}^{2b}|{\hat \psi}^{(k)}(\tau)|^2 d\tau \lesssim \delta^{2k+1-2b} \| \jap{t}^k \psi \|_{H^1}^2  \lesssim \delta^{2k+1-2b}2^{2 k} k^2.
			\end{align*}
			Therefore, as $k \ge1$ and $b' <0$, $2k+1 -2b \le 2(1+b'-b)$ and
\[
			\|\jap{\tau}^b\jap{\xi}^\mu  \mathcal{F}_{t,x} I_1\|_{L^2_\tau L^\infty_\xi} \lesssim \sum_{k\ge 1}\frac{2^{k}}{{(k-1)!}} \delta^{1+b'-b}\|\jap{\tau}^{b'}\jap{\xi}^\mu  \mathcal{F}_{t,x}g\|_{ L^\infty_\xi L^2_\tau}\lesssim \delta^{1+b'-b}\|\jap{\tau}^{b'}\jap{\xi}^\mu  \mathcal{F}_{t,x}g\|_{ L^\infty_\xi L^2_\tau}.
\]
			For $I_2$, notice that
		\begin{align*}
			\int_{|\tau|\delta\ge 1} (i\tau')^{-1}(\mathcal{F}_{t,x}g)(\tau',\xi)d\tau' & \lesssim \left(\int_{|\tau'|\delta\ge 1} |\tau'|^{-2} \jap{\tau'}^{-2b'} d\tau' \right)^{1/2} \|\jap{\tau}^{b'}\mathcal{F}_{t,x}g(\xi)\|_{L^2_\tau} \\
			&\lesssim \delta^{1/2 + b'}\|\jap{\tau}^{b'}\mathcal{F}_{t,x}g(\xi)\|_{L^\infty_\xi L^2_\tau}.
		\end{align*}
			Hence, as $b > \frac{1}{2}$, $(1-2b) +\left( \frac{1}{2} + b'\right) \ge 1 +b'-b$ and
			\begin{align*}
				\|\jap{\tau}^b\jap{\xi}^\mu  \mathcal{F}_{t,x} I_2 \|_{L^2_\tau L^\infty_\xi} \lesssim \|\jap{\tau}^{b}\hat{\psi}_\delta\|_{L^2_\tau}\delta^{1/2 + b'}\|\jap{\tau}^{b'}\jap{\xi}^\mu\mathcal{F}_{t,x}g\|_{L^\infty_\xi L^2_\tau}\lesssim \delta^{1+b'-b}\|\jap{\tau}^{b'}\jap{\xi}^\mu  \mathcal{F}_{t,x}g\|_{ L^\infty_\xi L^2_\tau}.
			\end{align*}
			Finally, taking the Fourier transform of $I_3$ in $(t,x)$, and by Cauchy-Schwarz inequality,
			\begin{align*}
				|\mathcal{F}_{t,x} I_3| &= \left|\int_{|\tau'|\delta\ge 1} \hat{\psi}_\delta(\tau-\tau')(i\tau')^{-1}(\mathcal{F}_{t,x}g)(\tau',\xi)d\tau'\right| \\
				&\lesssim \left(\int_{|\tau'|\delta\ge 1}|\hat{\psi}_\delta(\tau-\tau')|^2|\tau'|^{-2}\jap{\tau'}^{2b'} d\tau' \right)^{1/2}\|\jap{\tau}^{b'}\mathcal{F}_{t,x}g(\xi)\|_{L^2_\tau}.
			\end{align*}
			Therefore
			\begin{align*}
				\|\jap{\tau}^b\jap{\xi}^\mu  \mathcal{F}_{t,x} I_3 \|_{L^2_\tau L^\infty_\xi} &\lesssim \left(\int_{|\tau'|\delta\ge 1}|\hat{\psi}_\delta(\tau-\tau')|^2|\tau'|^{-2}\jap{\tau'}^{-2b'} \jap{\tau}^{2b} d\tau' d\tau \right)^{1/2}\|\jap{\tau}^{b'}\jap{\xi}^\mu \mathcal{F}_{t,x}g\|_{L^\infty_\xi L^2_\tau}\\
				&\lesssim \delta^{1+b'-b}\|\jap{\tau}^{b'}\jap{\xi}^\mu  \mathcal{F}_{t,x}g\|_{ L^\infty_\xi L^2_\tau}.
			\end{align*}
			Summing up together the estimates for $I_1$, $I_2$ and $I_3$, \eqref{eq:estimativa_inhomog} follows.
		\end{proof}
		
		\begin{lem}\label{lem:bourg_fre}
			For $\mu >0$, there exist $b>1/2$ and $b-1<b'<0$ such that
			\begin{equation}\label{eq:multilinear}
				\|\partial_x(u_1 u_2 u_3)\|_{Y^{\mu ,b'}_\infty}\lesssim \| u_1 \|_{X^{\mu ,b}_\infty} \| u_2 \|_{X^{\mu ,b}_\infty} \| u_3 \|_{X^{\mu ,b}_\infty}.
			\end{equation}
		\end{lem}
		
		\begin{proof}
			We work in space-time Fourier variables: by duality, the multilinear estimate is equivalent to proving that the continuity of the quadrilinear operator $T:(L^2_\tau L^\infty_\xi)^3 \times L^2_\tau \to L^\infty_\xi$ defined by
\[  	T[u_1,u_2,u_3,v] (\xi)=\int_{\Gamma_\xi} \int_{ \Gamma_\tau} \frac{\jap{\tau-\xi^3}^{b'}|\xi|\jap{\xi}^\mu }{\prod_{j=1}^3\jap{\tau_j-\xi_j^3}^b\jap{\xi_j}^\mu }\left(\prod_{j=1}^3u_j(\tau_j,\xi_j)\right)v(\tau)d\xi_1d\xi_2d\tau_1d\tau_2d\tau,
\]
			where $\Gamma_\xi$ and $\Gamma_\tau$ are the surfaces defined respectively by
\[
			\xi=\xi_1+\xi_2+\xi_3,\quad \text{and} \quad \tau=\tau_1+\tau_2+\tau_3 + \Phi(\xi,\xi_1,\xi_2,\xi_3).
			\]
			We first prove the continuity of the operator $T_1:(L^1_\tau L^\infty_\xi)^2\times L^\infty_\tau L^\infty_\xi \times L^\infty_\tau \to L^\infty_\xi$, where
\[
			T_1[u_1,u_2,u_3,v](\xi)=\int_{\Gamma_\xi} \int_{ \Gamma_\tau} \frac{\jap{\tau-\xi^3}^{2b'}|\xi|\jap{\xi}^\mu }{\jap{\tau_3-\xi_3^3}^{2b}\prod_{j=1}^3\jap{\xi_j}^\mu }\left(\prod_{j=1}^3u_j(\tau_j,\xi_j)\right)v(\tau)d\xi_1d\xi_2d\tau_1d\tau_2d\tau.
\]
Assuming unit norms $\| u_1 \|_{L^1_\tau L^\infty_\xi} = \| u_2 \|_{L^1_\tau L^\infty_\xi} = \| u_3 \|_{L^\infty_\tau L^\infty_\xi} = \| v \|_{L^\infty_\tau} =1$, and as $b >1/2$,
we bound
\begin{align*}
	\MoveEqLeft \sup_\xi |T_1[u_1,u_2,u_3,v] (\xi)| \\
	& \lesssim \sup_\xi \int_{\Gamma_\xi} \frac{|\xi|\jap{\xi}^\mu }{\prod_{j=1}^3\jap{\xi_j}^\mu } \left(\int_{\m R^3}\frac{\jap{\tau-\xi^3}^{2b'}}{\jap{\tau_3-\xi^3_3}^{2b}} \|u_1(\tau_1)\|_{L^\infty_\xi}\|u_2(\tau_2)\|_{L^\infty_\xi} d\tau d\tau_1 d\tau_2\right)d\xi_1 d\xi_2\\
	& \lesssim \sup_\xi \int_{\Gamma_\xi} \frac{|\xi|\jap{\xi}^\mu }{\prod_{j=1}^3\jap{\xi_j}^\mu } \left[ \int_{\m R^2} \|u_1(\tau_1)\|_{L^\infty_\xi}\|u_2(\tau_2)\|_{L^\infty_\xi} \left( \int_{\m R}\frac{\jap{\tau}^{2b'} d\tau}{\jap{\tau - \tau_1 - \tau_2 - \Phi}^{2b}} \right) d\tau_1 d\tau_2\right] d\xi_1 d\xi_2\\
	& \lesssim \int_{\m R^2} \left( \sup_\xi \int_{\Gamma_\xi} \frac{|\xi|\jap{\xi}^\mu }{\left(\prod_{j=1}^3\jap{\xi_j}^\mu \right)\jap{\tau_1+\tau_2+ \Phi}^{-2b'}} d\xi_1d\xi_2\right) \|u_1(\tau_1)\|_{L^\infty_\xi}\|u_2(\tau_2)\|_{L^\infty_\xi} d\tau_1 d\tau_2 \\
	& \lesssim \sup_{\xi,\alpha} \int_{\Gamma_\xi} \frac{|\xi|\jap{\xi}^\mu }{\left(\prod_{j=1}^3\jap{\xi_j}^\mu \right)\jap{\Phi-\alpha}^{-2b'}} d\xi_1d\xi_2\\
	&\lesssim \sup_{\xi,\alpha} \sum_{M \ge 1, \text{ dyadic}} \frac{1}{M^{-2b'}}\int_{\Gamma_\xi} \frac{|\xi|\jap{\xi}^\mu }{\jap{\xi_1}^\mu \jap{\xi_2}^\mu }\mathbbm{1}_{|\Phi-\alpha|<M} d\xi_1d\xi_2 \\
	&\lesssim \sup_{\xi,\alpha} \sum_{M \ge 1,  \text{ dyadic}} \frac{1}{M^{-2b'}}M^{1-\mu /3} \lesssim 1,
\end{align*}
as soon as $1-\mu/3 + 2b'<0$, which is possible as soon as $b < 1/2+\mu/6$ (we used in the frequency-restricted estimate \eqref{eq:TalphaM_multiplier} in the penultimate inequality). Similarly, one can show the continuity of $T_2:(L^\infty_\tau L^\infty_\xi)^2\times L^1_\tau L^\infty_\xi \times L^1_\tau \to L^\infty_\xi$, where
\[
	T_2[u_1,u_2,u_3,v]=\int_{\Gamma_\xi} \int_{\Gamma_\tau} \frac{|\xi|\jap{\xi}^\mu }{\jap{\tau_1-\xi_1^3}^{2b}\jap{\tau_2-\xi_2^3}^{2b}\prod_{j=1}^3\jap{\xi_j}^\mu }\left(\prod_{j=1}^3u_j(\tau_j,\xi_j)\right)v(\tau)d\xi_1d\xi_2d\tau_1d\tau_2d\tau.
\]
The claimed boundedness of $T$ follows by Stein interpolation between $T_1$ and $T_2$.
\end{proof}
		
\begin{proof}[Proof of Proposition \ref{prop:bourg}]
The proof is standard. Take $b$ as in Lemma \ref{lem:bourg_fre} and $C$ as in Lemma \ref{lem:Xsb}. Consider the closed ball
\[
	B=\{ u\in X^{\mu ,b}_\infty : \|u\|_{X^{\mu ,b}_\infty} \le 2C\|\hat{u}_0\|_{L^\infty_\mu }\},
\]
	of the complete space $X^{\mu,b}_\infty$ and the map
\[
	\Theta[u](t)=\psi(t)e^{-t\partial_x^3}u_0 + \psi_\delta(t)\int_0^t e^{-(t-s)\partial_x^3}(u^3)_x ds.
\]
			Applying the estimates of Lemmas \ref{lem:Xsb} and \ref{lem:bourg_fre},
			\[
			\|\Theta[u]\|_{X^{\mu ,b}_\infty}\le C\|\hat{u}_0\|_{L^\infty_\mu } + C\delta^{1+b'-b}\|u\|_{X^{\mu ,b}_\infty}^3  
			\]
			and
			\[
			\|\Theta[u]-\Theta[v]\|_{X^{\mu ,b}_\infty}\le  C\delta^{1+b'-b}(\|u\|_{X^{\mu ,b}_\infty}^2 + \|v\|_{X^{\mu ,b}_\infty}^2)\|u-v\|_{X^{\mu ,b}_\infty}.
			\]
			Therefore, for $\delta$ sufficiently small, $\Theta$ is a contraction over $B$, yielding the result.
		\end{proof}

		\nocite{HM80}
		\nocite{HN01}
		\nocite{HN99}
		\nocite{CV08}
		\nocite{DZ95}
		\nocite{GPR16}
		\nocite{HaGr16}
		\nocite{PV07}
		\nocite{FIKN06}
		\nocite{FA83}
		
		\bibliography{biblio_mkdv}
		\bibliographystyle{plain}
		
		\bigskip
		\bigskip
		
		\normalsize
		
		\begin{center}
			{\scshape Simão Correia}\\
			{\footnotesize
				
				CAMGSD, Department of Mathematics,\\
				Instituto Superior T\'ecnico, Universidade de Lisboa\\
				Av. Rovisco Pais, 1049-001 Lisboa, Portugal\\
				simao.f.correia@tecnico.ulisboa.pt
			}
			
			\bigskip
			
			{\scshape Raphaël Côte}\\
			{\footnotesize
				Université de Strasbourg\\
				CNRS, IRMA UMR 7501\\
				F-67000 Strasbourg, France\\
				cote@math.unistra.fr
			}

		\end{center}

	\end{document}